\documentclass[reqno,a4paper,11pt]{amsart}

\usepackage{mystyle}
\usepackage{etoolbox}

\usepackage[final]{microtype}%
\usepackage{amsthm,mathtools}%
\usepackage{xcolor}%
\usepackage{float}%
\usepackage[export]{adjustbox}%
\usepackage{tikz, tikz-cd, mathtools, amssymb, stmaryrd}\usetikzlibrary{calc}
\usetikzlibrary{decorations.pathmorphing}

\tikzset{curve/.style={settings={#1},to path={(\tikztostart)
    .. controls ($(\tikztostart)!\pv{pos}!(\tikztotarget)!\pv{height}!270:(\tikztotarget)$)
    and ($(\tikztostart)!1-\pv{pos}!(\tikztotarget)!\pv{height}!270:(\tikztotarget)$)
    .. (\tikztotarget)\tikztonodes}},
    settings/.code={\tikzset{quiver/.cd,#1}
        \def\pv##1{\pgfkeysvalueof{/tikz/quiver/##1}}},
    quiver/.cd,pos/.initial=0.35,height/.initial=0}

\tikzset{tail reversed/.code={\pgfsetarrowsstart{tikzcd to}}}
\tikzset{2tail/.code={\pgfsetarrowsstart{Implies[reversed]}}}
\tikzset{2tail reversed/.code={\pgfsetarrowsstart{Implies}}}
\tikzset{no body/.style={/tikz/dash pattern=on 0 off 1mm}}
    \usepackage{mathpartir}

    \usepackage{quiver}
\newtheorem{theorem}{Theorem}[section]%
\newtheorem{lemma}[theorem]{Lemma}%
\newtheorem{corollary}[theorem]{Corollary}%
\theoremstyle{definition}%
\newtheorem{definition}[theorem]{Definition}%
\newtheorem{proposition}[theorem]{Proposition}%
\newtheorem{example}[theorem]{Example}%
\newtheorem{remark}[theorem]{Remark}%
\usepackage[colorlinks=true,linkcolor={blue!30!black}]{hyperref}%
\usepackage{newpxmath,newpxtext}%
\usepackage{cleveref}%
\usepackage[mode=buildmissing]{standalone}%
\usepackage{string-diagrams}
\newtheorem*{theorem*}{Theorem 3.8}

\newcommand{\id}{\mathrm{id}}
\newcommand{\op}{\mathrm{op}}
\newcommand{\dom}{\mathrm{dom}}
\newcommand{\cod}{\mathrm{cod}}
\newcommand{\ps}{\mathrm{ps}}
\newcommand{\ExistStr}{\mathrm{ExistStr}}
\newcommand{\TwoCat}{\mathrm{2CAT}}
\newcommand{\Dbl}{\mathrm{Dbl}}
\newcommand{\Ar}{\mathrm{Arrow}}
\newcommand{\lax}{\mathrm{lax}}
\newcommand{\SM}{\mathrm{SM}}
\newcommand{\colax}{\mathrm{colax}}
\newcommand{\Span}[1]{\mathbb{S}{pan}(#1)}
\newcommand{\Qt}[1]{\mathbb{Q}{t}(#1)}

    \title{Double-functorial representation of regular hyperdoctrines}
    
  \author[J. Siqueira]{Jos\'e Siqueira}
\address{Topos Institute\\ Oxford, England}
\email{jose@topos.institute}

    \date{\today
}

\begin{document}

\maketitle

    \tableofcontents
\begin{abstract}
    It is well-known that pseudofunctors from bicategories of spans are equivalent to Beck-Chevalley bifibrations, and therefore capture the relationships underlying the adjunctions 
suitable as semantics for existential quantification. This was further expanded upon by Dawson, Par\'e and Pronk in the context of double categories. 
By viewing hyperdoctrines from a double-categorical lens, this paper shows that we can also characterise the Frob\"enius property: (generalised) regular
hyperdoctrines correspond to certain lax symmetric monoidal pseudo double functors from spans to quintets whose monoidal laxators provide companion commuter cells (in the sense
of Paré). This facilitates the study of the compositionality of regular hyperdoctrines and hints at a new notion of \emph{regular double hyperdoctrine}. As an application, we discuss how we can recover
a form of graphical regular logic suitable for modelling specifications of systems (e.g., port-plugging systems) that compose operadically. 
\end{abstract}
    \label{section:abstract}

\subsection*{Acknowledgments} Funded by the Advanced Research + Invention Agency (ARIA) through
project code MSAI-PR01-P14. The author thanks Alex Mallard for the assistance with type-setting string diagrams, as well as David Jaz Myers, Jason Brown, Nathanael Arkor, and the anonymous referees for the valuable discussions.

\section{Introduction}
\label{section:introduction}
\par{}One of the most fruitful modern insights into logic was Lawvere's observation that quantifiers ought to be thought of as adjoint functors
 \cite{lawvere-1969-adjointness}, which is at the heart 
of the notion of \emph{hyperdoctrine}: an indexed poset (whose posets we think of as forming \emph{predicates} of a certain type)
 that satisfies properties enabling the adequate 
interpretation of a target fragment of logic (chiefly the Beck-Chevalley condition). Fundamentally, this means that the data of a logical doctrine 
manifests itself in two directions: that of \emph{context operations}, and that of the \emph{action on predicates}.\par{}Double Category Theory provides the tools to capture these two actions and 
 their relationship concomitantly: in a double category of spans, re-indexing 
 and acting on predicates can be packaged as tight and loose arrows respectively, 
 and in there tight arrows always have suitably constrained conjoints.  By mapping these conjoints into 
 a double category of quintets, we obtain adjunctions internal to a 2-category 
 (classically, of posets) which satisfy the Beck-Chevalley condition, and thus serve
  as semantics for substitution and existential quantification.
 \par{}What was described thus far is well-known in the literature. The way spans come into the picture is very natural if one is interested in logical semantics:  start with a category and add adjoints to every morphism, 
 but make them constrained by the Beck-Chevalley condition so that you can define quantification. If the category
 had pullbacks to begin with, you end up with a bicategory of spans (c.f.  \cite{dawson-2003-adjoining}).  The first record of the universal property of spans and the relationship to 
  the Beck-Chevalley condition was given by Hermida \cite{hermida-2000-representable}  and referred to as folklore.  Dawson, Paré  and Pronk further explored spans as a bicategory and the adjointness structure
   \cite{dawson-2004-universal}: sinister morphisms of bicategories \(\mathcal {A} \to  \mathcal {B}\)
  (i.e., whose morphism images are left adjoints) satisfying the Beck-Chevalley condition are equivalent to strong morphisms \(\mathcal {S}pan(\mathcal {A})\to  \mathcal {B}\). 
  The same authors extended this to a double categorical picture in \cite{dawson-2010-span}, where the notions of \emph{gregarious double category} (also called \emph{fibrant double categories} 
  or \emph{framed bicategories} in \cite{shulman-2008-framed-bicats} ) and Beck-Chevalley double category were introduced.  Notably, the double-categorical span construction
  provides a canonical way to obtain a Beck-Chevalley double category out of a category with pullbacks, and has better universal properties than its bicategorical counterpart. 
  \par{}The main contribution of this paper is to extend these ideas to the study of \emph{regular logic} and analogous systems: by thinking in terms of double
functors and factoring in a monoidal structure, one can also capture the Frob\"enius condition for regular hyperdoctrines. In its simplest form, our main result establishes 
a correspondence between regular hyperdoctrines \(P\) over a cartesian category \(\mathsf {C}\) of contexts and certain symmetric 
monoidal double functors \(Q \colon  \Span{\mathsf{C}}^{\op} \to \Qt{\mathsf{Pos}}\) --- those for which
 the cell components of the monoidal laxator are \emph{companion commuters} in the sense of Paré (see \cite{pare-2024-retrocells} ).
\par{}There is no reason to restrict ourselves to \(\land \)-semilattices and classical regular logic, though. More generally,
 we will allow the predicates of \(P\) 
 to form symmetric pseudomonoids in a symmetric monoidal 2-category \(\mathcal {K}\), while still being subject to restricted Beck-Chevalley and Frob\"enius conditions. 
 We refer to such analogues of regular hyperdoctrines as \emph{regular $\mathcal{K}$-hyperdoctrines}.$^{\footnotesize 1}$ \footnote{$\footnotesize 1$ The idea for this concept is implicit in works like \cite{Shulman2011}, but not formalised.}
 Moreover, the indexing category \(\mathsf {C}\) might only have pullbacks for certain kinds of morphisms,
 which can be formalised by the notion of \textit{adequate triple}  (adapted from \cite{barwick-spectral-2017, haugseng-2020-twovariable});  the Beck-Chevalley and Frob\"enius conditions are thus only assumed to hold for certain morphisms. Making these adjustments, we can state the main theorem of this paper (Theorem~\ref{MainTheorem}) as follows: \\

\textbf{Theorem}: \emph{Let $\mathfrak{C}$ be a cartesian adequate triple and $\mathcal{K}$ be a symmetric monoidal 2-category. Every $\mathfrak{C}$-regular 
 $\mathcal{K}$-hyperdoctrine $P$ can be extended to a lax symmetric monoidal pseudo double functor
  $P^{\bullet} \colon \Span{\mathfrak{C}}^{\op} \to \mathcal{K}$ whose monoidal laxators are companion commuter transformations. Conversely,
  if the class $\mathfrak{C}^l$ includes all diagonals, then the tight component of every such double functor which is strict is a $\mathfrak{C}$-regular $\mathcal{K}$-hyperdoctrine. }
\\

Viewing regular hyperdoctrines as double functors facilitates the study of their compositionality, with finer control over the contexts and the allowed semantics. Moreover, the result suggests that a broader notion of \emph{regular double hyperdoctrine} could be defined, where even greater fine-tuning of the semantics can be achieved by changing the double category where the semantics takes place, and replacing `quantification and substitution form an adjunction' with `quantification and substitution form a conjunction'. Conceptually, this means existential quantification and substitution need not be in the same 2-category (e.g., they can form a doctrinal adjunction). We will not pursue this perspective in this paper, but it will be the subject of future work. Instead, we illustrate how these ideas on compositionality can be used in the double categorical treatment of systems theories of \cite{libkind-2025-towards} in Section~\ref{section:discussion}.
  
Section~\ref{section:preliminaries} introduces the preliminary notions required on both the logical and double
categorical sides, as well as some representative examples. We recall the definition of \emph{adequate triple} needed for a general span construction, then
introduce a broad notion of \emph{regular hyperdoctrine} and discuss how it relates to the classical one. The proof of the main result is divided in two sections: in Section~\ref{section:DoctrinesAsDoublePseudofunctors}
we show how a (generalised) regular hyperdoctrine can be extended to a suitable type of pseudo double functor, and in Section~\ref{section:DoublePseudofunctorsAsDoctrines} we argue that one extra mild
condition on the adequate triples (which is trivially satisfied for the conventional spans) yields a converse. We finalise in Section~\ref{section:discussion} with a discussion of the significance of the result, the application to systems theories, and further directions.

\subsection{Notation}

Throughout this paper, categories will be drawn in \emph{sans serif font}  $\mathsf{C}$, 2-categories will be denoted in \emph{calligraphic style} $\mathcal{K}$, and double categories will be denoted in \emph{blackboard bold script}, e.g., $\mathbb{D}$, $\mathbb{Q}{t}(\mathcal{K})$ or $\mathbb{S}{pan}(\mathsf{C})$. We also write $\mathsf{SM}^{\text{lax}}(\mathcal{K})$ (respectively, $\mathsf{SM}(\mathcal{K})$) for the 2-category of symmetric pseudomonoids, lax (respectively, strong) morphisms of pseudomonoids, and monoidal 2-cells between them in $\mathcal{K}$. We take $\TwoCat$ to be the 2-category of locally small 2-categories, 2-functors, and 2-natural transformations, and $\Dbl$ to be the 2-category of double categories, pseudo double functors, and tight transformations. If $\mathcal{W}$ is a 2-category, we write $\Ar^{\colax}(\mathcal{W})$ for the 2-category of arrows in $\mathcal{W}$ and colax morphisms between them (i.e., the 2-cells in $\mathcal{W}$ that occur as part of $1$-morphisms of arrows go \emph{down} the page). \\

The compositor and unitor for a pseudofunctor $P$ will be denoted by $\gamma^P$ and $\iota^P$, respectively. A projection map $A \times B \to A$ will be denoted $\pi^{A \times B}_A$ (and similarly for the second projection). We will reserve $\Delta$ for the diagonal transformation. The naturality cell of a pseudonatural transformation (in particular, the naturality square of a natural transformation) $\alpha$ with respect to a morphism $f$ will be denoted by $\alpha_f$. \\ 


We will employ 2-categorical string diagrams, which depict  objects as regions (left implicit), morphisms as wires (with the left-most being applied first), and 2-cells as squares (with the bottom ones being applied first). Wires bent as cups and caps denote units and counits of adjunctions, respectively, or their images under given pseudofunctors; it will be clear from context. As the adjunction in question can be easily deduced from the labels on the wires, we will label those and leave the adjunctions themselves implicit (e.g., the unit of $\exists^P f \vdash Pf$ is a cup with an $\exists^P f$-labelled wire on the left and a $Pf$-labelled wire on the right). In diagrams, a 2-cell witnessing an isomorphism $f \cong g$ will be denoted $\cong_{f}^{g}$, and similarly for equalities.

\section{Preliminaries}
\label{section:preliminaries}
This paper reinforces the idea that double categories provide a better environment to speak of logic, where we can neatly organise the action
on contexts and the action on predicates on separate dimensions, which are nonetheless linked by the double cells and the axioms of a double category. This section is dedicated to introducing the objects of interest: generalised hyperdoctrines and the double categories they embed in. 

We start by recalling the notion of \emph{adequate triple}, which allows us to consider hyperdoctrines for which meaningful quantification is only allowed for a select class of morphisms. We then introduce \emph{existential symmetric monoidal structures} and (generalised) \emph{regular hyperdoctrines}, which allow for different monoidal structures on predicates instead of just the usual cartesian $\land$. This differs slightly from the usual notion of hyperdoctrine (even when ignoring adequate triples and using the cartesian monoidal structure on the the predicates), but we explain how the notions are equivalent in that case. Finally, we recall the essential notions of double category theory used in the paper, particularly the concepts of \emph{companion} and \emph{conjoint} of a tight morphism, and the double categories of spans and quintets.\\

In its simplest incarnation, the classical definition of regular hyperdoctrine is as a contravariant functor $P \colon \mathsf{C}^{\op} \to \land\mathrm{SLat}$ from some category $\mathsf{C}$ with finite limits to the category of $\land$-semilattices, with the property that $Pf$ has a left adjoint $\exists f$ for each morphism $f \in \mathsf{C}$, which are subject to the Beck-Chevalley and Frob\"enius laws.$^{\footnotesize 2}$ \footnote{$^{\footnotesize 2}$ As far as I can tell, the earliest form of this idea appears in \cite{lawvere2}.} We see $\mathsf{C}$ as a category of \emph{contexts}, each fibre $PA$ (for lack of a better word) as the $\land$-semilattice of \emph{predicates over $A$}, and the maps $Pf$ and $\exists f$ as \emph{term substitution} and \emph{existential quantification}, respectively. The Beck-Chevalley condition then asserts that given a pullback of contexts, substitution and existential quantification commute in a suitable way --- this replaces having to deal with free variables. Finally, the Frob\"enius law establishes that existential quantification is suitably compatible with conjunctions.

This notion, as it is, can be too rigid. The first issue is that there are situations (appearing
chiefly in geometry, but also in the study of \emph{lenses}) where $\mathsf{C}$ only admits certain pullbacks, and only some $Pf$ have left adjoints. From the point of view of logic, this would mean that only quantification along certain morphisms exists and behaves well. In the setting of nonstandard analysis, for example, one might envision having substitution of arbitrary terms, but only allowing for quantification along standard maps.

Now, it is well-known that the Beck-Chevalley condition is faithfully captured by the underyling algebra of \emph{spans} \cite{Evangelia, dawson-2010-span, hermida-2000-representable, Hoshino2025}, which compose by taking pullbacks. The relaxation of which kind of pullbacks is allowed was done  by C. Barwick in the \(\infty \)-categorical
literature by introducing the notion of \emph{adequate triple}, which singles out what is needed to compose spans \cite{barwick-spectral-2017}. Here we employ the following 1-categorical incarnation of the concept:
\begin{definition}[adapted from \cite{haugseng-2020-twovariable}]
\par{}An \emph{adequate triple} $\mathfrak{C} = (\mathsf {C}, \mathfrak{C}^l, \mathfrak{C}^r)$ consists of a category \(\mathsf {C}\), together with classes \(\mathfrak{C}^l\) and \(\mathfrak{C}^r\) of morphisms
of \(\mathsf {C}\) such that:\begin{itemize}\item{}\(\mathfrak{C}^l\) and \(\mathfrak{C}^r\) contain all identities and are closed under composition;

\item{}Every cospan \(X \xrightarrow {x} Z \xleftarrow {y} Y\) with \(x \in  \mathfrak{C}^l\) and \(y \in  \mathfrak{C}^r\) admits a pullback;

\item{}The pullback of a morphism in \(\mathfrak{C}^l\) along a morphism in \(\mathfrak{C}^r\) is a morphism in \(\mathfrak{C}^l\) (and vice-versa). We refer to such a pullback as
a $\mathfrak{C}$-\emph{pullback}.\end{itemize}

We refer to $\mathfrak{C}^l$ as the \emph{left class} and to $\mathfrak{C}^r$ as the \emph{right class} of the adequate triple $\mathfrak{C}$.
\end{definition}

Adequate triples can be organised into a 2-category $\mathcal{A}{dq}$ by taking as morphisms the 2-functors that preserve $\mathfrak{C}^l$, $\mathfrak{C}^r$ and $\mathfrak{C}$-pullbacks, and the 2-cells as 2-natural transformations. It is easy to see that it is cartesian, with products computed component-wise. Moreover, the construction of spans provides a cartesian 2-functor $\Span{-} \colon \mathcal{A}{dq} \to \Dbl$, where $\Dbl$ is the 2-category of (weak) double categories, pseudo double functors, and tight transformations \cite{libkind-2025-towards}. We will discuss double categories of spans in more detail towards the end of the section.\\ 

\par{} In logic, we like to be able to put contexts in parallel and form lists. This means we require a monoidal structure in $\mathsf{C}$ that is compatible with the adequate triple structure (usually cartesian). For this reason, we work with \emph{symmetric monoidal adequate triples}, which are the symmetric pseudomonoids in $\mathcal{A}{dq}$. Thus, the underlying category $\mathsf{C}$ of one such triple is a symmetric monoidal category, the classes $\mathfrak{C}^l$ and $\mathfrak{C}^r$ are closed under the tensor product of $\mathsf{C}$, and the tensor product preserves $\mathfrak{C}$-pullbacks. An adequate triple is \emph{cartesian} if it has a cartesian monoidal structure.

\begin{example}
    For any category $\mathsf{C}$, there is a \emph{trivial adequate triple} with underlying category $\mathsf{C}$, where $\mathfrak{C}^l=\mathfrak{C}^r=\mathrm{mor}(\mathsf{C})$. The associated double category of spans is the usual double category $\Span{\mathsf{C}}$ whose loose morphisms are spans in $\mathsf{C}$ and whose squares are morphisms of spans.
\end{example}

\begin{example}
    The category $\mathrm{Met}$ of metric spaces and continuous maps admits a cartesian adequate triple structure, where both the left and right class are the Lipschitz-continuous maps.
\end{example}

\begin{example}[{}]
\par{}For any Grothendieck fibration \(\pi  \colon  \mathsf{E} \to  \mathsf{B}\), there is an associated adequate triple $\mathfrak{E}$ with underlying category $\mathsf{E}$, where \(\mathfrak{E}^l\) is the class of
vertical maps and \(\mathfrak{E}^r\) is the class of cartesian maps. The double category  \(\mathbb {S}\mathsf {pan}(\mathfrak{E})\) is then the double category
\(\mathbb {L}\mathsf {ens}(E)\), whose loose maps are \emph{Spivak lenses} \cite{spivak-2019-generalized}  (see \cite{libkind-2025-towards} for details).\end{example}


Given an adequate triple \(\mathfrak{C}\), a natural extension of the classical concept of existential hyperdoctrine is to restrict the Beck-Chevalley condition to \(\mathfrak{C}\)-pullbacks only, and to allow for predicates to form more general structures than orderings (e.g., be objects in some 2-category). This leads to the following notion:

\begin{definition}[$\mathfrak{C}$-existential $\mathcal{K}$-structure]
\label{def:ExistentialStructure}
    Let $\mathfrak{C}$ be an adequate triple and $\mathcal{K}$ be a locally small 2-category. A $\mathfrak{C}$-existential $\mathcal{K}$-structure is a 2-functor $P \colon \mathsf{C}^{\op} \to \mathcal{K}$ such that:
    \begin{itemize}
      \item For each morphism \(f \colon  A \to  B\) in \(\mathfrak{C}^r\), the morphism \(Pf \colon  PB \to  PA\) has a  left adjoint \(\exists^P f \colon  PA \to  PB\) (internal to $\mathcal{K}$). We denote the unit
of this adjunction by \(\eta^{Pf}\) and the counit by \(\epsilon^{Pf}\);

\item These adjoints satisfy the Beck-Chevalley condition with respect to the adequate triple \(\mathfrak{C}\): for any \(\mathfrak{C}\)-pullback \begin{tikzcd}
	A & I \\
	B & J
	\arrow[""{name=0, anchor=center, inner sep=0}, "h", from=1-1, to=1-2]
	\arrow["k"', from=1-1, to=2-1]
	\arrow["\lrcorner"{anchor=center, pos=0.125}, draw=none, from=1-1, to=2-2]
	\arrow["g", from=1-2, to=2-2]
	\arrow[""{name=1, anchor=center, inner sep=0}, "f"', from=2-1, to=2-2]
	\arrow["\alpha"{description}, color={rgb,255:red,92;green,92;blue,214}, draw=none, from=0, to=1]
\end{tikzcd} where $g \in \mathfrak{C}^l$ and $f \in \mathfrak{C}^r$, the 2-cell 
\begin{center}\begin{tikzcd}
	PB && PJ && PI && PI \\
	\\
	PB && PB && PA && PI
	\arrow[""{name=0, anchor=center, inner sep=0}, "{\exists^Pf}", from=1-1, to=1-3]
	\arrow[equals, from=1-1, to=3-1]
	\arrow[""{name=1, anchor=center, inner sep=0}, "Pg", from=1-3, to=1-5]
	\arrow["Pf"', from=1-3, to=3-3]
	\arrow[""{name=2, anchor=center, inner sep=0}, equals, from=1-5, to=1-7]
	\arrow["Ph", from=1-5, to=3-5]
	\arrow[equals, from=1-7, to=3-7]
	\arrow[""{name=3, anchor=center, inner sep=0}, equals, from=3-1, to=3-3]
	\arrow[""{name=4, anchor=center, inner sep=0}, "Pk"', from=3-3, to=3-5]
	\arrow[""{name=5, anchor=center, inner sep=0}, "{\exists^Ph}"', from=3-5, to=3-7]
	\arrow["{\eta^{Pf}}"', color={rgb,255:red,153;green,92;blue,214}, between={0.2}{0.8}, Rightarrow, from=3, to=0]
	\arrow["P(\alpha)"', color={rgb,255:red,92;green,92;blue,214}, between={0.2}{0.8}, Rightarrow, from=4, to=1]
	\arrow["{\epsilon^{Ph}}"', color={rgb,255:red,153;green,92;blue,214}, between={0.2}{0.8}, Rightarrow, from=5, to=2]
\end{tikzcd}\end{center}
 has an inverse \(\mathcal {B}_{\alpha}\) in $\mathcal{K}$.
    \end{itemize}
    \end{definition}

The 2-category of existential structures and weak morphisms of existential structures $\ExistStr_{\mathrm{w}}$ is obtained by first forming a pullback \begin{equation*}
\begin{tikzcd}
	{\ExistStr'} && {\Ar^{\colax}(\TwoCat)} \\
	\\
	{\mathcal{A}{dq}} && {\TwoCat}
	\arrow[""{name=0, anchor=center, inner sep=0}, from=1-1, to=1-3]
	\arrow[from=1-1, to=3-1]
	\arrow["\lrcorner"{anchor=center, pos=0.125}, draw=none, from=1-1, to=3-3]
	\arrow["\dom", from=1-3, to=3-3]
	\arrow[""{name=1, anchor=center, inner sep=0}, from=3-1, to=3-3]
	\arrow[""{description}, shift left=4, color={rgb,255:red,92;green,92;blue,214}, draw=none, from=0, to=1]
\end{tikzcd},
\end{equation*}
where $\mathcal{A}{dq} \to \TwoCat$ is pointwise the opposite of the obvious forgetful map, then taking its full sub-2-category on those objects that are existential structures. More explicitly,:

\begin{itemize}
    \item objects are pairs $(\mathfrak{C}, P)$, where $\mathfrak{C}$ is an adequate triple and $P$ is a $\mathfrak{C}$-existential structure;

    \item morphisms $(\mathfrak{C}_1, P_1) \to (\mathfrak{C}_2, P_2)$ are triples $(F, \rho, L)$ providing a diagram 
    \begin{equation*}
\begin{tikzcd}
	{\mathsf{C}^{\op}_1} &&& {\mathcal{K}_1} \\
	\\
	{\mathsf{C}^{\op}_2} &&& {\mathcal{K}_2}
	\arrow[""{name=0, anchor=center, inner sep=0}, "{P_1}", from=1-1, to=1-4]
	\arrow["{F^{\op}}"', color={rgb,255:red,92;green,92;blue,214}, from=1-1, to=3-1]
	\arrow["L", color={rgb,255:red,92;green,92;blue,214}, from=1-4, to=3-4]
	\arrow[""{name=1, anchor=center, inner sep=0}, "{P_2}"', from=3-1, to=3-4]
	\arrow["\rho", color={rgb,255:red,214;green,92;blue,92}, between={0.2}{0.8}, Rightarrow, from=0, to=1]
\end{tikzcd}
    \end{equation*}
in $\TwoCat$ for which $F$ is a 1-morphism of adequate triples;


\item $2$-morphisms $\alpha \colon (F_1, \rho_1, L_1) \to (F_2, \rho_2, L_2)$ are pairs $(F_2 \xrightarrow{\alpha} F_1, L_1 \xrightarrow{\beta} L_2)$ of 2-natural transformations such that the cube
\begin{equation*}
\begin{tikzcd}
	& {\mathsf{C}^{\op}_1} &&& {\mathcal{K}_1} \\
	{\mathsf{C}^{\op}_2} &&& {\mathcal{K}_2} \\
	\\
	& {\mathsf{C}^{\op}_1} &&& {\mathcal{K}_1} \\
	{\mathsf{C}^{\op}_2} &&& {\mathcal{K}_2}
	\arrow[""{name=0, anchor=center, inner sep=0}, "{P_1}", from=1-2, to=1-5]
	\arrow[""{name=1, anchor=center, inner sep=0}, "{F_1^{\op}}"', color={rgb,255:red,92;green,92;blue,214}, from=1-2, to=2-1]
	\arrow[equals, from=1-2, to=4-2]
	\arrow[""{name=2, anchor=center, inner sep=0}, "L_1", color={rgb,255:red,92;green,92;blue,214}, from=1-5, to=2-4]
	\arrow[equals, from=1-5, to=4-5]
	\arrow[""{name=3, anchor=center, inner sep=0}, "{P_2}"', from=2-1, to=2-4]
	\arrow[equals, from=2-1, to=5-1]
	\arrow[equals, from=2-4, to=5-4]
	\arrow[""{name=4, anchor=center, inner sep=0}, "{P_1}", from=4-2, to=4-5]
	\arrow[""{name=5, anchor=center, inner sep=0}, "{F_2^{\op}}"', color={rgb,255:red,92;green,92;blue,214}, from=4-2, to=5-1]
	\arrow[""{name=6, anchor=center, inner sep=0}, "{L_2}", color={rgb,255:red,92;green,92;blue,214}, from=4-5, to=5-4]
	\arrow[""{name=7, anchor=center, inner sep=0}, "{P_2}"', from=5-1, to=5-4]
	\arrow["\rho_1", color={rgb,255:red,214;green,92;blue,92}, between={0.2}{0.8}, Rightarrow, from=2, to=1]
	\arrow["{{\alpha}^{\op}}"', color={rgb,255:red,153;green,92;blue,214}, between={0.2}{0.8}, Rightarrow, from=1, to=5]
	\arrow["{\beta}", color={rgb,255:red,153;green,92;blue,214}, between={0.2}{0.8}, Rightarrow, from=2, to=6]
	\arrow["{\rho_2}", color={rgb,255:red,214;green,92;blue,92}, between={0.2}{0.8}, Rightarrow, from=6, to=5]
\end{tikzcd}
\end{equation*}
commutes.     
\end{itemize}

As adequate triples and $\mathfrak{C}$-existential structures are closed under finite products, this is a cartesian 2-category.\\

The reason morphisms in $\ExistStr_{\mathrm{w}}$ are called \emph{weak} is because the 2-cells $\rho$ of $\mathcal{K}$ that appear in 1-morphisms of $\ExistStr_{\mathrm{w}}$ only preserve existential quantification weakly, via their mates --- in the case of classical poset-based hyperdoctrines, they only give an inequality comparing existential quantifications. To say that $\rho$ preserves existential quantification (recovering the notion of morphism of existential doctrines in \cite{trotta}) is to say that $\mathrm{mate}(\rho_f)$ is invertible in $\mathcal{K}$ for each $f \in \mathfrak{C}^r$. We write $\ExistStr$ for the sub-2-category of $\ExistStr_{\mathrm{w}}$ that has the same objects and 2-morphisms, but whose 1-morphisms are strong in this sense.\\

For classical hyperdoctrines, there is still the stipulation that each fibre has a monoidal structure, typically the cartesian one yielding conjunctions, but not always (e.g., in linear logic). We can account for this by considering symmetric pseudomonoids in $\ExistStr_{\mathrm{w}}$.

\begin{definition}
\label{def:existential pseudomonoid}
    An existential (symmetric) pseudomonoid is a (symmetric) pseudomonoid in the cartesian 2-category $\ExistStr_{\mathrm{w}}$. If its underlying 2-functor has (the underlying category of) $\mathfrak{C}$ as its domain and $\mathcal{K}$ as its codomain, we say it is a $\mathfrak{C}$-existential (symmetric) monoidal $\mathcal{K}$-structure.
\end{definition}

It is easy to see that specifying an existential (symmetric) pseudomonoid amounts to selecting a (symmetric) monoidal adequate triple $\mathfrak{C}$, a locally small (symmetric) monoidal $2$-category $\mathcal{K}$, and a $\mathfrak{C}$-existential $\mathcal{K}$-structure $P$ whose underlying 2-functor is lax (symmetric) monoidal. For example, the pseudomonoid multiplication $(\mathfrak{C}, P) \times (\mathfrak{C}, P) \cong (\mathfrak{C} \times \mathfrak{C}, P \times P) \to (\mathfrak{C}, P)$ provides symmetric monoidal structures for $\mathsf{C}$ and $\mathcal{K}$, as well as a lax monoidal structure for the 2-functor $P \colon \mathsf{C}^{\op} \to \mathcal{K}$ with respect to those. \\

We will later see that if (the underlying object) of an existential symmetric pseudomonoid is a regular hyperdoctrine as below, then its laxator is a strong morphism of existential doctrines, so we actually have a symmetric pseudomonoid in $\ExistStr$ (c.f. Lemma~\ref{lemma: regular hyperdoctrine extends to double functor}).\\

Now, Definition~\ref{def:existential pseudomonoid} does not quite match what the reader is likely used to calling an \emph{existential hyperdoctrine}, even in the case where $\mathfrak{C}$ is a trivial cartesian adequate triple and $\mathcal{K}$ is the 2-category of posets. This is since a $\mathfrak{C}$-regular $\mathcal{K}$-hyperdoctrine $P$ is a lax symmetric monoidal 2-functor $P \colon \mathsf{C}^{\op} \to \mathcal{K}$, whereas in a classical hyperdoctrine we expect the \emph{fibres} to have a monoidal structure (e.g., as a lattice), not $P$ itself. Moreover, in many natural examples $P$ is only pseudofunctorial. However, both issues are circumventable. 

Regarding connectives, there is a correspondence between ``external'' monoidal structures $(P, \mu^P, I^P)$ on $P$ and ``internal" monoidal structures $(PA, \otimes_{PA}, I^{PA})$ on the fibres $PA$ for each object $A \in \mathsf{C}$ whenever $\mathsf{C}$ is cartesian and $\mathcal{K}$ is symmetric monoidal --- in the case $\mathcal{K}=\mathcal{C}\mathsf{at}$, this is (modulo the symmetry, which is discussed only in the context of the correspondence between indexed categories and fibrations) as in Theorem 4.2 in \cite{moeller-vasilakopoulou2020}. The same phenomenon was discussed in \cite{shulman-2008-framed-bicats} in the context of fibrations (c.f. Theorem 12.7) so we only sketch the argument here. The relevant definitions can be found in \cite{moeller-vasilakopoulou2020}.

\begin{proposition}
\label{Proposition:Internal vs External monoidal structures}
    Let $\mathsf{C}$ be a cartesian category and $\mathcal{K}$ be a symmetric monoidal 2-category. Then there is an isomorphism between $\mathrm{SM}^{\lax}2\mathrm{Cat}_{\ps}(\mathsf{C}^{\op}, \mathcal{K})$ and $2\mathrm{Cat}_{\ps}(\mathsf{C}^{\op}, \mathrm{SM}(\mathcal{K}))$, where $2\mathrm{Cat}_{\ps}(\mathcal{A}, \mathcal{B})$ is the 2-category of pseudofunctors, pseudonatural transformations, and modifications between $\mathcal{A}$ and $\mathcal{B}$.
\end{proposition}
\begin{proof}[Proof: (Sketch)]
        If $(P, \mu^P, I^P) \colon \mathsf{C}^{\op} \to \mathcal{K}$ is lax symmetric monoidal, then it induces a 2-functor \begin{equation*}
            \SM^{\lax} \Ar^{\colax}(\mathsf{C}^{\op}) \xrightarrow{\SM^{\lax} \Ar^{\colax} (P)} \SM^{\lax}\Ar^{\colax}(\mathcal{K}).
        \end{equation*}
        Since $\mathcal{C}$ is cartesian, we have canonical monoid structures \begin{equation*}
\begin{tikzcd}
	{B\times B} && {A\times A} && 1 && 1 \\
	\\
	B && A && B && A
	\arrow[""{name=0, anchor=center, inner sep=0}, "{f^{\op} \times f^{\op}}", from=1-1, to=1-3]
	\arrow["{\Delta_B^{\op}}"', from=1-1, to=3-1]
	\arrow["{\Delta_A^{\op}}", from=1-3, to=3-3]
	\arrow[""{name=1, anchor=center, inner sep=0}, equals, from=1-5, to=1-7]
	\arrow["{!_B^{\op}}"', from=1-5, to=3-5]
	\arrow["{!_A^{\op}}", from=1-7, to=3-7]
	\arrow[""{name=2, anchor=center, inner sep=0}, "{f^{\op}}"', from=3-1, to=3-3]
	\arrow[""{name=3, anchor=center, inner sep=0}, "{f^{\op}}"', from=3-5, to=3-7]
	\arrow["\Delta_f^{\op}"{description}, color={rgb,255:red,92;green,92;blue,214}, draw=none, from=0, to=2]
	\arrow["\circlearrowleft"{description}, color={rgb,255:red,92;green,92;blue,214}, draw=none, from=1, to=3]
\end{tikzcd}
        \end{equation*} on each object $f^{\op} \colon B \to A$ in $\Ar^{\colax}(\mathsf{C}^{\op})$, which thus get mapped to symmetric pseudomonoids in $ \Ar^{\colax}(\mathcal{K})$. By a simple (but tedious) manual check, or by appealing to the symmetry of internalisation (Theorem 7.5 in \cite{arkor2024enhanced2categoricalstructurestwodimensional}), this makes each fibre $PA$ a symmetric pseudomonoid in $\mathcal{K}$ with multiplication $PA \otimes_{\mathcal{K}} PA \xrightarrow{\mu^P_{A,A}} P(A \times A) \xrightarrow{P(\Delta_A)} A$, and each $Pf$ lax symmetric monoidal. Concretely, the monoidal laxators for $Pf$ are given by
        \begin{equation*}       
\begin{tikzcd}
	&&&&& {I^{\mathcal{K}}} && {I^{\mathcal{K}}} \\
	& {PB \otimes_{\mathcal{K}} PB} && {PA \otimes_{\mathcal{K}} PA} && P1 && P1 \\
	{\mu^{PF}\coloneqq} & {P(B\times B)} && {P(A\times A)} & {I^{Pf}\coloneqq} & P1 && P1 \\
	& PB && PA && PB && PA
	\arrow[equals, from=1-6, to=1-8]
	\arrow["{{I^P}}"', from=1-6, to=2-6]
	\arrow["{{I^P}}", from=1-8, to=2-8]
	\arrow[""{name=0, anchor=center, inner sep=0}, "{{Pf \otimes_{\mathcal{K}} Pf}}", from=2-2, to=2-4]
	\arrow["{{\mu^P_{B,B}}}"', from=2-2, to=3-2]
	\arrow["{{\mu^P_{A,A}}}", from=2-4, to=3-4]
	\arrow[""{name=1, anchor=center, inner sep=0}, equals, from=2-6, to=2-8]
	\arrow[equals, from=2-6, to=3-6]
	\arrow[equals, from=2-8, to=3-8]
	\arrow[""{name=2, anchor=center, inner sep=0}, "{{P(f \times f)}}"', from=3-2, to=3-4]
	\arrow["{{P(\Delta_B)}}"', from=3-2, to=4-2]
	\arrow["{{P(\Delta_A)}}", from=3-4, to=4-4]
	\arrow[""{name=3, anchor=center, inner sep=0}, "{{P(\id_1)}}", from=3-6, to=3-8]
	\arrow["{{P!_B}}"', from=3-6, to=4-6]
	\arrow["{{P!_A}}", from=3-8, to=4-8]
	\arrow[""{name=4, anchor=center, inner sep=0}, "Pf"', from=4-2, to=4-4]
	\arrow[""{name=5, anchor=center, inner sep=0}, "Pf"', from=4-6, to=4-8]
	\arrow["{{\mu^P_{f,f}}}"{description}, color={rgb,255:red,92;green,92;blue,214}, draw=none, from=0, to=2]
	\arrow["{\iota^P_1}"{description}, color={rgb,255:red,92;green,92;blue,214}, draw=none, from=1, to=3]
	\arrow["{P(\Delta_f)}"{description}, color={rgb,255:red,92;green,92;blue,214}, draw=none, from=2, to=4]
	\arrow["\cong"{description}, color={rgb,255:red,92;green,92;blue,214}, draw=none, from=3, to=5]
\end{tikzcd}.
        \end{equation*}
 where $\Delta$ is the diagonal map and $I^{\mathcal{K}}$ is the monoidal unit of $\mathcal{K}$.
 
  As $\mu^P$ and $\iota^P$ are pseudonatural, this shows that each $Pf$ is in fact \emph{strong} monoidal. With this, we can define a 2-functor $(-)'\colon \SM^{\lax}2\mathrm{Cat}_{\ps}(\mathsf{C}^{\op}, \mathcal{K}) \to2\mathrm{Cat}_{\ps}(\mathsf{C}^{\op}, \SM(\mathcal{K}))$ --- it takes a lax symmetric monoidal pseudofunctor $P \colon \mathsf{C}^{\op} \to \mathcal{K}$ to the pseudofunctor $P' \colon \mathsf{C}^{\op}\to \mathrm{SM}(\mathcal{K})$ mapping $A$ to the symmetric pseudomonoid $PA$, and a morphism $f\colon A \to B$ in $\mathsf{C}$ to the strong monoidal map $Pf \colon PB \to PA$.

  A monoidal pseudonatural transformation $\alpha\colon P_1 \to P_2$ gets mapped to the pseudonatural transformation $\alpha' \colon P_1' \to P_2'$ whose component at an object $A$ is the strong monoidal morphism $\alpha_A$: the compatibility cells for the monoidal structure are given by
  \begin{equation*}
\begin{tikzcd}
	{P_1A \otimes_{\mathcal{K}}P_1A} && {P_2A \otimes_{\mathcal{K}} P_2A} \\
	{P_1(A\times A)} && {P_2(A\times A)} \\
	{P_1A} && {P_2A}
	\arrow[""{name=0, anchor=center, inner sep=0}, "{\alpha_A \otimes_{\mathcal{K}} \alpha_A}", from=1-1, to=1-3]
	\arrow["{\mu^{P_1}_{A,A}}"', from=1-1, to=2-1]
	\arrow["{\mu^{P_2}_{A,A}}", from=1-3, to=2-3]
	\arrow[""{name=1, anchor=center, inner sep=0}, "{\alpha_{A \times A}}"', from=2-1, to=2-3]
	\arrow["{P_1(\Delta_A)}"', from=2-1, to=3-1]
	\arrow["{P_2(\Delta_A)}", from=2-3, to=3-3]
	\arrow[""{name=2, anchor=center, inner sep=0}, "{\alpha_A}"', from=3-1, to=3-3]
	\arrow["\alpha^*_{A,A}"{description}, color={rgb,255:red,92;green,92;blue,214}, draw=none, from=0, to=1]
	\arrow["\alpha_{\Delta_A}"{description}, color={rgb,255:red,92;green,92;blue,214}, draw=none, from=1, to=2]
\end{tikzcd},
 \end{equation*}
 \noindent where $\alpha^*$ denotes the invertible modification for the monoidal compatibility of $\alpha$, and 
\begin{equation*}
\begin{tikzcd}
	{I^{\mathcal{K}}} && {I^{\mathcal{K}}} && {I^{\mathcal{K}}} & {I^{\mathcal{K}}} \\
	&& {P_1(1)} && {P_1(1)} \\
	&& {P_1(1)} && {P_1(1)} & {P_1A} \\
	{I^{\mathcal{K}}} && {P_2(1)} && {P_2(1)} & {P_2A} \\
	{P_2(1)} && {P_2(1)} && {P_2(1)} & {P_2A}
	\arrow[""{name=0, anchor=center, inner sep=0}, equals, from=1-1, to=1-3]
	\arrow[equals, from=1-1, to=4-1]
	\arrow[""{name=1, anchor=center, inner sep=0}, equals, from=1-3, to=1-5]
	\arrow["{I^{P_1}}"', from=1-3, to=2-3]
	\arrow[""{name=2, anchor=center, inner sep=0}, equals, from=1-5, to=1-6]
	\arrow["{I^{P_1}}", from=1-5, to=2-5]
	\arrow["{P_1!_A \circ I^{P_1}}", from=1-6, to=3-6]
	\arrow[""{name=3, anchor=center, inner sep=0}, equals, from=2-3, to=2-5]
	\arrow[equals, from=2-3, to=3-3]
	\arrow[equals, from=2-5, to=3-5]
	\arrow[""{name=4, anchor=center, inner sep=0}, "{P_1(\id_1)}", from=3-3, to=3-5]
	\arrow["{\alpha_1}"', from=3-3, to=4-3]
	\arrow[""{name=5, anchor=center, inner sep=0}, "{P_1!_A}", from=3-5, to=3-6]
	\arrow["{\alpha_1}"', from=3-5, to=4-5]
	\arrow["{\alpha_A}", from=3-6, to=4-6]
	\arrow[""{name=6, anchor=center, inner sep=0}, "{I^{P_2}}"', from=4-1, to=4-3]
	\arrow["{I^{P^2}}"', from=4-1, to=5-1]
	\arrow[""{name=7, anchor=center, inner sep=0}, "{P_2(\id_1)}"', from=4-3, to=4-5]
	\arrow[equals, from=4-3, to=5-3]
	\arrow[""{name=8, anchor=center, inner sep=0}, "{P_2!_A}"', from=4-5, to=4-6]
	\arrow[equals, from=4-5, to=5-5]
	\arrow[equals, from=4-6, to=5-6]
	\arrow[""{name=9, anchor=center, inner sep=0}, equals, from=5-1, to=5-3]
	\arrow[""{name=10, anchor=center, inner sep=0}, equals, from=5-3, to=5-5]
	\arrow[""{name=11, anchor=center, inner sep=0}, "{P_2!_A}"', from=5-5, to=5-6]
	\arrow["\cong"{description}, color={rgb,255:red,92;green,92;blue,214}, draw=none, from=0, to=6]
	\arrow["{=}"{description}, color={rgb,255:red,92;green,92;blue,214}, draw=none, from=1, to=3]
	\arrow["{=}"{description}, color={rgb,255:red,92;green,92;blue,214}, draw=none, from=2, to=5]
	\arrow["{\iota^{P_1}_1}"{description}, color={rgb,255:red,92;green,92;blue,214}, draw=none, from=3, to=4]
	\arrow["{\alpha_{\id_1}}"{description}, color={rgb,255:red,92;green,92;blue,214}, draw=none, from=4, to=7]
	\arrow["{\alpha_{!_A}}"{description}, color={rgb,255:red,92;green,92;blue,214}, draw=none, from=5, to=8]
	\arrow["{{\iota^{P_2}_1}^{-1}}"{description}, color={rgb,255:red,92;green,92;blue,214}, draw=none, from=7, to=10]
	\arrow["{=}"{description}, color={rgb,255:red,92;green,92;blue,214}, draw=none, from=9, to=6]
	\arrow["{=}"{description}, color={rgb,255:red,92;green,92;blue,214}, draw=none, from=11, to=8]
\end{tikzcd}.
\end{equation*}
Similarly, its naturality cell for $f \colon A \to B$ in $\mathsf{C}$ is just $\alpha_f$ (seen as a monoidal 2-cell in $\mathcal{K}$). 

By the same line of reasoning, a monoidal modification $\aleph: \alpha \to \beta$ between monoidal pseudonatural transformations $P_1 \to P_2$ is such that each component $\aleph_A \colon \alpha_A \to \beta_A$ is a monoidal 2-cell between the strong monoidal morphisms $\alpha_A$ and $\beta_A$. The argument for the compatibility with the pseudomonoid multiplications is simple: since the 2-cell $\alpha_A \otimes_{P_2} \alpha_A \Rightarrow \alpha_A \otimes_{P_1}$ is given by $\frac{\alpha^*_{A,A}}{\alpha_{\Delta_A}}$, we just need to check that $\frac{\frac{\alpha^*_{A,A}}{\alpha_{\Delta_A}}}{\aleph_A}= \frac{\frac{\aleph_A \otimes \aleph_A}{\beta^*_{A,A}}}{\beta_{\Delta_A}}$ for every object $A$ in $\mathsf{C}$. But that readily follows from $\aleph$ being a monoidal modification, natural with respect to the morphism $\Delta_A^{\op} \colon A\times A \to A$ in $\mathsf{C}^{\op}$. Hence such a modification can be mapped to itself, but seen as a modification $\alpha'_A \to \beta'_A$. This gives a 2-functor $(-)' \colon \mathrm{SM}^{\lax}2\mathrm{Cat}_{\ps}(\mathsf{C}^{\op}, \mathcal{K}) \to 2\mathrm{Cat}_{\ps}(\mathsf{C}^{\op}, \mathrm{SM}(\mathcal{K}))$.

        Conversely, $P' \colon \mathsf{C}^{\op}\to \mathrm{SM}(\mathcal{K})$, induces a pseudofunctor $P \colon \mathsf{C}^{\op} \xrightarrow{P'} \SM(\mathcal{K}) \xrightarrow{U} \mathcal{K}$, where $U$ is the forgetful 2-functor. We can then construct  a monoidal laxator $P \otimes_{\mathcal{K}} P \xrightarrow{\mu^P} P(-\times-)$ for $P$ whose component at $(A,B)$ is $PA \otimes_{\mathcal{K}}   PB \xrightarrow {P(\pi _A) \otimes_{\mathcal{K}}   P(\pi _B)} P(A\times  B)^2 \xrightarrow {\otimes _{P(A \times  B)}} P(A\times  B)$, where \(\pi \) denotes projections as usual and $\otimes_{P'(A \times B)}$ is the pseudomonoid multiplication of $P'(A \times B)$. The monoidal unit is the morphism $I^{\mathcal{K}} \xrightarrow{I^{P(1)}} P(1)$ in $\mathcal{K}$. The fact that each $P'(f)$ is strong monoidal ensures that $\mu^P$ is pseudonatural: the naturality 2-cell of $\mu^P$ with respect to $f\times g\colon A \times B \to X \times Y$ in $\mathsf{C}$ is
        \begin{equation*}
\begin{tikzcd}
	{PX \otimes_{\mathcal{K}}  PY} && {PA \otimes_{\mathcal{K}}  PB} \\
	{P(X \times Y) \times P(X \times Y)} && {P(A \times B) \otimes_{\mathcal{K}}  P(A \times B)} \\
	{P(X\times Y)} && {P(A \times B)}
	\arrow[""{name=0, anchor=center, inner sep=0}, "{Pf \otimes_{\mathcal{K}}  Pg}", from=1-1, to=1-3]
	\arrow["{P(\pi_X) \otimes_{\mathcal{K}}  P(\pi_Y)}"', from=1-1, to=2-1]
	\arrow["{P(\pi_A) \otimes_{\mathcal{K}}  P(\pi_B)}", from=1-3, to=2-3]
	\arrow[""{name=1, anchor=center, inner sep=0}, "{\tiny{P(f\times g) \otimes_{\mathcal{K}}  P(f \times  g)}}", from=2-1, to=2-3]
	\arrow["{\otimes_{P(X\times Y)}}"', from=2-1, to=3-1]
	\arrow["{\otimes_{P(A\times B)}}", from=2-3, to=3-3]
	\arrow[""{name=2, anchor=center, inner sep=0}, "{P(f\times g)}"', from=3-1, to=3-3]
	\arrow["\cong"{description}, color={rgb,255:red,92;green,92;blue,214}, draw=none, from=0, to=1]
	\arrow["{\mu^{P(f\times g)}}"{description}, color={rgb,255:red,92;green,92;blue,214}, draw=none, from=1, to=2]
\end{tikzcd}.
        \end{equation*}

Since $\mathsf{C}$ is cartesian, $P$ is in fact symmetric monoidal with this structure: consider the pasting diagram below, where $\sigma$ denotes symmetries:
\begin{equation*}
\begin{tikzcd}
	{PX \otimes_{\mathcal{K}}  PY} && {PX \otimes_{\mathcal{K}}  PY} &&& {PY \otimes_{\mathcal{K}}  PX} \\
	\\
	{P(X \times Y)^2} && {P(Y\times X)^2} &&& {P(Y \times X) ^2} \\
	\\
	{P(X \times Y)} && {P(Y \times X)} &&& {P(Y \times X)}
	\arrow[""{name=0, anchor=center, inner sep=0}, equals, from=1-1, to=1-3]
	\arrow["{{\scriptstyle{P(\pi^{X \times Y}_X) \otimes_{\mathcal{K}}  P(\pi^{X \times Y}_Y)}}}"', from=1-1, to=3-1]
	\arrow[""{name=1, anchor=center, inner sep=0}, "{{\sigma^{\mathcal{K}}_{PX, PY}}}", from=1-3, to=1-6]
	\arrow["{{\scriptstyle{P(\pi^{Y \times X}_X) \otimes_{\mathcal{K}}  P(\pi^{Y \times X}_Y)}}}"', from=1-3, to=3-3]
	\arrow["{{P(\pi_Y) \otimes_{\mathcal{K}}  P(\pi_X)}}", from=1-6, to=3-6]
	\arrow[""{name=2, anchor=center, inner sep=0}, "{{\scriptstyle{P(\sigma^{\mathsf{C}}_{Y, X}) \otimes_{\mathcal{K}}  P(\sigma^{\mathsf{C}}_{Y, X})}}}"', from=3-1, to=3-3]
	\arrow["{{\otimes_{P'(X\times Y)}}}"', from=3-1, to=5-1]
	\arrow[""{name=3, anchor=center, inner sep=0}, equals, from=3-3, to=3-6]
	\arrow["{{\otimes_{P'(Y\times X)}}}"', from=3-3, to=5-3]
	\arrow["{{\otimes_{P'(Y\times X)}}}", from=3-6, to=5-6]
	\arrow[""{name=4, anchor=center, inner sep=0}, "{{P(\sigma^{\mathsf{C}}_{Y, X})}}"', from=5-1, to=5-3]
	\arrow[""{name=5, anchor=center, inner sep=0}, equals, from=5-3, to=5-6]
	\arrow["\cong"{description, pos=0.3}, shift right=5, color={rgb,255:red,92;green,92;blue,214}, draw=none, from=0, to=2]
	\arrow["{{\scriptstyle{\sigma^{\mathcal{K}}_{P(\pi^{Y \times X}_X) \otimes_{\mathcal{K}}  P(\pi^{Y\times X}_{Y})}}}}"{description}, shift left=4, color={rgb,255:red,92;green,92;blue,214}, draw=none, from=1, to=3]
	\arrow["\cong"{description}, color={rgb,255:red,92;green,92;blue,214}, draw=none, from=2, to=4]
	\arrow["{=}"{description}, shift left=2, color={rgb,255:red,92;green,92;blue,214}, draw=none, from=3, to=5]
\end{tikzcd}.
\end{equation*}
The top-left cell is invertible since $\pi^{X \times Y}_X = \pi^{Y \times X}_X \circ \sigma^{\mathsf{C}}_{X,Y}$ and $\pi^{X \times Y}_Y = \pi^{Y \times X}_Y \circ \sigma^{\mathsf{C}}_{X,Y}$ in $\mathsf{C}$, and so is the bottom-left cell since the morphism $\sigma^{\mathsf{C}}_{Y, X}$ must be mapped by $P'$ to a strong monoidal morphism.

If $\alpha' \colon P_1' \to P_2' \colon \mathsf{C}^{\op} \to \SM(\mathcal{K})$ is a pseudonatural transformation, then whiskering with the forgetful 2-functor $\SM(\mathcal{K}) \to \mathcal{K}$ gives a pseudonatural transformation $\alpha\coloneqq \colon P_1 \to P_2 \colon \mathsf{C}^{\op} \to \mathcal{K}$. It is in fact monoidal: since each $\alpha'_A$ is a strong morphism of pseudomonoids, we have invertible 2-cells $\mu^{\alpha'_A} \colon \otimes_{P_2'A} (\alpha'_A \otimes \alpha'_A) \Rightarrow \alpha'_A \otimes_{P_1'A}$, and we can combine it with the naturality squares for the projections to build the invertible modification witnessing that $\alpha= U\alpha'$ is monoidal. Explicitly, its component at objects $A$ and $B$ is the 2-cell 
\begin{equation*}
    \begin{tikzcd}
	{P_1A \otimes P_1B} && {P_2A \otimes P_2B} \\
	{P_1(A \times B)\otimes P_1(A\times B)} && {P_2(A \times B) \otimes P_2(A \times B)} \\
	{P_1(A\times B)} && {P_2(A\times B)}
	\arrow[""{name=0, anchor=center, inner sep=0}, "{\alpha_A \otimes \alpha_B}", from=1-1, to=1-3]
	\arrow["{P_1(\pi^{A \times B}_A) \otimes P_1(\pi^{A \times B}_B)}"', from=1-1, to=2-1]
	\arrow["{P_2(\pi^{A \times B}_A) \otimes P_2(\pi^{A \times B}_B)}", from=1-3, to=2-3]
	\arrow[""{name=1, anchor=center, inner sep=0}, "{\alpha_{A \times B} \otimes \alpha_{A \times B}}"', from=2-1, to=2-3]
	\arrow["{\otimes_{P_1'(A\times B)}}"', from=2-1, to=3-1]
	\arrow["{\otimes_{P'_2(A \times B)}}", from=2-3, to=3-3]
	\arrow[""{name=2, anchor=center, inner sep=0}, "{\alpha_{A \times B}}"', from=3-1, to=3-3]
	\arrow["{\alpha_{\pi^{A \times B}_A} \otimes \alpha_{\pi^{A \times B}_B}}"{description}, color={rgb,255:red,92;green,92;blue,214}, draw=none, from=0, to=1]
	\arrow["{\mu^{\alpha'}_{A \times B}}"{description}, color={rgb,255:red,214;green,92;blue,92}, draw=none, from=1, to=2]
\end{tikzcd}
\end{equation*}

Finally, if $\aleph' \colon \alpha' \to \beta'$ is a modification, then each component is a 2-morphism in $\SM(\mathcal{K})$, i.e., $\frac{\mu^{\alpha'_A}}{\aleph'_A} = \frac{\aleph'_A}{\mu^{\beta'_A}}$ for every object $A \in \mathsf{C}$. Once again applying the forgetful 2-functor yields a modification $\aleph \colon \alpha \Rightarrow \beta \colon P_1 \to P_2$. This is in fact monoidal by another explicit check using the construction of the monoidal structures for $\alpha$ and $\beta$, thus a 2-cell in $\mathrm{SM}^{\lax}2\mathrm{Cat}_{\ps}(\mathsf{C}^{\op}, \mathcal{K})$. Thus, whiskering with the forgetful 2-functor yields a 2-functor $\mathrm{SM}^{\lax}2\mathrm{Cat}_{\ps}(\mathsf{C}^{\op}, \mathcal{K}) \to \SM{\lax}(\mathsf{C}^{\op}, \mathcal{K})$, which is an inverse to $(-)'$.
\end{proof}

In a more restrictive manner, we also have a colax monoidal structure locally, for the existential quantifiers:  if $P$ is a $\mathfrak{C}$-existential monoidal structure and $f \colon A \to B$ is a morphism in $\mathfrak{C}^r$, then the left adjoint $\exists^P f$ to the strong monoidal functor $Pf \colon PB \to PA$ is colax monoidal by the Doctrinal Adjunction Theorem \cite{doctrinal}. Its monoidal structure is obtained from that of $Pf$ by taking mates.
With that, we can define generalised regular hyperdoctrines, where the Frob\"enius law only holds for morphisms in the right class (i.e., the ones for which existential quantification is defined):
    
\begin{definition}[$\mathfrak{C}$-regular hyperdoctrine]
    A  \(\mathfrak{C}\)-\emph{regular $\mathcal{K}$-hyperdoctrine} is a \(\mathfrak{C}\)-existential symmetric monoidal $\mathcal{K}$-structure for which the Frob\"enius law holds with respect to $\mathfrak{C}$: for each \(f \colon  A \to  B\) in \(\mathfrak{C}^r\), the canonical 2-cell

\begin{center}\adjustbox{max width=\textwidth}{\begin{tikzcd}
	{PB \otimes  PA} && {PB \otimes  PB} \\
	\\
	{PA \otimes  PA} && {PB\otimes  PB} \\
	\\
	PA && PB
	\arrow[""{name=0, anchor=center, inner sep=0}, "{{\mathrm {id}_{PB} \otimes  \exists  f}}", from=1-1, to=1-3]
	\arrow["{{Pf \otimes  \mathrm {id}_{PA}}}"', from=1-1, to=3-1]
	\arrow[equals, from=1-3, to=3-3]
	\arrow[""{name=1, anchor=center, inner sep=0}, "{{\exists  f\otimes  \exists  f}}", from=3-1, to=3-3]
	\arrow["{{\otimes_{PA}}}"', from=3-1, to=5-1]
	\arrow["{{\otimes_{PB}}}", from=3-3, to=5-3]
	\arrow[""{name=2, anchor=center, inner sep=0}, "{{\exists  f}}"', from=5-1, to=5-3]
	\arrow["{{\epsilon^f \otimes  v_{\exists  f}}}"', color={rgb,255:red,92;green,92;blue,214}, between={0.2}{0.8}, Rightarrow, from=1, to=0]
	\arrow["{\mathrm{mate}(\mu^{Pf})}"', color={rgb,255:red,92;green,92;blue,214}, between={0.2}{0.8}, Rightarrow, from=2, to=1]
\end{tikzcd}}\end{center}
 
 has an inverse \(\mathfrak{F}^r_f\), where \(\mu ^{Pf}\) is the laxator for \(Pf \colon  (PB, \otimes_{PB}, I^{PB}) \to  (PA, \otimes_{PA}, I^{PA})\).
 \end{definition}

\begin{remark}[{}]
\par{}We could also ask for the analogous cell \(\exists  f \circ  \otimes_{PA} \circ  (\mathrm {id}_{PA} \otimes  Pf) \Rightarrow  \otimes_{PB} \circ  (\exists  f \otimes  \mathrm {id}_{PB})\) to have an inverse
\(\mathfrak{F}^l_f\). Since \(\mathcal {K}\) is a symmetric monoidal 2-category, we can use its symmetry to show that \(\mathfrak{F}^l_f\) exists precisely when \(\mathfrak{F}^r_f\) does. \end{remark}

The matter of hyperdoctrines that are pseudofunctorial  versus strictly 2-functorial is more delicate, and it is one that must be contended with. Definition~\ref{def:ExistentialStructure} is made considerably simpler if we allow only 2-functors and 2-natural transformations as we do --- this is because 2-categories, 2-functors, and pseudonatural transformations do not form a 2-category, but rather a tricategory. For this reason, this paper focuses on the theory for 2-functorial hyperdoctrines.

Still, the pseudofunctorial examples to come can be strictified. Following a result of Power \cite{POWER1989165}, Lack showed that if $\mathcal{K}$ is a 2-category with an enhanced factorization system, then the inclusion of $\text{T-Alg}_s$ into $\text{Ps-T-Alg}$ for $T$ a 2-monad on $\mathcal{K}$ has
a left adjoint, and the components of the unit of the adjunction are equivalences in $\text{Ps-T-Alg}$ (Theorem 4.10 in \cite{LACK2002223}). In particular, any pseudofunctor $\mathcal{C}^{\op} \to \mathrm{Cat}$ for a 2-category $\mathcal{C}$ is pseudonaturally equivalent to a strict 2-functor $\mathcal{C}^{\op} \to \mathrm{Cat}$. However, it is not possible to simultaneously strictify pseudofunctors and pseudonatural transformations. Of course, these considerations do not come into play when the codomain of the pseudofunctors is a locally posetal $2$-category such as $\mathrm{Pos}$ and $\mathrm{Rel}$, as is the case for classical hyperdoctrines, and we can recover all classical regular hyperdoctrines:

\begin{example}
    If \(\mathfrak{C}\) is a trivial cartesian adequate triple  and \(\mathcal{K}=\mathrm{Pos}\) is the 2-category of posets, order-preserving maps and inequalities of maps,
then a \(\mathfrak{C}\)-regular $\mathcal{K}$-hyperdoctrine $P$ such that all fibres \(PA\) are cartesian monoids (i.e., \(\land \)-semilattices) is equivalent to a regular hyperdoctrine in the usual sense.
\end{example}

The following examples still use a trivial cartesian adequate triple, but have a non-cartesian structure of logical significance on the fibres:
\begin{example}[{}]
\par{}Consider the underlying monoidal poset of the \emph{Lawvere quantale}: the extended positive real numbers \([0,\infty ]\) ordered by \(\geq \) with addition as the tensor product. For each set \(A\), define \(PA\coloneqq  [0,\infty ]^{A}\) (i.e., a predicate of type \(A\) is a function \(A \to  [0, \infty ]\)), which is a monoidal poset under \(+\) and \(\geq \). For each function \(f\colon  A \to  B\),
precomposing predicates with \(f\) yields an order-preserving map \(Pf \colon  PB \to  PA\) (which is strict monoidal), and the assignment \(f \mapsto  Pf\) is strictly functorial. Each such \(Pf\) admits a left adjoint,
namely \(\exists  f \colon  PA \to  PB\) taking a predicate \(\phi \) over \(A\) to the predicate \(b \mapsto  \inf _{a \in  f^{-1}(b)} (\phi (a))\) over \(B\). It is easy to check that
this yields a regular $\mathrm{Pos}$-hyperdoctrine for the trivial adequate triple on $\mathrm{Set}$. Details on the logic of \([0, \infty ]\) can be found in \cite{bacci_2023}.\end{example}
\par{}There are also natural examples where we need to forego fibres being posets:
\begin{example}[{}]
\label{example:shulman}
\par{}If \(\mathsf {C}\) is a category with finite limits, then we can define a pseudofunctor
 \(P \colon  \mathcal {C}^{{\op}} \to  \SM{(\mathrm{Cat})}\) by mapping an object \(X\) to the slice category \(\mathsf {C}\downarrow  X\)
 (which is cartesian monoidal), and a morphism \(f \colon  X \to  Y\) to the (strong monoidal) functor \(Pf \colon  PY \to  PX\) given by pulling back along \(f\).
 Each such \(Pf\) has a left adjoint \(\Sigma  (f)\) (the \emph{dependent sum}), and the fact that these adjunctions satisfy the Beck-Chevalley and Frob\"enius properties
 is the well-known pullback pasting lemma.\par{}More interestingly, as shown in \cite{Shulman2011}, if \(\mathsf {Top}\) denotes the category of compactly generated topological spaces, 
then each slice  \(\mathsf {Top}\downarrow  B\) is a model category, and we can view its homotopy category as 
the object of predicates on \(B\). Each such \(\mathsf {Ho}(\mathsf {Top}\downarrow  B)\) is a monoidal category, 
and the contravariant mapping associating the space \(B\) to
it  strictifies to a $\mathfrak{C}$-regular $\mathrm{Cat}$-hyperdoctrine, where $\mathfrak{C}$ is the adequate triple on $\mathsf{Top}$ that has has the Serre fibrations as $\mathfrak{C}^l$, and arbitrary morphisms as $\mathfrak{C}^r$.\par{}The same paper also provides an example of a (co)-regular $\mathrm{Cat}$-hyperdoctrine for the adequate triple on locally compact Hausdorff spaces whose
left class includes all maps and whose right class consists of the proper maps. In this setting, the symmetric monoidal category of predicates over a space \(A\) is
the homotopy category of a model category of unbounded chain complexes of sheaves of abelian groups on \(A\).  \end{example}

The introduction of adequate triples relied on the fact that the Beck-Chevalley condition is faithfully captured by the algebra of spans, which form a double category. While we expect the reader to be familiar with the fundamentals of double categories, the terminology in the field can be notoriously varied, so we take the opportunity here to settle notation and jargon. The basic reference is \cite{grandis-2019-higher}, but the reader must beware the difference in terminology.


\begin{definition}[{Double category}]
\par{}By a \emph{double category} \(\mathbb {D}\) we mean a category \emph{weakly} internal$^{\footnotesize 3}$ to \(\mathsf{CAT}\). \footnote{$^{\footnotesize 3}$ This is what is sometimes referred to as a pseudo double category or weak double category in the literature.} More explicitly, it consists
of a category \(\mathbb {D}_0\) of objects and tight arrows (written vertically with the usual arrows), and 
a category \(\mathbb {D}_1\) of loose arrows (written horizontally with a marked arrow \(\mathrel {\mkern 3mu\vcenter {\hbox {$\shortmid $}}\mkern -10mu{\to }}\)) and squares between them, along with 
source and target functors \(S, T\colon  \mathbb {D}_1 \to  \mathbb {D}_0\), an identity assignment functor \(e \colon  \mathbb {D}_0 \to  \mathbb {D}_1\), and a composition functor \(\odot  \colon  \mathbb {D}_1 \times _{\mathbb {D}_0} \mathbb {D}_1 \to  \mathbb {D}_1\), equipped with
associativity \((X \odot  X') \odot  X'' \xrightarrow {\alpha _{X, X', X''}} X \odot  (X' \odot  X'')\) and unitality constraints \(e_{X_1} \odot  X \xrightarrow {\lambda _X} X\), \(X \odot  e_{X_2} \xrightarrow {\rho _X} X\) satisfying the usual coherence conditions (see Def. 2.1 in \cite{shulman-2008-framed-bicats}  for details). \end{definition}


\par{}We have four ways to compose data in \(\mathbb {D}\): tight arrows compose as in \(\mathbb {D}_0\),
 cells compose vertically as in \(\mathbb {D}_1\) and horizontally via \(\odot \), and loose arrows compose via the objects part of \(\odot \) (written in diagrammatic order). Note that
 composition of loose arrows is only weakly associative and unital, which is why the descriptive terminology `loose and tight' is used over the `horizontal and vertical' common in the literature (c.f. \cite{grandis-2019-higher}).
 If this composition is strict, we will say that \(\mathbb {D}\) is a strict double category.
 
 \par{} If $X$ is a loose morphism, we write $v_X$ for its identity square in $\mathbb{D}_1$. We will reserve the notation \(\mathrm {id}_X\) for tight identities, \(\alpha  \mid  \beta \) for the horizontal composition of cells, and 
 \(\frac {\alpha }{\beta }\) for vertical composition, in diagrammatic order. We thus have that 
 \(\frac {(\alpha  \mid  \beta )}{(\delta  \mid  \gamma )} = (\frac {\alpha }{ \delta}) \mid  (\frac{\beta}{\gamma})\) for all suitably composable squares, so that composition can be read off of a diagram
unambiguously.
\begin{center}\begin{tikzcd}
	{A_1} && {A_2} && {A_3} \\
	\\
	{B_1} && {B_2} && {B_3} \\
	\\
	{C_1} && {C_2} && {C_3}
	\arrow[""{name=0, anchor=center, inner sep=0}, "\shortmid"{marking}, from=1-1, to=1-3]
	\arrow["{{f_1}}"', from=1-1, to=3-1]
	\arrow[""{name=1, anchor=center, inner sep=0}, "\shortmid"{marking}, from=1-3, to=1-5]
	\arrow["{{g_1}}", from=1-3, to=3-3]
	\arrow["{{h_1}}", from=1-5, to=3-5]
	\arrow[""{name=2, anchor=center, inner sep=0}, "\shortmid"{marking}, from=3-1, to=3-3]
	\arrow["{{f_2}}"', from=3-1, to=5-1]
	\arrow[""{name=3, anchor=center, inner sep=0}, "\shortmid"{marking}, from=3-3, to=3-5]
	\arrow["{{g_2}}", from=3-3, to=5-3]
	\arrow["{{h_2}}", from=3-5, to=5-5]
	\arrow[""{name=4, anchor=center, inner sep=0}, "\shortmid"{marking}, from=5-1, to=5-3]
	\arrow[""{name=5, anchor=center, inner sep=0}, "\shortmid"{marking}, from=5-3, to=5-5]
	\arrow["{\alpha }"{description}, color={rgb,255:red,92;green,92;blue,214}, draw=none, from=2, to=0]
	\arrow["{\beta }"{description}, color={rgb,255:red,92;green,92;blue,214}, draw=none, from=3, to=1]
	\arrow["{\delta }"{description}, color={rgb,255:red,92;green,92;blue,214}, draw=none, from=4, to=2]
	\arrow["{\gamma }"{description}, color={rgb,255:red,92;green,92;blue,214}, draw=none, from=5, to=3]
\end{tikzcd}\end{center}

 \par{}There are many ways to produce double categories via involution. Most importantly for us,
 flipping the direction of the tight arrows (but not of the loose arrows) in a double category \(\mathbb {D}\) yields
 a new double category \(\mathbb {D}^{{\op}}\) in the obvious way.\\

 There are many notions of morphism between double categories. Here we focus on pseudo double functors, but the reader determined to keep regular hyperdoctrines pseudofunctorial rather than 2-functorial may be interested in the more general notion of double pseudofunctor (c.f. \cite{Shulman2011}):

 \begin{definition}[{Double pseudofunctor}]
\label{def:DoublePseudofunctor}
A double pseudofunctor \(Q \colon  \mathbb {A} \to  \mathbb {B}\) between double categories consists of mappings $Q_0 \colon  \mathbb {A}_0 \to  \mathbb {B}_0$ and $Q_1 \colon  \mathbb {A}_1 \to  \mathbb {B}_1$ (the tight and loose components of \(Q\), respectively$^{\footnotesize 4}$ \footnote{$^{\footnotesize 4}$ We will simply write $Q$ for both $Q_0$ and $Q_1$ on diagrams.})
    that preserve sources and targets (i.e., \(S \circ  Q_1 = Q_0 \circ  S\) and 
\(T \circ  Q_1 = Q_0 \circ  T\)). They come equipped with structure squares

\begin{equation*}
 \begin{tikzcd}
	&&&&&& QA & QA \\
	QA && QA & {QX_1} && {QX_2} & QA & QA \\
	QA && QA & {QX_1} & {QX_2} & {QX_2} & {QZ_1} & {QZ_1} \\
	{QX_1} & {QX_2} & {QX_3} & {QX_1} && {QX_2} & {QY_1} \\
	{QX_1} && {QX_3} & {QX_1} & {QX_1} & {QX_2} & {QX_1} & {QX_1}
	\arrow[""{name=0, anchor=center, inner sep=0}, "\shortmid"{marking}, equals, from=1-7, to=1-8]
	\arrow[equals, from=1-7, to=2-7]
	\arrow["{Q(\id_A)}", from=1-8, to=2-8]
	\arrow[""{name=1, anchor=center, inner sep=0}, "\shortmid"{marking}, equals, from=2-1, to=2-3]
	\arrow[equals, from=2-1, to=3-1]
	\arrow[equals, from=2-3, to=3-3]
	\arrow[""{name=2, anchor=center, inner sep=0}, "QX"{inner sep=.8ex}, "\shortmid"{marking}, from=2-4, to=2-6]
	\arrow[equals, from=2-4, to=3-4]
	\arrow[equals, from=2-6, to=3-6]
	\arrow[""{name=3, anchor=center, inner sep=0}, "\shortmid"{marking}, equals, from=2-7, to=2-8]
	\arrow[""{name=4, anchor=center, inner sep=0}, "{Q(e_A)}"'{inner sep=.8ex}, "\shortmid"{marking}, from=3-1, to=3-3]
	\arrow["QX"'{inner sep=.8ex}, "\shortmid"{marking}, from=3-4, to=3-5]
	\arrow["\shortmid"{marking}, equals, from=3-5, to=3-6]
	\arrow[""{name=5, anchor=center, inner sep=0}, "\shortmid"{marking}, equals, from=3-7, to=3-8]
	\arrow["{{Qg_1}}"', from=3-7, to=4-7]
	\arrow["{{Q(f_1 g_1)}}", from=3-8, to=5-8]
	\arrow["QX"{inner sep=.8ex}, "\shortmid"{marking}, from=4-1, to=4-2]
	\arrow[equals, from=4-1, to=5-1]
	\arrow["{QX'}"{inner sep=.8ex}, "\shortmid"{marking}, from=4-2, to=4-3]
	\arrow[equals, from=4-3, to=5-3]
	\arrow[""{name=6, anchor=center, inner sep=0}, "QX"{inner sep=.8ex}, "\shortmid"{marking}, from=4-4, to=4-6]
	\arrow[equals, from=4-4, to=5-4]
	\arrow[equals, from=4-6, to=5-6]
	\arrow["{{Qf_1}}"', from=4-7, to=5-7]
	\arrow[""{name=7, anchor=center, inner sep=0}, "{Q(X \odot X')}"'{inner sep=.8ex}, "\shortmid"{marking}, from=5-1, to=5-3]
	\arrow["\shortmid"{marking}, equals, from=5-4, to=5-5]
	\arrow["QX"'{inner sep=.8ex}, "\shortmid"{marking}, from=5-5, to=5-6]
	\arrow[""{name=8, anchor=center, inner sep=0}, "\shortmid"{marking}, equals, from=5-7, to=5-8]
	\arrow["{\iota^Q_A}"{description}, color={rgb,255:red,214;green,92;blue,92}, draw=none, from=3, to=0]
	\arrow["{(Q_e)_A}"{description}, color={rgb,255:red,92;green,92;blue,214}, draw=none, from=4, to=1]
	\arrow["{\rho^Q_X}"{description, pos=0.4}, color={rgb,255:red,92;green,92;blue,214}, draw=none, from=3-5, to=2]
	\arrow["{(Q_\odot)_{X, X'}}"{description}, color={rgb,255:red,92;green,92;blue,214}, draw=none, from=7, to=4-2]
	\arrow["{\lambda^Q_X}"{description, pos=0.4}, color={rgb,255:red,92;green,92;blue,214}, draw=none, from=5-5, to=6]
	\arrow["{\gamma^Q_{f_1,g_1}}"{description}, color={rgb,255:red,214;green,92;blue,92}, draw=none, from=8, to=5]
\end{tikzcd}
\end{equation*}

\noindent in $\mathbb{B}$ for each object $A$, composable tight morphisms $f_1, g_1$, loose morphism $X$, and composable loose morphisms $X$ and $X'$ in $\mathbb{A}$, which are globular isomorphisms (i.e., isomorphisms in $\mathbb{B}_1$ for the blue squares and in $\mathbb{B}_0$ for the red ones). They make $Q_0$ and $Q_1$ weakly preserve tight and loose composition in the following sense: 

\begin{itemize}
    \item{}[Unital coherence]: For every loose arrow \(X_1 \overset {X}{\mathrel {\mkern 3mu\vcenter {\hbox {$\shortmid $}}\mkern -10mu{\to }}} X_2  \in  \mathbb {A}\), 
    
\begin{center}
	\adjustbox{max width=\textwidth}{\begin{tikzcd}
	{QX_1} &&&& {QX_2} && {QX_1} &&&& {QX_2} \\
	{QX_1} && {QX_1} && {QX_2} && {QX_1} && {Q(X_2)} && {QX_2} \\
	&&&&& {=} \\
	{QX_1} && {QX_1} && {QX_2} && {QX_1} && {QX_2} && {QX_2} \\
	{QX_1} &&&& {QX_2} && {QX_1} &&&& {QX_2}
	\arrow[""{name=0, anchor=center, inner sep=0}, "QX"{inner sep=.8ex}, "\shortmid"{marking}, from=1-1, to=1-5]
	\arrow[equals, from=1-1, to=2-1]
	\arrow[equals, from=1-5, to=2-5]
	\arrow[""{name=1, anchor=center, inner sep=0}, "QX"{inner sep=.8ex}, "\shortmid"{marking}, from=1-7, to=1-11]
	\arrow[equals, from=1-7, to=2-7]
	\arrow[equals, from=1-11, to=2-11]
	\arrow[""{name=2, anchor=center, inner sep=0}, "\shortmid"{marking}, equals, from=2-1, to=2-3]
	\arrow[equals, from=2-1, to=4-1]
	\arrow[""{name=3, anchor=center, inner sep=0}, "QX"'{inner sep=.8ex}, "\shortmid"{marking}, from=2-3, to=2-5]
	\arrow["{Q(\id_{X_1})}"', from=2-3, to=4-3]
	\arrow[equals, from=2-5, to=4-5]
	\arrow[""{name=4, anchor=center, inner sep=0}, "QX"'{inner sep=.8ex}, "\shortmid"{marking}, from=2-7, to=2-9]
	\arrow[equals, from=2-7, to=4-7]
	\arrow[""{name=5, anchor=center, inner sep=0}, "\shortmid"{marking}, equals, from=2-9, to=2-11]
	\arrow[equals, from=2-9, to=4-9]
	\arrow["{Q(\id_{X_2})}", from=2-11, to=4-11]
	\arrow[""{name=6, anchor=center, inner sep=0}, "\shortmid"{marking}, equals, from=4-1, to=4-3]
	\arrow[equals, from=4-1, to=5-1]
	\arrow[""{name=7, anchor=center, inner sep=0}, "QX"'{inner sep=.8ex}, "\shortmid"{marking}, from=4-3, to=4-5]
	\arrow[equals, from=4-5, to=5-5]
	\arrow[""{name=8, anchor=center, inner sep=0}, "QX"'{inner sep=.8ex}, "\shortmid"{marking}, from=4-7, to=4-9]
	\arrow[draw=none, from=4-7, to=5-7]
	\arrow[equals, from=4-7, to=5-7]
	\arrow[""{name=9, anchor=center, inner sep=0}, "\shortmid"{marking}, equals, from=4-9, to=4-11]
	\arrow[equals, from=4-11, to=5-11]
	\arrow[""{name=10, anchor=center, inner sep=0}, "QX"'{inner sep=.8ex}, "\shortmid"{marking}, from=5-1, to=5-5]
	\arrow[""{name=11, anchor=center, inner sep=0}, "QX"'{inner sep=.8ex}, "\shortmid"{marking}, from=5-7, to=5-11]
	\arrow["{\lambda_{QX}}"{description}, draw=none, color={rgb,255:red,92;green,92;blue,214}, between={0.2}{0.8}, Rightarrow, from=2-3, to=0]
	\arrow["{{\rho _{QX}}}"{description}, draw=none, color={rgb,255:red,92;green,92;blue,214}, between={0.2}{0.8}, Rightarrow, from=2-9, to=1]
	\arrow["{\iota^Q_{X_1}}"{description}, draw=none, color={rgb,255:red,214;green,92;blue,92}, between={0.2}{0.8}, Rightarrow, from=6, to=2]
	\arrow["{{{Q(v_X)}}}"'{description}, draw=none, between={0.2}{0.8}, Rightarrow, from=7, to=3]
	\arrow["{{{v_{QX}}}}"'{description}, draw=none, between={0.2}{0.8}, Rightarrow, from=8, to=4]
	\arrow["{\iota^Q_{X_2}}"'{description}, draw=none, color={rgb,255:red,214;green,92;blue,92}, between={0.2}{0.8}, Rightarrow, from=9, to=5]
	\arrow["{{\lambda ^{-1}_{QX}}}"'{description}, draw=none, color={rgb,255:red,92;green,92;blue,214}, between={0.2}{1}, Rightarrow, from=10, to=4-3]
	\arrow["{{\rho ^{-1}_{QX}}}"'{description}, draw=none, color={rgb,255:red,92;green,92;blue,214}, between={0.2}{1}, Rightarrow, from=11, to=4-9]
\end{tikzcd};} \end{center}

\item{}[Compositional coherence]: For every composable pair of cells \(\alpha , \beta  \in  \mathbb {A}_1\), 
\begin{center}  \adjustbox{max width=\textwidth}{\begin{tikzcd}
	{QZ_1} &&&& {QZ_2} && {QZ_1} &&&& {QZ_2} \\
	{QZ_1} && {QZ_1} && {QZ_2} && {QZ_1} && {QZ_2} && {QZ_2} \\
	{QY_1} &&&&& {=} & {QY_1} && {QY_2} \\
	{QX_1} && {QX_1} && {QX_2} && {QX_1} && {QX_2} && {QX_2} \\
	{QX_1} &&&& {QX_2} && {QX_1} &&&& {QX_2}
	\arrow[""{name=0, anchor=center, inner sep=0}, "QZ"{inner sep=.8ex}, "\shortmid"{marking}, from=1-1, to=1-5]
	\arrow[equals, from=1-1, to=2-1]
	\arrow[equals, from=1-5, to=2-5]
	\arrow[""{name=1, anchor=center, inner sep=0}, "QZ"{inner sep=.8ex}, "\shortmid"{marking}, from=1-7, to=1-11]
	\arrow[no head, from=1-11, to=2-11]
	\arrow[""{name=2, anchor=center, inner sep=0}, "\shortmid"{marking}, equals, from=2-1, to=2-3]
	\arrow["{{Qg_1}}"', from=2-1, to=3-1]
	\arrow[""{name=3, anchor=center, inner sep=0}, "\shortmid"{marking}, from=2-3, to=2-5]
	\arrow["{{Q(f_1 g_1)}}"', from=2-3, to=4-3]
	\arrow["{{Q(f_2g_2)}}", from=2-5, to=4-5]
	\arrow[equals, from=2-7, to=1-7]
	\arrow[""{name=4, anchor=center, inner sep=0}, "QZ"{inner sep=.8ex}, "\shortmid"{marking}, from=2-7, to=2-9]
	\arrow["{{Qg_1}}"', from=2-7, to=3-7]
	\arrow[draw=none, from=2-9, to=2-11]
	\arrow[""{name=5, anchor=center, inner sep=0}, "\shortmid"{marking}, equals, from=2-9, to=2-11]
	\arrow["{{Qg_2}}", from=2-9, to=3-9]
	\arrow["{{Q(f_2g_2)}}", from=2-11, to=4-11]
	\arrow["{{Qf_1}}"', from=3-1, to=4-1]
	\arrow[""{name=6, anchor=center, inner sep=0}, "{QY }"{inner sep=.8ex}, "\shortmid"{marking}, from=3-7, to=3-9]
	\arrow["{{Qf_1}}"', from=3-7, to=4-7]
	\arrow["{{Qf_2}}", from=3-9, to=4-9]
	\arrow[""{name=7, anchor=center, inner sep=0}, "\shortmid"{marking}, equals, from=4-1, to=4-3]
	\arrow[""{name=8, anchor=center, inner sep=0}, "\shortmid"{marking}, from=4-3, to=4-5]
	\arrow[equals, from=4-5, to=5-5]
	\arrow[""{name=9, anchor=center, inner sep=0}, "\shortmid"{marking}, from=4-7, to=4-9]
	\arrow[""{name=10, anchor=center, inner sep=0}, "\shortmid"{marking}, equals, from=4-9, to=4-11]
	\arrow[equals, from=4-11, to=5-11]
	\arrow[equals, from=5-1, to=4-1]
	\arrow[""{name=11, anchor=center, inner sep=0}, "QX"'{inner sep=.8ex}, "\shortmid"{marking}, from=5-1, to=5-5]
	\arrow[equals, from=5-7, to=4-7]
	\arrow[""{name=12, anchor=center, inner sep=0}, "QX"'{inner sep=.8ex}, "\shortmid"{marking}, from=5-7, to=5-11]
	\arrow["{{\lambda _{QZ}}}"{description}, color={rgb,255:red,92;green,92;blue,214}, draw=none, from=2-3, to=0]
	\arrow["{{\rho _{QZ}}}"{description}, color={rgb,255:red,92;green,92;blue,214}, draw=none, from=2-9, to=1]
	\arrow["{{Q\beta }}"{description}, draw=none, from=6, to=4]
	\arrow["{{\gamma^Q_{f_1,g_1}}}"{description}, color={rgb,255:red,214;green,92;blue,92}, draw=none, from=7, to=2]
	\arrow["{{Q(\frac {\beta }{\alpha })}}"{description}, draw=none, from=8, to=3]
	\arrow["{{Q\alpha }}"{description}, draw=none, from=9, to=6]
	\arrow["{{\gamma^Q_{f_2, g_2}}}"{description}, color={rgb,255:red,214;green,92;blue,92}, draw=none, from=10, to=5]
	\arrow["{{\lambda ^{-1}_{QX}}}"{description}, color={rgb,255:red,92;green,92;blue,214}, draw=none, from=11, to=4-3]
	\arrow["{{\rho ^{-1}_{QX}}}"{description}, color={rgb,255:red,92;green,92;blue,214}, draw=none, from=12, to=4-9]
\end{tikzcd};}\end{center}

\item{}[Double associativity]: For all composable \(X_1 \overset {X}{\mathrel {\mkern 3mu\vcenter {\hbox {$\shortmid $}}\mkern -10mu{\to }}} X_2 \overset {X'}{\mathrel {\mkern 3mu\vcenter {\hbox {$\shortmid $}}\mkern -10mu{\to }}} X_3 \overset {X''}{\mathrel {\mkern 3mu\vcenter {\hbox {$\shortmid $}}\mkern -10mu{\to }}} X_4\),  
\begin{center}\adjustbox{max width=\textwidth}{\begin{tikzcd}
	{QX_1} && {QX_2} && {QX_3} && {QX_4} && {QX_1} && {QX_2} && {QX_3} && {QX_4} \\
	{QX_1} &&&& {QX_3} && {QX_4} && {QX_1} && {QX_2} &&&& {QX_4} \\
	&&&&&&& {=} & {QX_1} &&&&&& {QX_4} \\
	{QX_1} &&&&&& {QX_4} && {QX_1} &&&&&& {QX_4}
	\arrow["QX"{inner sep=.8ex}, "\shortmid"{marking}, from=1-1, to=1-3]
	\arrow[equals, from=1-1, to=2-1]
	\arrow["{QX'}"{inner sep=.8ex}, "\shortmid"{marking}, from=1-3, to=1-5]
	\arrow[""{name=0, anchor=center, inner sep=0}, "{QX''}"{inner sep=.8ex}, "\shortmid"{marking}, from=1-5, to=1-7]
	\arrow[equals, from=1-5, to=2-5]
	\arrow[equals, from=1-7, to=2-7]
	\arrow[""{name=1, anchor=center, inner sep=0}, "QX"{inner sep=.8ex}, "\shortmid"{marking}, from=1-9, to=1-11]
	\arrow[equals, from=1-9, to=2-9]
	\arrow["{QX'}"{inner sep=.8ex}, "\shortmid"{marking}, from=1-11, to=1-13]
	\arrow[equals, from=1-11, to=2-11]
	\arrow["{QX''}"{inner sep=.8ex}, "\shortmid"{marking}, from=1-13, to=1-15]
	\arrow[equals, from=1-15, to=2-15]
	\arrow[""{name=2, anchor=center, inner sep=0}, "{Q(X \odot X')}"'{inner sep=.8ex}, "\shortmid"{marking}, from=2-1, to=2-5]
	\arrow[""{name=3, anchor=center, inner sep=0}, equals, from=2-1, to=4-1]
	\arrow[""{name=4, anchor=center, inner sep=0}, "{QX''}"'{inner sep=.8ex}, "\shortmid"{marking}, from=2-5, to=2-7]
	\arrow[""{name=5, anchor=center, inner sep=0}, equals, from=2-7, to=4-7]
	\arrow[""{name=6, anchor=center, inner sep=0}, "QX"'{inner sep=.8ex}, "\shortmid"{marking}, from=2-9, to=2-11]
	\arrow[""{name=7, anchor=center, inner sep=0}, equals, from=2-9, to=3-9]
	\arrow[""{name=8, anchor=center, inner sep=0}, "{Q(X' \odot X'')}"'{inner sep=.8ex}, "\shortmid"{marking}, from=2-11, to=2-15]
	\arrow[""{name=9, anchor=center, inner sep=0}, equals, from=2-15, to=3-15]
	\arrow[""{name=10, anchor=center, inner sep=0}, "{Q(X \odot (X' \odot X''))}"'{inner sep=.8ex}, "\shortmid"{marking}, from=3-9, to=3-15]
	\arrow[equals, from=3-9, to=4-9]
	\arrow[equals, from=3-15, to=4-15]
	\arrow["{Q((X\odot X') \odot X'')}"'{inner sep=.8ex}, "\shortmid"{marking}, from=4-1, to=4-7]
	\arrow[""{name=11, anchor=center, inner sep=0}, "{Q((X\odot X') \odot X'')}"'{inner sep=.8ex}, "\shortmid"{marking}, from=4-9, to=4-15]
	\arrow["{{(Q_\odot )_{X, X'}}}"{description, pos=0.6}, color={rgb,255:red,92;green,92;blue,214}, draw=none, from=2, to=1-3]
	\arrow["{{(Q_{\odot })_{X\odot  X', X''}}}"{description}, color={rgb,255:red,92;green,92;blue,214}, draw=none, from=3, to=5]
	\arrow["{{v_{QX''}}}"{description}, draw=none, from=4, to=0]
	\arrow["{{v_{QX}}}"{description}, draw=none, from=6, to=1]
	\arrow["{{(Q_{\odot })_{X, X' \odot  X''}}}"{description}, color={rgb,255:red,92;green,92;blue,214}, draw=none, from=7, to=9]
	\arrow["{{(Q_{\odot })_{X', X''}}}"{description}, color={rgb,255:red,92;green,92;blue,214}, draw=none, from=8, to=1-13]
	\arrow["{{Q(\alpha _{X, X', X''})}}"{description}, draw=none, from=11, to=10]
\end{tikzcd};}\end{center}

\item{}[Double unitality]: For all loose arrows \(X_1 \overset {X}{\mathrel {\mkern 3mu\vcenter {\hbox {$\shortmid $}}\mkern -10mu{\to }}} X_1\) in \(\mathbb {A}\), 
\begin{center}\adjustbox{max width=\textwidth}{\begin {tikzcd}[cramped]
	{QX_1} &&&& {QX_2} \\
	{QX_1} && {QX_1} && {QX_2} && {QX_1} &&&& {QX_2} \\
	&&&&& {=} \\
	{QX_1} && {QX_1} && {QX_2} && {QX_1} &&&& {QX_2} \\
	\\
	{QX_1} &&&& {QX_2} && {}
	\arrow[""{name=0, anchor=center, inner sep=0}, "QX", "\shortmid "{marking}, from=1-1, to=1-5]
	\arrow[equals, from=2-1, to=1-1]
	\arrow[""{name=1, anchor=center, inner sep=0}, "\shortmid "{marking}, equals, from=2-1, to=2-3]
	\arrow[equals, from=2-1, to=4-1]
	\arrow[""{name=2, anchor=center, inner sep=0}, "QX", "\shortmid "{marking}, from=2-3, to=2-5]
	\arrow[equals, from=2-3, to=4-3]
	\arrow[equals, from=2-5, to=1-5]
	\arrow[equals, from=2-5, to=4-5]
	\arrow[""{name=3, anchor=center, inner sep=0}, "QX", "\shortmid "{marking}, from=2-7, to=2-11]
	\arrow[equals, from=2-7, to=4-7]
	\arrow[equals, from=2-11, to=4-11]
	\arrow[""{name=4, anchor=center, inner sep=0}, "{Q(e_{X_1})}"', "\shortmid "{marking}, from=4-1, to=4-3]
	\arrow[equals, from=4-1, to=6-1]
	\arrow[""{name=5, anchor=center, inner sep=0}, "QX"', "\shortmid "{marking}, from=4-3, to=4-5]
	\arrow[equals, from=4-5, to=6-5]
	\arrow[""{name=6, anchor=center, inner sep=0}, "{Q(e_{X_1} \odot  X)}"', "\shortmid "{marking}, from=4-7, to=4-11]
	\arrow[""{name=7, anchor=center, inner sep=0}, "{Q(e_{X_1} \odot  X)}"', "\shortmid "{marking}, from=6-1, to=6-5]
	\arrow["{\lambda _{QX}}"'{description}, draw=none, color={rgb,255:red,92;green,92;blue,214}, shorten >=3pt, Rightarrow, from=2-3, to=0]
	\arrow["{(Q_e)_X}"{description}, draw=none, color={rgb,255:red,92;green,92;blue,214}, shorten <=9pt, shorten >=9pt, Rightarrow, from=4, to=1]
	\arrow["{v_{QX}}"'{description}, draw=none, shorten <=9pt, shorten >=9pt, Rightarrow, from=5, to=2]
	\arrow["{Q(\lambda _X)}"'{description}, draw=none, shorten <=9pt, shorten >=9pt, Rightarrow, from=6, to=3]
	\arrow["{(Q_{\odot })_{e_{X_1}, X}}"'{description}, draw=none, color={rgb,255:red,92;green,92;blue,214}, shorten <=7pt, Rightarrow, from=7, to=4-3]
\end {tikzcd}}\end{center}

\begin{center}\adjustbox{max width=\textwidth}{\begin {tikzcd}[cramped]
	{QX_1} &&&& {QX_2} \\
	{QX_1} && {QX_2} && {QX_2} && {QX_1} &&& {QX_2} \\
	&&&&& {=} \\
	{QX_1} && {QX_2} && {QX_2} && {QX_1} &&& {QX_2} \\
	\\
	{QX_1} &&&& {QX_2} \\
	\arrow[""{name=0, anchor=center, inner sep=0}, "QX", "\shortmid "{marking}, from=1-1, to=1-5]
	\arrow[equals, from=2-1, to=1-1]
	\arrow[""{name=1, anchor=center, inner sep=0}, "QX", "\shortmid "{marking}, from=2-1, to=2-3]
	\arrow[equals, from=2-1, to=4-1]
	\arrow[""{name=2, anchor=center, inner sep=0}, "\shortmid "{marking}, equals, from=2-3, to=2-5]
	\arrow[equals, from=2-3, to=4-3]
	\arrow[equals, from=2-5, to=1-5]
	\arrow[equals, from=2-5, to=4-5]
	\arrow[""{name=3, anchor=center, inner sep=0}, "QX", "\shortmid "{marking}, from=2-7, to=2-10]
	\arrow[equals, from=2-7, to=4-7]
	\arrow[equals, from=2-10, to=4-10]
	\arrow[""{name=4, anchor=center, inner sep=0}, "QX"', "\shortmid "{marking}, from=4-1, to=4-3]
	\arrow[equals, from=4-1, to=6-1]
	\arrow[""{name=5, anchor=center, inner sep=0}, "{Q(e_{X_2})}"', "\shortmid "{marking}, from=4-3, to=4-5]
	\arrow[equals, from=4-5, to=6-5]
	\arrow[""{name=6, anchor=center, inner sep=0}, "{Q(X \odot  e_{X_2})}"', "\shortmid "{marking}, from=4-7, to=4-10]
	\arrow[""{name=7, anchor=center, inner sep=0}, "{Q(X \odot  e_{X_2})}"', "\shortmid "{marking}, from=6-1, to=6-5]
	\arrow["{\rho _{QX}}"'{description}, draw=none, color={rgb,255:red,92;green,92;blue,214}, shorten >=3pt, Rightarrow, from=2-3, to=0]
	\arrow["{v_{QX}}"'{description}, draw=none, shorten <=9pt, shorten >=9pt, Rightarrow, from=4, to=1]
	\arrow["{(Q_e)_{X_2}}"'{description}, draw=none, color={rgb,255:red,92;green,92;blue,214}, shorten <=9pt, shorten >=9pt, Rightarrow, from=5, to=2]
	\arrow["{Q(\rho _{X})}"'{description}, draw=none, shorten <=9pt, shorten >=9pt, Rightarrow, from=6, to=3]
	\arrow["{(Q_\odot )_{X, e_{X_2}}}"'{description}, draw=none, color={rgb,255:red,92;green,92;blue,214}, shorten <=7pt, Rightarrow, from=7, to=4-3]
\end {tikzcd}.}\end{center}
 \end{itemize}

 Moreover, $Q_e$ and $Q_{\odot}$ are natural in the sense that for any tight arrow $f \colon X \to Y$ and horizontally composable squares 
 \begin{equation*}
      \begin{tikzcd}
	{X_1} & {X_2} & {X_3} \\
	{Y_1} & {Y_2} & {Y_3}
	\arrow[""{name=0, anchor=center, inner sep=0}, "X"{inner sep=.8ex}, "\shortmid"{marking}, from=1-1, to=1-2]
	\arrow["{f_1}"', from=1-1, to=2-1]
	\arrow[""{name=1, anchor=center, inner sep=0}, "{X'}"{inner sep=.8ex}, "\shortmid"{marking}, from=1-2, to=1-3]
	\arrow["{f_2}", from=1-2, to=2-2]
	\arrow["{f_3}", from=1-3, to=2-3]
	\arrow[""{name=2, anchor=center, inner sep=0}, "Y"'{inner sep=.8ex}, "\shortmid"{marking}, from=2-1, to=2-2]
	\arrow[""{name=3, anchor=center, inner sep=0}, "{Y'}"'{inner sep=.8ex}, "\shortmid"{marking}, from=2-2, to=2-3]
	\arrow["\beta", shift right=2, draw=none, from=1, to=3]
	\arrow["\alpha"', shift left=2, draw=none, from=2, to=0]
\end{tikzcd}
 \end{equation*}
\noindent in $\mathbb{A}$, we have $\frac{(Q_e)_X}{Q(e_f)} = \frac{e_{Qf}}{(Q_e)_Y}$ and $\frac{Q(\alpha\mid \beta)}{(Q_{\odot})_{Y,Y'}} = \frac{(Q_{\odot})_{X, X'}}{Q\alpha\ \mid\ Q\beta}$.\\
 
A \emph{pseudo double functor} $Q$ (note the order of the words) is a double pseudofunctor for which all squares $\iota^Q_A$ and $\gamma^Q_{f, g}$ are horizontal identities, i.e., $Q_0$ is a strict functor $\mathbb{A}_0 \to \mathbb{B}_0$. If it also strictly preserves loose identities, we say it is a \emph{normal double functor}.\end{definition}

If $Q$ and $Q'$ are pseudo double functors, the notion of map between them that we will use is that of a \emph{tight transformation}. Once again, we describe a slightly more general notion in case the reader is interested in double pseudofunctors instead:
\begin{definition}[{Tight pseudotransformation, c.f. \cite{grandis-1999-limits}}]
\par{}A tight pseudotransformation \(\phi  \colon  F \Rightarrow  G\) between double pseudofunctors \(F, G \colon  \mathbb {A} \to  \mathbb {B}\) consists of:

\begin{itemize}\item{}A \emph{tight component} \(\phi _A \colon  FA \to  GA\) for each object \(A \in  \mathbb {A}\); 

\item{}A \emph{tight naturality square} 
\begin{center}\begin{tikzcd}
	FA && FA \\
	FB && GA \\
	GB && GB
	\arrow[""{name=0, anchor=center, inner sep=0}, "{ }"{marking, allow upside down}, "\shortmid"{marking}, equals, from=1-1, to=1-3]
	\arrow["Ff"', from=1-1, to=2-1]
	\arrow["{{\phi _A}}", from=1-3, to=2-3]
	\arrow["{{\phi _B}}"', from=2-1, to=3-1]
	\arrow["Gf", from=2-3, to=3-3]
	\arrow[""{name=1, anchor=center, inner sep=0}, "\shortmid"{marking}, equals, from=3-1, to=3-3]
	\arrow["{{\phi _f}}"{description}, color={rgb,255:red,92;green,92;blue,214}, draw=none, from=1, to=0]
\end{tikzcd},\end{center}
 which is an isomorphism in \(\mathbb{B}_1\) for each tight arrow \(f\colon  A \to  B\) in \(\mathbb {A}\), and 

\item{}A \emph{cell component} 
\begin{center}\begin {tikzcd}[cramped]
	{FX_1} & {FX_2} \\
	{GX_1} & {GX_2}
	\arrow [""{name=0, anchor=center, inner sep=0}, "FX", "\shortmid "{marking}, from=1-1, to=1-2]
	\arrow ["{\phi _{X_1}}"', from=1-1, to=2-1]
	\arrow ["{\phi _{X_2}}", from=1-2, to=2-2]
	\arrow [""{name=1, anchor=center, inner sep=0}, "GX"', "\shortmid "{marking}, from=2-1, to=2-2]
	\arrow ["{\phi _X}"{description}, draw=none, {text={rgb,255:red,92;green,92;blue,214}}, color={rgb,255:red,92;green,92;blue,214}, shorten <=4pt, shorten >=4pt, Rightarrow, from=1, to=0]
\end {tikzcd}\end{center}
 in \(\mathbb {B}\) for each loose map \(X_1 \overset {X}{\mathrel {\mkern 3mu\vcenter {\hbox {$\shortmid $}}\mkern -10mu{\to }}} X_2\) in \(\mathbb {A}\),\end{itemize} 

such that:

\begin{itemize}\item{}[Tight naturality]: The tight naturality squares are subject to the axioms of a pseudonatural transformation \(F_0 \to  G_0\),

\item{}[Unital coherence]: For all objects \(A \in  \mathbb {A}\), 
\begin{center}\adjustbox{max width=\textwidth}{\begin {tikzcd}[cramped]
	FA && FA && FA && FA && FA && FA \\
	FA && FA && GA & {=} & FA && GA && GA \\
	GA && GA && GA && GA && GA && GA
	\arrow [""{name=0, anchor=center, inner sep=0}, "\shortmid "{marking}, equals, from=1-1, to=1-3]
	\arrow [equals, from=1-1, to=2-1]
	\arrow [""{name=1, anchor=center, inner sep=0}, "\shortmid "{marking}, equals, from=1-3, to=1-5]
	\arrow [equals, from=1-3, to=2-3]
	\arrow ["{\phi _A}", from=1-5, to=2-5]
	\arrow [""{name=2, anchor=center, inner sep=0}, "\shortmid "{marking}, equals, from=1-7, to=1-9]
	\arrow [equals, from=1-7, to=2-7]
	\arrow [""{name=3, anchor=center, inner sep=0}, "\shortmid "{marking}, equals, from=1-9, to=1-11]
	\arrow ["{\phi _A}"', from=1-9, to=2-9]
	\arrow ["{\phi _A}", from=1-11, to=2-11]
	\arrow [""{name=4, anchor=center, inner sep=0}, "{F(e_A)}"', "\shortmid "{marking}, from=2-1, to=2-3]
	\arrow ["{\phi _A}"', from=2-1, to=3-1]
	\arrow ["{\phi _A}", from=2-3, to=3-3]
	\arrow [equals, from=2-5, to=3-5]
	\arrow ["{\phi _A}"', from=2-7, to=3-7]
	\arrow [""{name=5, anchor=center, inner sep=0}, "\shortmid "{marking}, equals, from=2-9, to=2-11]
	\arrow [equals, from=2-9, to=3-9]
	\arrow [equals, from=2-11, to=3-11]
	\arrow [""{name=6, anchor=center, inner sep=0}, "{G(e_A)}"', "\shortmid "{marking}, from=3-1, to=3-3]
	\arrow [""{name=7, anchor=center, inner sep=0}, "\shortmid "{marking}, equals, from=3-3, to=3-5]
	\arrow [""{name=8, anchor=center, inner sep=0}, "\shortmid "{marking}, equals, from=3-7, to=3-9]
	\arrow [""{name=9, anchor=center, inner sep=0}, "{G(e_A)}"', "\shortmid "{marking}, from=3-9, to=3-11]
	\arrow ["{(F_e)_A}"'{description}, draw=none, color={rgb,255:red,214;green,92;blue,92}, shorten <=4pt, shorten >=4pt, Rightarrow, from=4, to=0]
	\arrow ["{e_{\phi _A}}"', color={rgb,255:red,92;green,92;blue,214}, shorten <=4pt, shorten >=4pt, Rightarrow, from=5, to=3]
	\arrow ["{\phi _{e_A}}"'{description}, draw=none, color={rgb,255:red,92;green,92;blue,214}, shorten <=4pt, shorten >=9pt, Rightarrow, from=6, to=4]
	\arrow ["{\phi _{\mathrm {id}_A}}"'{description}, draw=none, color={rgb,255:red,92;green,92;blue,214}, shorten <=9pt, shorten >=9pt, Rightarrow, from=7, to=1]
	\arrow ["{\phi _{\mathrm {id}_A}}"'{description}, draw=none, color={rgb,255:red,92;green,92;blue,214}, shorten <=9pt, shorten >=9pt, Rightarrow, from=8, to=2]
	\arrow ["{(G_e)_A}"'{description}, draw=none, color={rgb,255:red,214;green,92;blue,92}, shorten <=4pt, shorten >=4pt, Rightarrow, from=9, to=5]
\end {tikzcd};}\end{center}

\item{}[Compositional coherence]: For all composable loose arrows \(X_1 \overset {X}{\mathrel {\mkern 3mu\vcenter {\hbox {$\shortmid $}}\mkern -10mu{\to }}} X_2 \overset {X'}{\mathrel {\mkern 3mu\vcenter {\hbox {$\shortmid $}}\mkern -10mu{\to }}} X_3\) in \(\mathbb {A}\), \\

\begin{center}
	\adjustbox{max width=\textwidth}{\begin{tikzcd}
	{FX_1} && {FX_1} & {FX_2} & {FX_3} && {FX_1} & {FX_2} & {FX_3} && {FX_3} \\
	{GX_1} && {FX_1} && FX3 & {=} & {GX_1} & {GX_2} & {GX_3} && {FX_3} \\
	{GX_1} && {GX_1} && {GX_3} && {GX_1} && {GX_3} && {GX_3}
	\arrow[""{name=0, anchor=center, inner sep=0}, "\shortmid"{marking}, equals, from=1-1, to=1-3]
	\arrow["{{\phi _{X_1}}}"', from=1-1, to=2-1]
	\arrow["FX"{inner sep=.8ex}, "\shortmid"{marking}, from=1-3, to=1-4]
	\arrow[equals, from=1-3, to=2-3]
	\arrow["{FX'}"{inner sep=.8ex}, "\shortmid"{marking}, from=1-4, to=1-5]
	\arrow[equals, from=1-5, to=2-5]
	\arrow[""{name=1, anchor=center, inner sep=0}, "{FX }"{inner sep=.8ex}, "\shortmid"{marking}, from=1-7, to=1-8]
	\arrow["{{\phi _{X_1}}}"', from=1-7, to=2-7]
	\arrow[""{name=2, anchor=center, inner sep=0}, "{FX'}"{inner sep=.8ex}, "\shortmid"{marking}, from=1-8, to=1-9]
	\arrow["{{\phi _{X_2}}}"', from=1-8, to=2-8]
	\arrow[""{name=3, anchor=center, inner sep=0}, "\shortmid"{marking}, equals, from=1-9, to=1-11]
	\arrow["{{\phi _{X_3}}}", from=1-9, to=2-9]
	\arrow[equals, from=1-11, to=2-11]
	\arrow[equals, from=2-1, to=3-1]
	\arrow[""{name=4, anchor=center, inner sep=0}, "{F(X \odot X')}"{inner sep=.8ex}, "\shortmid"{marking}, from=2-3, to=2-5]
	\arrow["{{\phi _{X_1}}}"', from=2-3, to=3-3]
	\arrow["{{\phi _{X_3}}}", from=2-5, to=3-5]
	\arrow[""{name=5, anchor=center, inner sep=0}, "GX"'{inner sep=.8ex}, "\shortmid"{marking}, from=2-7, to=2-8]
	\arrow[equals, from=2-7, to=3-7]
	\arrow[""{name=6, anchor=center, inner sep=0}, "{GX'}"'{inner sep=.8ex}, "\shortmid"{marking}, from=2-8, to=2-9]
	\arrow[equals, from=2-9, to=3-9]
	\arrow["{{\phi _{X_3}}}", from=2-11, to=3-11]
	\arrow[""{name=7, anchor=center, inner sep=0}, "\shortmid"{marking}, equals, from=3-1, to=3-3]
	\arrow[""{name=8, anchor=center, inner sep=0}, "{G(X \odot X')}"'{inner sep=.8ex}, "\shortmid"{marking}, from=3-3, to=3-5]
	\arrow[""{name=9, anchor=center, inner sep=0}, "{G(X \odot X')}"'{inner sep=.8ex}, "\shortmid"{marking}, from=3-7, to=3-9]
	\arrow[""{name=10, anchor=center, inner sep=0}, "\shortmid"{marking}, equals, from=3-9, to=3-11]
	\arrow["{{(F_{\odot })_{X, X'}}}"{pos=0.8}{description}, draw=none, color={rgb,255:red,214;green,92;blue,92}, between={0.5}{1}, Rightarrow, from=4, to=1-4]
	\arrow["{{\phi _X}}"{description}, draw=none, color={rgb,255:red,92;green,92;blue,214}, between={0.2}{0.8}, Rightarrow, from=5, to=1]
	\arrow["{{\phi _{X'}}}"'{description}, draw=none, color={rgb,255:red,92;green,92;blue,214}, between={0.2}{0.8}, Rightarrow, from=6, to=2]
	\arrow["{{\phi _{\mathrm {id}_{X_1}}}}"{description}, draw=none, color={rgb,255:red,92;green,92;blue,214}, between={0.2}{0.8}, Rightarrow, from=7, to=0]
	\arrow["{{\phi _{X\odot  X'}}}"{description}, draw=none, color={rgb,255:red,92;green,92;blue,214}, between={0.2}{0.8}, Rightarrow, from=8, to=4]
	\arrow["{{(G_{\odot })_{X, X'}}}"'{description}, draw=none, color={rgb,255:red,214;green,92;blue,92}, between={0.2}{1}, Rightarrow, from=9, to=2-8]
	\arrow["{{\phi _{\mathrm {id}_{X_3}}}}"'{description}, draw=none, color={rgb,255:red,92;green,92;blue,214}, between={0.2}{0.8}, Rightarrow, from=10, to=3]
\end{tikzcd};} \end{center}

\item{}[Cell naturality]: For every square 
$\begin {tikzcd}[cramped]
	{X_1} && {X_2} \\
	{Y_1} && {Y_2}
	\arrow [""{name=0, anchor=center, inner sep=0}, "X", "\shortmid "{marking}, from=1-1, to=1-3]
	\arrow ["{f_1}"', from=1-1, to=2-1]
	\arrow ["{f_2}", from=1-3, to=2-3]
	\arrow [""{name=1, anchor=center, inner sep=0}, "Y"', "\shortmid "{marking}, from=2-1, to=2-3]
	\arrow ["\alpha "{description}, draw=none, from=0, to=1]
\end {tikzcd}$
 in \(\mathbb {A}\), 
\begin{center}\adjustbox{max width=\textwidth}{\begin {tikzcd}[cramped]
	{FX_1} && {FX_2} && {FX_2} && {FX_1} && {FX_1} && {FX_2} \\
	{FY_1} && {FY_2} && {GX_2} & {=} & {FY_1} && {GX_1} && {GX_2} \\
	{GY_1} && {GY_2} && {GY_2} && {GY_1} && {GY_1} && {GY_2}
	\arrow [""{name=0, anchor=center, inner sep=0}, "FX", "\shortmid "{marking}, from=1-1, to=1-3]
	\arrow ["{Ff_1}"', from=1-1, to=2-1]
	\arrow [""{name=1, anchor=center, inner sep=0}, "\shortmid "{marking}, equals, from=1-3, to=1-5]
	\arrow ["{Ff_2}", from=1-3, to=2-3]
	\arrow ["{\phi _{X_2}}", from=1-5, to=2-5]
	\arrow [""{name=2, anchor=center, inner sep=0}, "\shortmid "{marking}, equals, from=1-7, to=1-9]
	\arrow ["{Ff_1}"', from=1-7, to=2-7]
	\arrow [""{name=3, anchor=center, inner sep=0}, "FX", "\shortmid "{marking}, from=1-9, to=1-11]
	\arrow ["{\phi _{X_1}}"', from=1-9, to=2-9]
	\arrow ["{\phi _{X_2}}", from=1-11, to=2-11]
	\arrow [""{name=4, anchor=center, inner sep=0}, "FY"', "\shortmid "{marking}, from=2-1, to=2-3]
	\arrow ["{\phi _{Y_1}}"', from=2-1, to=3-1]
	\arrow ["{\phi _{Y_2}}", from=2-3, to=3-3]
	\arrow ["{Gf_2}", from=2-5, to=3-5]
	\arrow ["{\phi _{Y_1}}"', from=2-7, to=3-7]
	\arrow [""{name=5, anchor=center, inner sep=0}, "GX"', "\shortmid "{marking}, from=2-9, to=2-11]
	\arrow ["{Gf_1}"', from=2-9, to=3-9]
	\arrow ["{Gf_2}", from=2-11, to=3-11]
	\arrow [""{name=6, anchor=center, inner sep=0}, "GY"', "\shortmid "{marking}, from=3-1, to=3-3]
	\arrow [""{name=7, anchor=center, inner sep=0}, "\shortmid "{marking}, equals, from=3-3, to=3-5]
	\arrow [""{name=8, anchor=center, inner sep=0}, "\shortmid "{marking}, equals, from=3-7, to=3-9]
	\arrow [""{name=9, anchor=center, inner sep=0}, "GY"', "\shortmid "{marking}, from=3-9, to=3-11]
	\arrow ["{F\alpha }"{description}, draw=none, color={rgb,255:red,214;green,92;blue,92}, shorten <=4pt, shorten >=4pt, Rightarrow, from=4, to=0]
	\arrow ["{\phi _X}"'{description}, draw=none, color={rgb,255:red,92;green,92;blue,214}, shorten <=4pt, shorten >=4pt, Rightarrow, from=5, to=3]
	\arrow ["{\phi _Y}"{description}, draw=none, color={rgb,255:red,92;green,92;blue,214}, shorten <=4pt, shorten >=6pt, Rightarrow, from=6, to=4]
	\arrow ["{\phi _{f_2}}"{description}, draw=none, color={rgb,255:red,92;green,92;blue,214}, shorten <=9pt, shorten >=9pt, Rightarrow, from=7, to=1]
	\arrow ["{\phi _{f_1}}"'{description}, draw=none, color={rgb,255:red,92;green,92;blue,214}, shorten <=9pt, shorten >=9pt, Rightarrow, from=8, to=2]
	\arrow ["{G\alpha }"'{description}, draw=none, color={rgb,255:red,214;green,92;blue,92}, shorten <=4pt, shorten >=6pt, Rightarrow, from=9, to=5]
\end {tikzcd}.}\end{center}\end{itemize}

A \emph{tight transformation} is a tight pseudotransformation whose tight naturality squares form a strict 2-natural transformation. 
\end{definition}

Weak double categories, pseudo double functors, and tight transformations form a cartesian 2-category $\Dbl$, and so we can talk about pseudomonoids in it.

\begin{definition}[{Monoidal double categories, c.f.  \cite{ShulmanBicategories}}]

A monoidal double category \(\mathbb {D}\) is a pseudomonoid in the 2-category $\Dbl$.\end{definition}

This amounts to \(\mathbb {D}_0\) and \(\mathbb {D}_1\) being monoidal categories, the source and target functors \(\mathbb {D}_1 \to  \mathbb {D}_0\) being strict monoidal,
and both the identity assignment map \(e \colon  \mathbb {D}_0 \to  \mathbb {D}_1\) and internal composition \(\odot  \colon  \mathbb {D}_1 \times  \mathbb {D}_1 \to  \mathbb {D}_1\)
preserving the monoidal unit and tensor product. Moreover, the unit and associativity isomorphisms for \(\otimes \) form tight transformations. Similarly, symmetric pseudomonoids in $\Dbl$ are \emph{symmetric monoidal double categories}.\\

Accordingly, by a lax symmetric monoidal pseudo double functor \(Q \colon  \mathbb {C} \to  \mathbb {D}\), we mean a lax morphism of symmetric pseudomonoids in $\Dbl$: a pseudo double functor
 \(Q\) equipped with tight transformations \(\mu  \colon  Q \otimes  Q \to  Q(-\otimes  -)\) and  \(I_{\mathbb {D}} \to  Q(I_{\mathbb {C}})\) 
between monoidal units satisfying the usual axioms for a symmetric monoidal functor.\\

There are two important types of double categories for the purposes of this paper:

\begin{example}[{Spans}]
\par{}Given an adequate triple $\mathfrak{C}$, we can construct its \emph{double category of spans} \(\mathbb {S}\mathsf {pan}(\mathfrak{C})\): its
underlying tight category is just \(\mathsf {C}\), loose arrows \(X_1 \overset {X}{\mathrel {\mkern 3mu\vcenter {\hbox {$\shortmid $}}\mkern -10mu{\to }}} X_2\) are $\mathfrak{C}$-spans \(X_1 \xleftarrow {x_1} X \xrightarrow {x_2} X_2\) (which
compose by pullback), and squares are morphisms of spans (see \cite{dawson-2010-span}). Its opposite double category will play a central role in
this paper. The span construction can be extended to a cartesian 2-functor $\mathcal{A}{dq} \to \Dbl$ (c.f. Lemma 2.14 in \cite{libkind-2025-towards}), and therefore if $\mathfrak{C}$ is a symmetric monoidal adequate triple, then $\mathbb{S}\mathsf{pan}(\mathfrak{C})$ is  a symmetric monoidal double category. A characterisation for spans on trivial adequate triples was provided in \cite{Evangelia}, which also discusses the relationship between spans and the Beck-Chevalley condition.
\end{example}

\begin{remark}
   Composition of $\mathfrak{C}$-spans is  defined by a universal property, and therefore is only weakly associative and unital. We can, however, take $\Span{\mathfrak{C}}$ to be an \emph{unitary double category}: one for which the double unitors \(\lambda _X\) and \(\rho _X\) are identities. This is the case since \(\mathfrak{C}^l\) and \(\mathfrak{C}^r\) contain identities, so we can always make adequate choices of pullbacks when composing $\mathfrak{C}$-spans.  We will assume that is the case throughout this paper. 
\end{remark}

\begin{example}[Quintets]
    For any 2-category \(\mathcal {K}\), Ehresmann's \cite{Ehresmann} \emph{double category of quintets 
over \(\mathcal {K}\)} is the double category \(\mathbb {Q}\mathsf {t}(\mathcal {K})\) whose tight and loose arrows are the \(1\)-morphisms of \(\mathcal {K}\), which compose as in \(\mathcal {K}\), and whose squares 

\begin{center}\begin{tikzcd}
	A & B \\
	D & C
	\arrow[""{name=0, anchor=center, inner sep=0}, "k"{inner sep=.8ex}, "\shortmid"{marking}, from=1-1, to=1-2]
	\arrow["f"', from=1-1, to=2-1]
	\arrow["{{f'}}", from=1-2, to=2-2]
	\arrow[""{name=1, anchor=center, inner sep=0}, "{k'}"'{inner sep=.8ex}, "\shortmid"{marking}, from=2-1, to=2-2]
	\arrow["{\alpha }"{description}, color={rgb,255:red,92;green,92;blue,214}, draw=none, from=1, to=0]
\end{tikzcd}\end{center}
 are 2-cells \(\alpha  \colon  k'\circ  f \Rightarrow  f' \circ  k\) in \(\mathcal {K}\) going \emph{up the page}, which compose as in $\mathcal{K}$. 
Note that \(\mathbb {Q}\mathsf {t}(\mathcal {K})\) is a symmetric monoidal double category whenever \(\mathcal {K}\) is a symmetric monoidal 2-category, since $\Qt{-} \colon \TwoCat \to \Dbl$ is clearly a cartesian 2-functor.\end{example}

\par{}In order to speak of hyperdoctrines in double-categorical language, we must be able to incorporate adjunctions. For double categories, the
analogous statement of the triangular identities is given through the notion of \emph{conjoint}, but these also come with a dual notion of \emph{companion}:
\begin{definition}[{Companions and conjoints}]
\par{}Let \(\mathbb {D}\) be a double category and \(f \colon  A \to  B\) be a tight morphism in \(\mathbb {D}\).
A \emph{companion} for \(f\) is a loose morphism \(f^*\colon  A \overset {}{\mathrel {\mkern 3mu\vcenter {\hbox {$\shortmid $}}\mkern -10mu{\to }}} B\), together with squares 
\begin{equation*}
    \begin{tikzcd}
	A & B & A & A \\
	B & B & A & B
	\arrow[""{name=0, anchor=center, inner sep=0}, "{{f^*}}"{inner sep=.8ex}, "\shortmid"{marking}, from=1-1, to=1-2]
	\arrow["f"', from=1-1, to=2-1]
	\arrow[equals, from=1-2, to=2-2]
	\arrow[""{name=1, anchor=center, inner sep=0}, "\shortmid"{marking}, equals, from=1-3, to=1-4]
	\arrow["f", from=1-4, to=2-4]
	\arrow[""{name=2, anchor=center, inner sep=0}, "\shortmid"{marking}, equals, from=2-1, to=2-2]
	\arrow[equals, from=2-3, to=1-3]
	\arrow[""{name=3, anchor=center, inner sep=0}, "{f^*}"'{inner sep=.8ex}, "\shortmid"{marking}, from=2-3, to=2-4]
	\arrow["{\phi_f}"{description}, color={rgb,255:red,92;green,92;blue,214}, draw=none, from=2, to=0]
	\arrow["{\psi_f}"{description}, color={rgb,255:red,92;green,92;blue,214}, draw=none, from=3, to=1]
\end{tikzcd}
\end{equation*}
 such that 
\begin{center}\adjustbox{max width=\textwidth}{\begin{tikzcd}
	A && B && A && A && B \\
	&&& {=} \\
	A && B && A && B && B
	\arrow[""{name=0, anchor=center, inner sep=0}, "{f^*}"{inner sep=.8ex}, "\shortmid"{marking}, from=1-1, to=1-3]
	\arrow[equals, from=1-1, to=3-1]
	\arrow[equals, from=1-3, to=3-3]
	\arrow[""{name=1, anchor=center, inner sep=0}, "\shortmid"{marking}, equals, from=1-5, to=1-7]
	\arrow[equals, from=1-5, to=3-5]
	\arrow[""{name=2, anchor=center, inner sep=0}, "{f^*}"{inner sep=.8ex}, "\shortmid"{marking}, from=1-7, to=1-9]
	\arrow["f"', from=1-7, to=3-7]
	\arrow[equals, from=1-9, to=3-9]
	\arrow[""{name=3, anchor=center, inner sep=0}, "{f^*}"'{inner sep=.8ex}, "\shortmid"{marking}, from=3-1, to=3-3]
	\arrow[""{name=4, anchor=center, inner sep=0}, "{f^*}"'{inner sep=.8ex}, "\shortmid"{marking}, from=3-5, to=3-7]
	\arrow[""{name=5, anchor=center, inner sep=0}, "\shortmid"{marking}, equals, from=3-7, to=3-9]
	\arrow["{{v_{f^*}}}"{description}, color={rgb,255:red,92;green,92;blue,214}, draw=none, from=3, to=0]
	\arrow["{\psi_f}"{description}, color={rgb,255:red,92;green,92;blue,214}, draw=none, from=4, to=1]
	\arrow["{\phi_f}"{description}, color={rgb,255:red,92;green,92;blue,214}, draw=none, from=5, to=2]
\end{tikzcd}}\end{center}
 and 
\begin{center}\adjustbox{max width=\textwidth}{\begin{tikzcd}
	A && A && A & A \\
	&&& {=} & A & B \\
	B && B && B & B
	\arrow[""{name=0, anchor=center, inner sep=0}, "\shortmid"{marking}, equals, from=1-1, to=1-3]
	\arrow["f"', from=1-1, to=3-1]
	\arrow["f", from=1-3, to=3-3]
	\arrow[""{name=1, anchor=center, inner sep=0}, "\shortmid"{marking}, equals, from=1-5, to=1-6]
	\arrow[equals, from=1-5, to=2-5]
	\arrow["f", from=1-6, to=2-6]
	\arrow[""{name=2, anchor=center, inner sep=0}, "{f^*}"'{inner sep=.8ex}, "\shortmid"{marking}, from=2-5, to=2-6]
	\arrow["f"', from=2-5, to=3-5]
	\arrow[equals, from=2-6, to=3-6]
	\arrow[""{name=3, anchor=center, inner sep=0}, "\shortmid"{marking}, equals, from=3-1, to=3-3]
	\arrow[""{name=4, anchor=center, inner sep=0}, "\shortmid"{marking}, equals, from=3-5, to=3-6]
	\arrow["{\psi_f}"{description}, color={rgb,255:red,92;green,92;blue,214}, draw=none, from=2, to=1]
	\arrow["{{e_f}}"{description}, color={rgb,255:red,92;green,92;blue,214}, draw=none, from=3, to=0]
	\arrow["{\phi_f}"{description}, color={rgb,255:red,92;green,92;blue,214}, draw=none, from=4, to=2]
\end{tikzcd}}.\end{center}

\par{}Similarly, a conjoint for \(f\) is a loose map \(B \overset {f_! }{\mathrel {\mkern 3mu\vcenter {\hbox {$\shortmid $}}\mkern -10mu{\to }}} A\) together with
squares 
\begin{center}\begin{tikzcd}
	B & A & A & A \\
	B & B & B & A
	\arrow[""{name=0, anchor=center, inner sep=0}, "\shortmid"{marking}, from=1-1, to=1-2]
	\arrow[equals, from=1-1, to=2-1]
	\arrow["f", from=1-2, to=2-2]
	\arrow[""{name=1, anchor=center, inner sep=0}, "\shortmid"{marking}, equals, from=1-3, to=1-4]
	\arrow["f"', from=1-3, to=2-3]
	\arrow[equals, from=1-4, to=2-4]
	\arrow[""{name=2, anchor=center, inner sep=0}, "\shortmid"{marking}, equals, from=2-1, to=2-2]
	\arrow[""{name=3, anchor=center, inner sep=0}, "{f_!}"'{inner sep=.8ex}, "\shortmid"{marking}, from=2-3, to=2-4]
	\arrow["{{\hat{\eta}_f}}"{description}, color={rgb,255:red,92;green,92;blue,214}, draw=none, from=2, to=0]
	\arrow["{{\hat{\epsilon}_f}}"{description}, color={rgb,255:red,92;green,92;blue,214}, draw=none, from=3, to=1]
\end{tikzcd}\end{center}
 in \(\mathbb {D}\) such that

\begin{center}\begin{tikzcd}
	B && A && B && A && A \\
	&&& {=} \\
	B && A && B && B && A
	\arrow[""{name=0, anchor=center, inner sep=0}, "{f_!}"{inner sep=.8ex}, "\shortmid"{marking}, from=1-1, to=1-3]
	\arrow[equals, from=1-1, to=3-1]
	\arrow[equals, from=1-3, to=3-3]
	\arrow[""{name=1, anchor=center, inner sep=0}, "{ }"{marking, allow upside down}, "\shortmid"{marking}, from=1-5, to=1-7]
	\arrow[equals, from=1-5, to=3-5]
	\arrow[""{name=2, anchor=center, inner sep=0}, "\shortmid"{marking}, equals, from=1-7, to=1-9]
	\arrow["f"', from=1-7, to=3-7]
	\arrow[equals, from=1-9, to=3-9]
	\arrow[""{name=3, anchor=center, inner sep=0}, "{f_!}"'{inner sep=.8ex}, "\shortmid"{marking}, from=3-1, to=3-3]
	\arrow[""{name=4, anchor=center, inner sep=0}, "\shortmid"{marking}, equals, from=3-5, to=3-7]
	\arrow[""{name=5, anchor=center, inner sep=0}, "{f_!}"'{inner sep=.8ex}, "\shortmid"{marking}, from=3-7, to=3-9]
	\arrow["{{v_{f_! }}}"{description}, color={rgb,255:red,92;green,92;blue,214}, draw=none, from=3, to=0]
	\arrow["{{\hat{\eta}_f}}"{description}, color={rgb,255:red,92;green,92;blue,214}, draw=none, from=4, to=1]
	\arrow["{\hat{\epsilon}_f}"{description}, color={rgb,255:red,92;green,92;blue,214}, draw=none, from=5, to=2]
\end{tikzcd}\end{center}
 and 
\begin{center}\begin{tikzcd}
	A && A && A && A \\
	&&& {=} & B && A \\
	B && B && B && B
	\arrow[""{name=0, anchor=center, inner sep=0}, "\shortmid"{marking}, equals, from=1-1, to=1-3]
	\arrow["f"', from=1-1, to=3-1]
	\arrow["f", from=1-3, to=3-3]
	\arrow[""{name=1, anchor=center, inner sep=0}, "{\shortmid }"{marking, allow upside down}, equals, from=1-5, to=1-7]
	\arrow["f"', from=1-5, to=2-5]
	\arrow[equals, from=1-7, to=2-7]
	\arrow[""{name=2, anchor=center, inner sep=0}, "{f_! }"{inner sep=.8ex}, "\shortmid"{marking}, from=2-5, to=2-7]
	\arrow[equals, from=2-5, to=3-5]
	\arrow["f", from=2-7, to=3-7]
	\arrow[""{name=3, anchor=center, inner sep=0}, "\shortmid"{marking}, equals, from=3-1, to=3-3]
	\arrow[""{name=4, anchor=center, inner sep=0}, "\shortmid"{marking}, equals, from=3-5, to=3-7]
	\arrow["{{\hat{\epsilon}_f}}"{description}, color={rgb,255:red,92;green,92;blue,214}, draw=none, from=2, to=1]
	\arrow["{{e_f}}"{description}, color={rgb,255:red,92;green,92;blue,214}, draw=none, from=3, to=0]
	\arrow["{\hat{\eta}_f}"{description}, color={rgb,255:red,92;green,92;blue,214}, draw=none, from=4, to=2]
\end{tikzcd}.\end{center}\end{definition}

\begin{remark}[{}]
\par{}If \(\mathcal {K}\) is a 2-category, then a conjoint pair in \(\mathbb {Q}\mathsf {t}(\mathcal {K})\) is an
adjunction internal to \(\mathcal {K}\). In there, every tight morphism has a companion (namely itself, seen as a loose morphism). \end{remark}

Recall that a strong morphism of existential structures is a 2-natural transformation that preserves existential quantification, which is to say that the mate of its component at any quantifiable morphism is invertible.  Our consideration of transformations with this property and of Frob\"enius reciprocity  lead to the notion of \emph{companion commuter} in the sense of Paré  \cite{pare-2024-retrocells}:
\begin{definition}
\par{}A square 
\begin{center}\begin {tikzcd}[cramped]
	{Y_1} && {Y_2} \\
	\\
	{X_1} && {X_2}
	\arrow[""{name=0, anchor=center, inner sep=0}, "Y"{inner sep=.8ex}, "\shortmid "{marking}, from=1-1, to=1-3]
	\arrow["{f_1}"', from=1-1, to=3-1]
	\arrow["{f_2}", from=1-3, to=3-3]
	\arrow[""{name=1, anchor=center, inner sep=0}, "X"'{inner sep=.8ex}, "\shortmid "{marking}, from=3-1, to=3-3]
	\arrow["\alpha "'{description}, draw=none, color={rgb,255:red,92;green,92;blue,214}, shorten <=10pt, shorten >=10pt, Rightarrow, from=1, to=0]
\end {tikzcd}\end{center}
 in a double category \(\mathbb {D}\) is a \emph{companion commuter} if \(f_1\) and \(f_2\) have companions and
its associated transpose square 

\begin{center}\begin{tikzcd}
	{Y_1} && {Y_1} && {Y_2} && {X_2} \\
	\\
	{Y_1} && {X_1} && {X_2} && {X_2}
	\arrow[""{name=0, anchor=center, inner sep=0}, "\shortmid"{marking}, equals, from=1-1, to=1-3]
	\arrow[equals, from=1-1, to=3-1]
	\arrow[""{name=1, anchor=center, inner sep=0}, "Y"{inner sep=.8ex}, "\shortmid"{marking}, from=1-3, to=1-5]
	\arrow["{{f_1}}"', from=1-3, to=3-3]
	\arrow[""{name=2, anchor=center, inner sep=0}, "{f_2^*}"{inner sep=.8ex}, "\shortmid"{marking}, from=1-5, to=1-7]
	\arrow["{{f_2}}", from=1-5, to=3-5]
	\arrow[equals, from=1-7, to=3-7]
	\arrow[""{name=3, anchor=center, inner sep=0}, "{f_1^*}"'{inner sep=.8ex}, "\shortmid"{marking}, from=3-1, to=3-3]
	\arrow[""{name=4, anchor=center, inner sep=0}, "X"'{inner sep=.8ex}, "\shortmid"{marking}, from=3-3, to=3-5]
	\arrow[""{name=5, anchor=center, inner sep=0}, "\shortmid"{marking}, equals, from=3-5, to=3-7]
	\arrow["{\psi_{f_1}}"{description}, color={rgb,255:red,92;green,92;blue,214}, draw=none, from=3, to=0]
	\arrow["{\alpha }"'{description}, draw=none, color={rgb,255:red,153;green,92;blue,214}, between={0.25}{0.75}, Rightarrow, from=4, to=1]
	\arrow["{\phi_{f_2}}"{description}, color={rgb,255:red,92;green,92;blue,214}, draw=none, from=5, to=2]
\end{tikzcd}\end{center}
 is an isomorphism in $\mathbb{D}_1$. The dual notion is that of a \emph{conjoint commuter}.  By a \emph{companion commuter transformation} 
we mean a tight transformation whose cell components are companion commuter squares.
\end{definition}

We aim to show that, under mild conditions, \(\mathfrak{C}\)-regular $\mathcal{K}$-hyperdoctrines correspond to lax symmetric monoidal pseudo double functors
for which the monoidal laxators are companion commuter transformations. We will split the proof of this correspondence into two sections:
in Section~\ref{section:DoctrinesAsDoublePseudofunctors}, we show that if $\mathfrak{C}$ is cartesian, then the underlying 2-functor of a \(\mathfrak{C}\)-regular $\mathcal{K}$-hyperdoctrine can be extended to
such a pseudo double functor \(\mathbb {S}\mathsf {pan}(\mathfrak{C})^{\op} \to  \mathbb {Q}\mathsf {t}(\mathcal {K})\). Section~\ref{section:DoublePseudofunctorsAsDoctrines} is then dedicated to proving a converse: if the class $\mathsf{C}^l$ also has all diagonals, then the tight component of such a map yields a $\mathfrak{C}$-regular hyperdoctrine. 

\section{Hyperdoctrines as pseudo double functors}
\label{section:DoctrinesAsDoublePseudofunctors}

 We divide the task into two: first we argue that \emph{existential symmetric monoidal structures} can be pseudofunctorially extended to lax symmetric monoidal  pseudo double functors, and then show that their laxators have the companion commuter property. For the first goal, the main technical lemma is the following:

\begin{lemma}
\label{lemma:BulletConstruction}
    There is a cartesian 2-functor $(-)^{\bullet} \colon \ExistStr_{\mathrm{w}} \to \Ar^{\colax}(\mathrm{Dbl})$, where $\mathrm{Dbl}$ is the cartesian 2-category of (weak) double categories, pseudo double functors, and tight transformations.
\end{lemma}

We will describe the construction of $(-)^{\bullet}$ here, but leave the proof of Lemma~\ref{lemma:BulletConstruction} to the Appendix. We will sometimes explicitly include the compositors of 2-functors and the naturality 2-cells of 2-natural transformations for clarity, even though they are identities.

\begin{itemize}
    \item An object $(\mathfrak{C},\mathsf{C}^{op} \xrightarrow{P} \mathcal{K})$ is mapped to the pseudo double functor $P^{\bullet} \colon \Span{\mathfrak{C}}^{\op} \to \Qt{\mathcal{K}}$ whose tight component is just $P$, and that maps a $\mathfrak{C}$-span $X_1 \xleftarrow{x_1} X \xrightarrow{x_2} X_2$ to the composite $\exists^P x_2 \circ Px_1$ (as a loose arrow in $\Qt{\mathcal{K}}$) and a square 
    \begin {tikzcd}[cramped]
	{Y_1} && {Y_2} \\
	\\
	{X_1} && {X_2}
	\arrow ["Y", "\shortmid "{marking}, from=1-1, to=1-3]
	\arrow [""{name=0, anchor=center, inner sep=0}, "{f_1}"', from=1-1, to=3-1]
	\arrow [""{name=1, anchor=center, inner sep=0}, "{f_2}", from=1-3, to=3-3]
	\arrow ["X"', "\shortmid "{marking}, from=3-1, to=3-3]
	\arrow ["\alpha "{description}, color={rgb,255:red,92;green,92;blue,214}, draw=none, from=0, to=1]
\end {tikzcd}
 in  \(\mathbb {S}\mathsf {pan}(\mathfrak{C})^{\mathsf {op}}\) to the square in $\Qt{\mathcal{K}}$ corresponding to  \adjustbox{max width=\textwidth}{\begin{tikzcd}[cramped]
	{PY_1} & PY & {PY_2} & {PX_2} & {PX_2} \\
	{PY_1} & PY & PY & PX & {PX_2} \\
	{PX_1} && PX && {PX_2}
	\arrow ["{P(y_1)}", from=1-1, to=1-2]
	\arrow [""{name=0, anchor=center, inner sep=0},  equal, from=1-1, to=2-1]
	\arrow [""{name=1, anchor=center, inner sep=0}, "{\exists  y_2}", from=1-2, to=1-3]
	\arrow [""{name=2, anchor=center, inner sep=0}, equal, from=1-2, to=2-2]
	\arrow ["{P(f_2)}", from=1-3, to=1-4]
	\arrow [""{name=3, anchor=center, inner sep=0}, "{P(y_2)}", from=1-3, to=2-3]
	\arrow [""{name=4, anchor=center, inner sep=0}, equal, from=1-4, to=1-5]
	\arrow [""{name=5, anchor=center, inner sep=0}, "{P(x_2)}", from=1-4, to=2-4]
	\arrow [equal, from=1-5, to=2-5]
	\arrow ["{P(y_1)}", from=2-1, to=2-2]
	\arrow [""{name=6, anchor=center, inner sep=0}, "{P(f_1)}"', from=2-1, to=3-1]
	\arrow [""{name=7, anchor=center, inner sep=0}, equal, from=2-2, to=2-3]
	\arrow ["{P(\alpha )}"', from=2-3, to=2-4]
	\arrow [""{name=8, anchor=center, inner sep=0}, "{P(\alpha )}"', from=2-3, to=3-3]
	\arrow [""{name=9, anchor=center, inner sep=0}, "{\exists  x_2}"', from=2-4, to=2-5]
	\arrow [""{name=10, anchor=center, inner sep=0}, equal, from=2-5, to=3-5]
	\arrow ["{P(x_1)}"', from=3-1, to=3-3]
	\arrow ["{\exists  x_2}"', from=3-3, to=3-5]
	\arrow ["\circlearrowleft "{description}, color={rgb,255:red,92;green,92;blue,214}, draw=none, from=0, to=2]
	\arrow ["\circlearrowleft "{description}, xshift=1ex, color={rgb,255:red,92;green,92;blue,214}, draw=none, from=3, to=5]
	\arrow ["\circlearrowleft"{description}, color={rgb,255:red,92;green,92;blue,214}, draw=none, from=6, to=8]
	\arrow ["{\eta ^{y_2}}"', color={rgb,255:red,92;green,92;blue,214}, shorten <=4pt, shorten >=4pt, Rightarrow, from=7, to=1]
	\arrow ["\circlearrowleft "{description}, color={rgb,255:red,92;green,92;blue,214}, draw=none, from=8, to=10]
	\arrow ["{\epsilon ^{x_2}}"', color={rgb,255:red,92;green,92;blue,214}, shorten <=4pt, shorten >=4pt, Rightarrow, from=9, to=4]
\end{tikzcd}} in $\mathcal{K}$.\\

\noindent The component of the natural isomorphism  \(P^{\bullet }_{\odot }\) at a composite $\mathfrak{C}$-span  \begin{equation*}
    \begin{tikzcd}
	&& {X\odot X'} && \\
	& X && {X'} \\
	{X_1} && {X_2} && {X_3}
	\arrow["\chi"', from=1-3, to=2-2]
	\arrow["{\chi'}", from=1-3, to=2-4]
	\arrow["\lrcorner"{anchor=center, pos=0.125, rotate=-45}, draw=none, from=1-3, to=3-3]
	\arrow["{x_1}"', from=2-2, to=3-1]
	\arrow["{x_2}"', from=2-2, to=3-3]
	\arrow["{x_2'}", from=2-4, to=3-3]
	\arrow["{x_3'}", from=2-4, to=3-5]
\end{tikzcd}
\end{equation*} is given by the map \(P^{\bullet }(X \odot  X')  \Rightarrow  P^{\bullet }(X) \odot  P^{\bullet }(X')\) corresponding to the pasting diagram
\begin{center}\adjustbox{max width=\textwidth}{
\begin{tikzcd}
	{PX_1} && {P(X\odot  X')} && {PX_3} \\
	{PX_1} & PX & {P(X\odot  X')} & {PX'} & {PX_3} \\
	{PX_1} & PX & {PX_2} & {PX'} & {PX_3}
	\arrow[""{name=0, anchor=center, inner sep=0}, "{P(x_1 \chi)}", from=1-1, to=1-3]
	\arrow[equals, from=1-1, to=2-1]
	\arrow[""{name=1, anchor=center, inner sep=0}, "{\exists^P(x_3'\chi')}", from=1-3, to=1-5]
	\arrow[equals, from=1-3, to=2-3]
	\arrow[equals, from=1-5, to=2-5]
	\arrow["{Px_1}"', from=2-1, to=2-2]
	\arrow[equals, from=2-1, to=3-1]
	\arrow["{P \chi}"', from=2-2, to=2-3]
	\arrow[equals, from=2-2, to=3-2]
	\arrow["{\exists^P \chi'}"', from=2-3, to=2-4]
	\arrow["{\exists^P x_3'}"', from=2-4, to=2-5]
	\arrow[equals, from=2-4, to=3-4]
	\arrow[equals, from=2-5, to=3-5]
	\arrow["{Px_1}"', from=3-1, to=3-2]
	\arrow["{\exists^P x_2}"', from=3-2, to=3-3]
	\arrow["{\mathcal{B}_{x_2, x_2'}^{\chi, \chi'}}", color={rgb,255:red,153;green,92;blue,214}, Rightarrow, from=3-3, to=2-3]
	\arrow["{Px_2'}"', from=3-3, to=3-4]
	\arrow["{\exists^P x_3'}"', from=3-4, to=3-5]
	\arrow["{\gamma^P_{x_1, \chi~}}"', color={rgb,255:red,92;green,92;blue,214}, between={0}{0.8}, Rightarrow, from=2-2, to=0]
	\arrow["{\gamma^{\exists^P}_{\chi', x_3'}}", color={rgb,255:red,92;green,92;blue,214}, between={0}{0.8}, Rightarrow, from=2-4, to=1]
\end{tikzcd}},\end{center}
where $\mathcal{B}_{x_2, x_2'}^{\chi, \chi'}$ is the Beck-Chevalley 2-cell for the $\mathfrak{C}$-pullback defining $X \odot X'$. This exists by definition of $\mathfrak{C}$-existential $\mathcal{K}$-structure.\\

The unitor $P^{\bullet}_e \colon e \circ P \to P_1^{\bullet} \circ e$ is the natural isomorphism whose component at a $\mathfrak{C}$-span $X_1 \xleftarrow{x_1} X \xrightarrow{x_2} X_2$ is the counit $\epsilon^{\id_X}$ of the adjunction $\exists^P(\id_X) \dashv P(\id_X)$.\\

\item A 1-morphism $(F, \rho, L) \colon P_1 \to P_2$ in $\ExistStr_{\mathrm{w}}$ is mapped to  $(\Span{F}^{\op}, \rho^{\bullet}, \Qt{L})$ in $\Ar^{\colax}(\mathrm{Dbl})$, where the tight transformation $\rho^{\bullet}$ agrees with the 2-natural transformation $\rho \colon LP_1 \to P_2F^{\op}$ on objects and tight arrows, and its component at a $\mathfrak{C}$-span $X_1 \xleftarrow{x_1} X \xrightarrow{x_2} X_2$ corresponds to the pasting diagram
\begin{equation*}
\begin{tikzcd}
	{LP_1X_1} & {LP_1X} && {LP_1X_2} && {P_2FX_2} && {P_2FX_2} \\
	{LP_1X_1} & {LP_1X} && { LP_1X } && {P_2FX} && {P_2FX_2} \\
	\\
	{P_2FX_1} &&&&& {P_2FX} && {P_2FX_2}
	\arrow["{LP_1{x_1}}", from=1-1, to=1-2]
	\arrow[equals, from=1-1, to=2-1]
	\arrow[""{name=0, anchor=center, inner sep=0}, "{L\exists^{P_1}x_2}", from=1-2, to=1-4]
	\arrow[equals, from=1-2, to=2-2]
	\arrow[""{name=1, anchor=center, inner sep=0}, "{\rho_{X_2}}", from=1-4, to=1-6]
	\arrow["{LP_2x_2}", from=1-4, to=2-4]
	\arrow[""{name=2, anchor=center, inner sep=0}, equals, from=1-6, to=1-8]
	\arrow["{P_2Fx_2}", from=1-6, to=2-6]
	\arrow[equals, from=1-8, to=2-8]
	\arrow["{LP_1x_1}"', from=2-1, to=2-2]
	\arrow["{\rho_{X_1}}"', from=2-1, to=4-1]
	\arrow[""{name=3, anchor=center, inner sep=0}, equals, from=2-2, to=2-4]
	\arrow[""{name=4, anchor=center, inner sep=0}, "{\rho_X}"', from=2-4, to=2-6]
	\arrow[""{name=5, anchor=center, inner sep=0}, "{\exists^{P_2}Fx_2}"', from=2-6, to=2-8]
	\arrow[equals, from=2-6, to=4-6]
	\arrow[equals, from=2-8, to=4-8]
	\arrow[""{name=6, anchor=center, inner sep=0}, "{P_2Fx_1}"', from=4-1, to=4-6]
	\arrow["{\exists^{P_2}Fx_2}"', from=4-6, to=4-8]
	\arrow["{L\eta^{P_1 x_2}}", color={rgb,255:red,92;green,92;blue,214}, between={0.2}{0.8}, Rightarrow, from=3, to=0]
	\arrow["{\rho_{x_2}}"', color={red}, between={0.2}{0.8}, Rightarrow, from=4, to=1]
	\arrow["{\epsilon^{P_2Fx_2}}"', color={rgb,255:red,92;green,92;blue,214}, between={0.2}{0.8}, Rightarrow, from=5, to=2]
	\arrow["{\rho^{-1}_{x_1}}"', color={red}, between={0.2}{0.8}, Rightarrow, from=6, to=3]
\end{tikzcd}
\end{equation*}
\\

\item A 2-morphism $(\alpha, \beta)$ in $\ExistStr$ is mapped to  $(\Span{\alpha}^{\op}, \Qt{\beta})$ in $\Ar^{\colax}(\mathrm{Dbl})$.\\

\end{itemize}

\begin{corollary}
\label{corollary:Double functors from existential structures}
   The 2-functor $(-)^{\bullet}$ induces a pseudofunctor 
   \begin{equation*}
       \SM^{\lax} \ExistStr_{\mathrm{w}}  \to \SM^{\lax} \Ar^{\colax}( \mathrm{Dbl}) \cong \Ar^{\colax}(\SM^{\lax} \mathrm{Dbl}).
   \end{equation*}
   In particular, every $\mathfrak{C}$-existential symmetric monoidal $\mathcal{K}$-structure $P \colon \mathsf{C}^{\op} \to \mathcal{K}$ can be extended to a lax symmetric monoidal pseudo double functor $\Span{\mathfrak{C}}^{\op} \to \Qt{\mathcal{K}}$.
\end{corollary}
\begin{proof}
    Lax symmetric monoidal  pseudofunctors preserve symmetric pseudomonoids. 
\end{proof}

\begin{remark}
    It is clear from the definition of $(-)^{\bullet}$ on 1-morphisms that $(-)^{\bullet}$ restricts to a cartesian 2-functor on $\ExistStr$ whose image lies in $\Ar^{\colax}(\Dbl)_{\mathrm{cc}}$, the sub-2-category of $\Ar^{\colax}(\Dbl)$ whose 1-morphisms involve only tight transformations that are companion commuter transformations.
\end{remark}

In face of this remark, we now argue that if $P$ is not just an existential structure, but a $\mathfrak{C}$-regular $\mathcal{K}$-hyperdoctrine $(P, \mu, I)$, then $(P^{\bullet}, \mu^{\bullet}, I^{\bullet})$ is a symmetric pseudomonoid in $\Ar^{\colax}(\Dbl)_{\mathrm{cc}}$.  Clearly, \(I^{\bullet }\) is a companion commuter transformation; a square in $\Qt{\mathcal{K}}$ is a companion commuter precisely when it is an iso-2-cell in $\mathcal{K}$. In principle, however, there is no reason to expect the cell components of \(\mu ^{\bullet }\) to be invertible --- the mate of an invertible 2-cell, or even an identity 2-cell, need not be invertible. This is true here precisely due to the Frob\"enius law, as long as the adequate triple $\mathfrak{C}$ is cartesian.

\begin{lemma} 
\label{lemma: regular hyperdoctrine extends to double functor}
 Let $\mathfrak{C}$ be a cartesian adequate triple, $\mathcal{K}$ be a symmetric monoidal 2-category, and $P$ be a $\mathfrak{C}$-regular $\mathcal{K}$-hyperdoctrine. Then the component  \(\mu ^{\bullet }_{X, Y}\) of \(P^{\bullet } \times  P^{\bullet } \xrightarrow{\mu^{\bullet}}  P^{\bullet }(-\times -)\) at $\mathfrak{C}$-spans \((X_1 \xleftarrow {x_1} X \xrightarrow {x_2} X_2, Y_1 \xleftarrow {y_1} Y \xrightarrow {y_2} Y_2)\) 
is a companion commuter square in $\Qt{\mathcal{K}}$. 
\end{lemma}

\begin{proof}\par{}Since \(P^{\bullet }\) takes values in \(\mathbb {Q}\mathsf {t}(\mathcal {K})\), the cell \(\mu ^{\bullet }_{X,Y}\) being a companion commuter amounts to it being invertible as a cell in \(\mathcal {K}\). By the definition
of \(\mu ^{\bullet }_{X,Y}\), it is enough to show that the mate of the naturality square \(\mu^P_{x_2, y_2}\) is invertible. 

Factorise $x_2 \times y_2$ as $(\id_{X_2} \times y_2) \circ (x_2 \times \id_Y)$ in $\mathsf{C}$. This lets us express the mate of $\mu^P_{x_2, y_2}$ as the composite of the mates of $\mu^P_{\id_{X_2}, y_2}$ and $\mu^P_{x_2, \id_Y}$, by the functoriality of mates. It is thus enough to show that each of these mates is invertible.

Recall from Proposition~\ref{Proposition:Internal vs External monoidal structures} that $\mu^P_{\id_{X_2}, y_2}$ can be expressed in terms of the monoidal structure of the fibres as
\begin{equation*}
\begin{tikzcd}
	& {PX_2 \otimes_{\mathcal{K}}PY_2} &&&& {PX_2 \otimes_{\mathcal{K}} PY} \\
	\\
	{\mu^P_{\id_{X_2}, y_2} =} & {P(X_2 \times Y_2) \otimes_{\mathcal{K}}P(X_2 \times Y_2)} &&&& {P(X_2 \times Y) \otimes_{\mathcal{K}} P(X_2 \times Y)} \\
	\\
	& {P(X_2 \times Y_2)} &&&& {P(X_2 \times Y)}
	\arrow[""{name=0, anchor=center, inner sep=0}, "{\id_{PX_2} \otimes_{\mathcal{K}} Py_2}", from=1-2, to=1-6]
	\arrow["{P(\pi^{X_2, Y_2}_{X_2}) \otimes_{\mathcal{K}}P(\pi^{X_2, Y_2}_{Y_2})}"', from=1-2, to=3-2]
	\arrow["{P(\pi^{X_2, Y}_{X_2}) \otimes_{\mathcal{K}}P(\pi^{X_2, Y}_Y)}", from=1-6, to=3-6]
	\arrow[""{name=1, anchor=center, inner sep=0}, "{P(\id_{X_2} \times y_2) \otimes_{\mathcal{K}} P(\id_{X_2} \times y_2)}", from=3-2, to=3-6]
	\arrow["{\otimes_{P(X_2 \times Y_2)}}"', from=3-2, to=5-2]
	\arrow["{\otimes_{P(X_2 \times Y)}}", from=3-6, to=5-6]
	\arrow[""{name=2, anchor=center, inner sep=0}, "{P(\id_{X_2} \times y_2)}"', from=5-2, to=5-6]
	\arrow["{P(\alpha_1) \otimes_{\mathcal{K}} P(\alpha_2)}", color={rgb,255:red,214;green,92;blue,92}, between={0.2}{0.8}, Rightarrow, from=0, to=1]
	\arrow["{\mu^{P(\id_{X_2} \times y_2)}}", color={rgb,255:red,92;green,92;blue,214}, between={0.2}{0.8}, Rightarrow, from=1, to=2]
\end{tikzcd},
\end{equation*}

where 

\begin{equation*}
    \begin{tikzcd}
	{X_2 \times Y} && {X_2} && {X_2 \times Y} && Y \\
	\\
	{X_2 \times Y_2} && {X_2} && {X_2 \times Y_2} && {Y_2}
	\arrow[""{name=0, anchor=center, inner sep=0}, "{\pi^{X_2, Y}_{X_2}}", from=1-1, to=1-3]
	\arrow["{\id_{X_2} \times y_2}"', from=1-1, to=3-1]
	\arrow[equals, from=1-3, to=3-3]
	\arrow[""{name=1, anchor=center, inner sep=0}, "{\pi^{X_2, Y}_Y}", from=1-5, to=1-7]
	\arrow["{\id_{X_2} \times y_2}"', from=1-5, to=3-5]
	\arrow["\lrcorner"{anchor=center, pos=0.125}, draw=none, from=1-5, to=3-7]
	\arrow["{y_2}", from=1-7, to=3-7]
	\arrow[""{name=2, anchor=center, inner sep=0}, "{\pi^{X_2, Y_2}_{X_2}}"', from=3-1, to=3-3]
	\arrow[""{name=3, anchor=center, inner sep=0}, "{\pi^{X_2, Y_2}_{Y_2}}"', from=3-5, to=3-7]
	\arrow["{\alpha_1}"{description}, color={rgb,255:red,214;green,92;blue,92}, draw=none, from=0, to=2]
	\arrow["{\alpha_2}"{description}, color={rgb,255:red,214;green,92;blue,92}, draw=none, from=1, to=3]
\end{tikzcd}
\end{equation*}

The mate of $\mu^P_{\id_{X_2}, y_2}$ can therefore be expressed as the following composite of mates:

\begin{equation*}
    = \adjustbox{scale=0.7}{\begin{tikzpicture}[baseline=5em, marker/.style={circle, fill, inner sep=1pt, text=white}]
\node[box, box ports south=5, box ports north=5, color=black] (Box1) at (0,0) {e_{\id_{PX_2}} \otimes_{\mathcal{K}} \eta^{Py_2}};

\node[box, box ports south=5, box ports north=5, minimum width=5cm, color=red] (Box2) at ($(Box1.north.5)+(1.5,2)$) {P\alpha_1 \otimes_{\mathcal{K}} P\alpha_2};

\node[box, box ports south=5, box ports north=5, minimum width=5cm, color=blue] (Box3) at ($(Box2.north.5)+(4,2)$) {\mu^{P(\id_{X_2} \times y_2)}};

\node[coordinate, label=below:{{\scriptsize $P^{\pi^{X_2, Y}_{X_2}} \otimes_{\mathcal{K}} P(\pi^{X_2, Y}_Y)$}}]                          (input1) at ($(Box2.south.4)+(0,-3)$) {};
\node[coordinate, label={[text=black]below:{{\scriptsize $\otimes_{P(X_2 \times Y)}$}}}]             (input2) at ($(Box3.south.4)+(0,-5.5)$) {};
\node[coordinate, label={[text=black]below:{{\scriptsize $\exists^P(\id_{X_2} \times y_2)$}}}]            (input3) at ($(input2)+(3,0)$) {};

\node[coordinate, label={[text=black]above:{{\scriptsize $\id_{PX_2} \otimes_{\mathcal{K}} \exists^P y_2$}}}]            (output1) at ($(Box1.north.1) + (0,7)$) {};
\node[coordinate, label={[text=black]above:{{\scriptsize $P(\pi^{x_2, Y_2}_{X_2}) \otimes_{\mathcal{K}} \pi^{X_2, Y_2}_{Y_2})$}}}]  (output2) at ($(Box2.north.2)+(0,4.5)$) {};
\node[coordinate, label={[text=black]above:{{\scriptsize $\otimes_{P(X_2, \times Y_2)}$}}}]           (output3) at ($(output2)+(6,0)$) {};


\node[coordinate] (cap1) at ($(Box2.north.5)+(1,0)$){};

\node[coordinate] (cup1) at ($(cap1)+(1,0)$){};

\node[coordinate] (cup2) at ($(Box3.north.4)+(3,0)$){};

\wires[violet, looseness=3, line width=2pt]{cap1 = {north = Box2.north.5}}{};

\wires[black, looseness=3, line width=2pt]{cup1 = {south = cap1.south, north = Box3.south.1}}{};

\wires[black, looseness=3]{cup2 = {north = Box3.north.4, south=input3.north}}{};

\wires[black, looseness=0]{Box1 = {north.1 = output1.south, north.4 = Box2.south.1}}{};

\wires[black, looseness=0]{Box2 = {north.2 = output2.south, south.4 = input1.north}}{};

\wires[black, looseness=0]{Box3 = {north.2 = output3.south, south.4 = input2.north}}{};


\node[coordinate, label={[text=black]right:{{\scriptsize $\id_{P_{X_2}} \otimes_{\mathcal{K}} Py_2$}}}]            (label1) at ($(Box2.south.1) + (0,-1)$) {};

\node[coordinate, label={[text=black]right:{{\scriptsize $P(\id_{X_2} \times y_2)^2$}}}]            (label2) at ($(Box3.south.1) + (0,-1)$) {};

\end{tikzpicture}}
\end{equation*}

The purple wire represents the tensor product of counits $\textcolor{violet}{\epsilon^{P(\id_{X_2} \times y_2)} \otimes_{\mathcal{K}} \epsilon^{P(\id_{X_2} \times y_2)}}$, which we can factorise as $(e_{\id_{P(X_2 \times Y_2)}} \otimes_{\mathcal{K}} \epsilon^{P(\id_{X_2} \times y_2)}) \mid (\epsilon^{P(\id_{X_2} \times y_2)} \otimes_{\mathcal{K}} v_{\exists^P (\id_{X_2} \times y_2)})$. Similarly, the red box $P\alpha_1 \otimes_{\mathcal{K}} P\alpha_2$ can be factorised as $(v_{P(\pi^{X_2, Y_2}_{X_2})} \otimes_{\mathcal{K}} P\alpha_2) \mid (P\alpha_1 \otimes_{\mathcal{K}} v_{P(\pi^{X_2, Y}_Y)})$. Making these changes lets us rewrite the mate of $\mu^P_{\id_{X_2}, y_2}$ as 

\begin{equation*}
    \adjustbox{scale=0.7}{\begin{tikzpicture}[baseline=5em, marker/.style={circle, fill, inner sep=1pt, text=white}]
\node[box, box ports south=5, box ports north=5, color=black] (Box1) at (0,0) {e_{\id_{PX_2}} \otimes_{\mathcal{K}} \eta^{Py_2}};

\node[box, box ports south=5, box ports north=5, minimum width=5cm, color=red] (Box2) at ($(Box1.north.5)+(1.5,2)$) {v_{P(\pi^{X_2, Y_2}_{X_2})} \otimes_{\mathcal{K}} P\alpha_2};

\node[box, box ports south=5, box ports north=5, minimum width=5cm, color=blue] (Box3) at ($(Box2.north.5)+(8,0)$) {\mu^{P(\id_{X_2} \times y_2)}};

\node[box, box ports south=5, box ports north=5, color=red] (Box4) at ($(Box2.south.4)+(1,-2)$) {P\alpha_1 \otimes_{\mathcal{K}} v_{P(\pi^{X_2, Y}_Y)}};

\node[box, box ports south=5, box ports north=5, color=violet] (Box5) at ($(cup1)+(1,1)$) {\epsilon^{P(\id_{X_2} \times y_2)} \times v_{\exists^P (\id_{X_2} \times y_2)}};

\node[box, box ports south=5, box ports north=5, color=violet] (Box6) at ($(Box2.north.5)+(1,4)$) {e_{\id_{P(X_2 \times Y_2)}} \otimes_{\mathcal{K}} \epsilon^{P(\id_{X_2} \times y_2)}};

\node[coordinate, label=below:{{\scriptsize $P({\pi^{X_2, Y}_{X_2}}) \otimes_{\mathcal{K}} P(\pi^{X_2, Y}_Y)$}}]                          (input1) at ($(Box4.south.3)+(0,-0.5)$) {};
\node[coordinate, label={[text=black]below:{{\scriptsize $\otimes_{P(X_2 \times Y)}$}}}]             (input2) at ($(Box3.south.4)+(0,-4)$) {};
\node[coordinate, label={[text=black]below:{{\scriptsize $\exists^P(\id_{X_2} \times y_2)$}}}]            (input3) at ($(input2)+(3,0)$) {};

\node[coordinate, label={[text=black]above:{{\scriptsize $\id_{PX_2} \otimes_{\mathcal{K}} \exists^P y_2$}}}]            (output1) at ($(Box1.north.1) + (0,7)$) {};
\node[coordinate, label={[text=black]above:{{\scriptsize $P(\pi^{x_2, Y_2}_{X_2}) \otimes_{\mathcal{K}} P(\pi^{X_2, Y_2}_{Y_2})$}}}]  (output2) at ($(Box2.north.2)+(0,4.5)$) {};
\node[coordinate, label={[text=black]above:{{\scriptsize $\otimes_{P(X_2, \times Y_2)}$}}}]           (output3) at ($(output2)+(10,0)$) {};


\node[coordinate] (cup1) at ($(Box3.south.1)+(-1.5,0)$){};

\node[coordinate] (cup2) at ($(Box3.north.4)+(3,0)$){};

\wires[black, looseness=3]{cup1 = {south = Box3.south.1}}{};

\wires[black, looseness=3]{cup2 = {north = Box3.north.4, south=input3.north}}{};

\wires[black, looseness=0]{Box1 = {north.1 = output1.south, north.4 = Box2.south.1}}{};

\wires[black, looseness=1]{Box2 = {north.2 = output2.south}}{};

\wires[black, looseness=0]{Box3 = {north.2 = output3.south, south.4 = input2.north}}{};

\wires[violet, looseness=1]{Box5 = {north.3 = Box6.south.5, south.1 = Box4.north.5, south.5 = cup1.north}}{};

\wires[violet, looseness=1]{Box6 = {south.1 = Box2.north.4}}{};

\wires[red, looseness=1]{Box4 = {south.3 = input1.north, north.1 = Box2.south.3}}{};


\node[coordinate, label={[text=black]right:{{\scriptsize $\id_{P_{X_2}} \otimes_{\mathcal{K}} Py_2$}}}]            (label1) at ($(Box2.south.1) + (0,-1)$) {};

\node[coordinate, label={[text=black]right:{{\scriptsize $P(\id_{X_2} \times y_2)^2$}}}]            (label2) at ($(Box3.south.1) + (0,-1)$) {};

\node[coordinate, label={[text=black]right:{{\tiny $\id_{P(X_2 \times Y_2)} \otimes_{\mathcal{K}} P(\id_{X_2} \times y_2)$}}}]            (label3) at ($(Box6.south.1) + (0,-1)$) {};

\node[coordinate, label={[text=black]right:{{\tiny $\id_{P(X_2 \times Y_2)} \otimes_{\mathcal{K}} \exists^P(\id_{X_2} \times y_2)$}}}]            (label4) at ($(Box6.south.5) + (0.2,-1)$) {};

\node[coordinate, label={[text=black]right:{{\tiny $P(\id_{X_2} \times y_2) \otimes_{\mathcal{K}} \id_{P(X_2 \times Y)}$}}}]            (label5) at ($(Box4.north.5) + (0.1,0.5)$) {};

\end{tikzpicture}}
\end{equation*}

This is a composite of invertible 2-cells. To make it apparent, we recolour the diagram like so:

\begin{equation*}
    \adjustbox{scale=0.7}{\begin{tikzpicture}[baseline=5em, marker/.style={circle, fill, inner sep=1pt, text=white}]
\node[box, box ports south=5, box ports north=5, color=red] (Box1) at (0,0) {e_{\id_{PX_2}} \otimes_{\mathcal{K}} \eta^{Py_2}};

\node[box, box ports south=5, box ports north=5, minimum width=5cm, color=red] (Box2) at ($(Box1.north.5)+(1.5,2)$) {v_{P(\pi^{X_2, Y_2}_{X_2})} \otimes_{\mathcal{K}} P\alpha_2};

\node[box, box ports south=5, box ports north=5, minimum width=5cm, color=blue] (Box3) at ($(Box2.north.5)+(8,0)$) {\mu^{P(\id_{X_2} \times y_2)}};

\node[box, box ports south=5, box ports north=5, color=black] (Box4) at ($(Box2.south.4)+(1,-2)$) {P\alpha_1 \otimes_{\mathcal{K}} v_{P(\pi^{X_2, Y}_Y)}};

\node[box, box ports south=5, box ports north=5, color=blue] (Box5) at ($(cup1)+(-1,2)$) {\epsilon^{P(\id_{X_2} \times y_2)} \times v_{\exists^P (\id_{X_2} \times y_2)}};

\node[box, box ports south=5, box ports north=5, color=red] (Box6) at ($(Box2.north.5)+(1,4)$) {e_{\id_{P(X_2 \times Y_2)}} \otimes_{\mathcal{K}} \epsilon^{P(\id_{X_2} \times y_2)}};

\node[coordinate, label=below:{{\scriptsize $P({\pi^{X_2, Y}_{X_2}}) \otimes_{\mathcal{K}} P(\pi^{X_2, Y}_Y)$}}]                          (input1) at ($(Box4.south.3)+(0,-0.5)$) {};
\node[coordinate, label={[text=black]below:{{\scriptsize $\otimes_{P(X_2 \times Y)}$}}}]             (input2) at ($(Box3.south.4)+(0,-4)$) {};
\node[coordinate, label={[text=black]below:{{\scriptsize $\exists^P(\id_{X_2} \times y_2)$}}}]            (input3) at ($(input2)+(3,0)$) {};

\node[coordinate, label={[text=black]above:{{\scriptsize $\id_{PX_2} \otimes_{\mathcal{K}} \exists^P y_2$}}}]            (output1) at ($(Box1.north.1) + (0,7)$) {};
\node[coordinate, label={[text=black]above:{{\scriptsize $P(\pi^{x_2, Y_2}_{X_2}) \otimes_{\mathcal{K}} P(\pi^{X_2, Y_2}_{Y_2})$}}}]  (output2) at ($(Box2.north.2)+(0,4.5)$) {};
\node[coordinate, label={[text=black]above:{{\scriptsize $\otimes_{P(X_2, \times Y_2)}$}}}]           (output3) at ($(output2)+(10,0)$) {};


\node[coordinate] (cup1) at ($(Box3.south.1)+(-1.5,0)$){};

\node[coordinate] (cup2) at ($(Box3.north.4)+(3,0)$){};

\wires[blue, looseness=3]{cup1 = {south = Box3.south.1}}{};

\wires[blue, looseness=3]{cup2 = {north = Box3.north.4, south=input3.north}}{};

\wires[red, looseness=0]{Box1 = {north.1 = output1.south, north.4 = Box2.south.1}}{};

\wires[red, looseness=1]{Box2 = {north.2 = output2.south}}{};

\wires[blue, looseness=0]{Box3 = {north.2 = output3.south, south.4 = input2.north}}{};

\wires[blue, looseness=1]{Box5 = {north.3 = Box6.south.5, south.1 = Box4.north.5, south.5 = cup1.north}}{};

\wires[red, looseness=1]{Box6 = {south.1 = Box2.north.4}}{};

\wires[black, looseness=1]{Box4 = {south.3 = input1.north, north.1 = Box2.south.3}}{};


\node[coordinate, label={[text=black]right:{{\scriptsize $\id_{P_{X_2}} \otimes_{\mathcal{K}} Py_2$}}}]            (label1) at ($(Box2.south.1) + (0,-1)$) {};

\node[coordinate, label={[text=black]right:{{\scriptsize $P(\id_{X_2} \times y_2)^2$}}}]            (label2) at ($(Box3.south.1) + (0,-1)$) {};

\node[coordinate, label={[text=black]right:{{\tiny $\id_{P(X_2 \times Y_2)} \otimes_{\mathcal{K}} P(\id_{X_2} \times y_2)$}}}]            (label3) at ($(Box6.south.1) + (0,-1)$) {};

\node[coordinate, label={[text=black]right:{{\tiny $\id_{P(X_2 \times Y_2)} \otimes_{\mathcal{K}} \exists^P(\id_{X_2} \times y_2)$}}}]            (label4) at ($(Box6.south.5) + (0.5,-0.5)$) {};

\node[coordinate, label={[text=black]right:{{\tiny $P(\id_{X_2} \times y_2) \otimes_{\mathcal{K}} \id_{P(X_2 \times Y)}$}}}]            (label5) at ($(Box4.north.5) + (0.1,0.5)$) {};
\end{tikzpicture}}
\end{equation*}

The red part is $v_{P(\pi^{X_2, Y_2}_{X_2})} \otimes_{\mathcal{K}} \mathcal{B}_{\alpha_2}^{-1}$, where $\mathcal{B}_{\alpha_2}$ is the Beck-Chevalley 2-cell for the $\mathfrak{C}$-pullback $\alpha_2$ (note that $y_2 \in \mathfrak{C}^r$ and the left class contains all projections since $\mathfrak{C}$ is a cartesian adequate triple). The blue part is the inverse of the Frobenius $2$-cell $\mathfrak{F}^r_{\id_{X_2} \times y_2}$, and the black part is invertible as $P(\alpha_1)$ is the image of an identity $2$-cell. We can similarly decompose $\mu^P_{x_2, \id_Y}$ in terms of the inverse of $\mathfrak{F}^l_{x_2, \id_Y}$ and the Beck-Chevalley cell for the $\mathfrak{C}$-pullback 
\begin{equation*}
    \begin{tikzcd}
	{X \times Y} && X \\
	\\
	{X_2 \times Y} && {X_2}
	\arrow[""{name=0, anchor=center, inner sep=0}, "{\pi^{X, Y}_X}", from=1-1, to=1-3]
	\arrow["{x_2 \times \id_Y}"', from=1-1, to=3-1]
	\arrow["\lrcorner"{anchor=center, pos=0.125}, draw=none, from=1-1, to=3-3]
	\arrow["{x_2}", from=1-3, to=3-3]
	\arrow[""{name=1, anchor=center, inner sep=0}, "{\pi^{X_2, Y}_{X_2}}"', from=3-1, to=3-3]
	\arrow["{\beta_1}"{description}, color={rgb,255:red,214;green,92;blue,92}, draw=none, from=0, to=1]
\end{tikzcd}.
\end{equation*}

\end{proof}

\section{Double functors as hyperdoctrines}
\label{section:DoublePseudofunctorsAsDoctrines}
This section is dedicated to proving a converse result: under mild conditions on the adequate triple, the tight component of a suitable double functor \(Q\colon  \Span{\mathfrak{C}}^{\op} \to  \Qt{\mathcal{K}}\) is the underlying 2-functor
of a $\mathfrak{C}$-regular $\mathcal{K}$-hyperdoctrine. For the sake of simplicity, we shall prove it for strict double functors --- 
this is not too egregious since there are strictification results for pseudo double functors (see Theorem 7.5 in \cite{grandis-1999-limits}). Note that the strictification of a pseudo double functor has a different domain, but it is  equivalent to the original one via pseudo double functors.

The strategy is simple: we observe that \(\Span{\mathfrak{C}}^{\op}\) admits companions precisely for those tight morphisms $f^{\op}$ with
 \(f\) in the class \(\mathfrak{C}^l\) (namely the span \(f^*\coloneqq  B \xleftarrow {f} A = A\)), and conjoints for those \(f\) in the class \(\mathfrak{C}^r\) (the span \(f_! \coloneqq  A = A \xrightarrow {f} B\)). 
 These are preserved by double functors, and therefore the images of morphisms
 in  \(\mathfrak{C}^r\) will yield a conjunction in $\Qt{\mathcal{K}}$, i.e., an adjunction in \(\mathcal{K}\). That provides existential quantification. The fact that Beck-Chevalley holds
 for the appropriate maps is encoded by the loose composition of spans --- this is well-known, but we state it here nonetheless since there is the addition of the adequate triple to contend with.
\begin{proposition}
\label{Proposition:Preservation of companions and conjoints.}
\par{}Let \(Q\colon  \mathbb {C} \to  \mathbb {D}\) be a normal lax double functor (i.e, one that strictly preserves identities). Then \(Q\) maps companion pairs to companion pairs and conjunctions to conjunctions.\begin{proof}\par{}See Proposition 3.8 in \cite{dawson-2010-span} .\end{proof}\end{proposition}

\begin{proposition}
\label{prop:Double functor induces existential structure}
\par{}For any adequate triple $\mathfrak{C}$, the tight component \(Q_0 \colon  \mathsf {C}^{^{\mathsf {op}}} \to  \mathcal {K}\) of a double functor \(Q \colon  \Span{\mathfrak{C}}^{\op} \to  \Qt{\mathcal{K}}\) is a \(\mathfrak{C}\)-existential $\mathcal{K}$-structure.

\begin{proof}  Define \(\exists^Q  f \colon  QA \to  QB\) as \(Q(f_!)\)  (as a morphism in \(\mathcal {K}\)) for each \(f \in  \mathfrak{C}^r\), so that \(\exists^Q  f \dashv  Q_0(f)\) in \(\mathcal {K}\) by Proposition~\ref{Proposition:Preservation of companions and conjoints.}. We must show that these satisfy the 
Beck-Chevalley property with respect to the adequate triple $\mathfrak{C}$.\\

Out of a  
$\mathfrak{C}$-pullback 
\begin{center}\begin{tikzcd}
	A && X \\
	\\
	B && Y
	\arrow[""{name=0, anchor=center, inner sep=0}, "h", from=1-1, to=1-3]
	\arrow["f"', from=1-1, to=3-1]
	\arrow["\lrcorner"{anchor=center, pos=0.125}, draw=none, from=1-1, to=3-3]
	\arrow["g", from=1-3, to=3-3]
	\arrow[""{name=1, anchor=center, inner sep=0}, "k"', from=3-1, to=3-3]
	\arrow["\alpha"{description}, color={rgb,255:red,92;green,92;blue,214}, draw=none, from=0, to=1]
\end{tikzcd},\end{center}
\noindent define a square
\begin{equation*}
    \overline{\mathcal{B}}_{\alpha} \coloneqq \begin{tikzcd}
	X & Y & B & B \\
	X & X & A & B
	\arrow[""{name=0, anchor=center, inner sep=0}, "{{g_!}}"{inner sep=.8ex}, "\shortmid"{marking}, from=1-1, to=1-2]
	\arrow[equals, from=1-1, to=2-1]
	\arrow[""{name=1, anchor=center, inner sep=0}, "{{k^*}}"{inner sep=.8ex}, "\shortmid"{marking}, from=1-2, to=1-3]
	\arrow["{{g^{\op}}}", from=1-2, to=2-2]
	\arrow[""{name=2, anchor=center, inner sep=0}, "\shortmid"{marking}, equals, from=1-3, to=1-4]
	\arrow["{{f^{\op}}}", from=1-3, to=2-3]
	\arrow[equals, from=1-4, to=2-4]
	\arrow[""{name=3, anchor=center, inner sep=0}, "\shortmid"{marking}, equals, from=2-1, to=2-2]
	\arrow[""{name=4, anchor=center, inner sep=0}, "{{h^*}}"'{inner sep=.8ex}, "\shortmid"{marking}, from=2-2, to=2-3]
	\arrow[""{name=5, anchor=center, inner sep=0}, "{{f_!}}"'{inner sep=.8ex}, "\shortmid"{marking}, from=2-3, to=2-4]
	\arrow["{{\hat{\eta}_g}}"{inner sep=.8ex}, ""{marking, text={rgb,255:red,214;green,92;blue,92}}, color={rgb,255:red,214;green,92;blue,92}, between={0.2}{0.8}, Rightarrow, from=3, to=0]
	\arrow["f"', color={rgb,255:red,153;green,92;blue,214}, between={0.2}{0.8}, Rightarrow, from=4, to=1]
	\arrow["{{\hat{\epsilon}_f}}"', color={rgb,255:red,214;green,92;blue,92}, between={0.2}{0.8}, Rightarrow, from=5, to=2]
\end{tikzcd}
\end{equation*}

\noindent in $\Span{\mathfrak{C}}^{\op}$. Now, note that the composite of span morphisms
\begin{equation*}
    \begin{tikzcd}
	Y & Y \\
	Y & B \\
	X & A \\
	A & A
	\arrow[""{name=0, anchor=center, inner sep=0}, "\shortmid"{marking}, equals, from=1-1, to=1-2]
	\arrow[equals, from=1-1, to=2-1]
	\arrow["{k^{\op}}", from=1-2, to=2-2]
	\arrow[""{name=1, anchor=center, inner sep=0}, "{k^*}"{inner sep=.8ex}, "\shortmid"{marking}, from=2-1, to=2-2]
	\arrow["{g^{\op}}"', from=2-1, to=3-1]
	\arrow["{f^{\op}}", from=2-2, to=3-2]
	\arrow[""{name=2, anchor=center, inner sep=0}, "{h^*}"'{inner sep=.8ex}, "\shortmid"{marking}, from=3-1, to=3-2]
	\arrow["{h^{\op}}"', from=3-1, to=4-1]
	\arrow[equals, from=3-2, to=4-2]
	\arrow[""{name=3, anchor=center, inner sep=0}, "\shortmid"{marking}, equals, from=4-1, to=4-2]
	\arrow["{\psi_k}", color={rgb,255:red,92;green,92;blue,214}, between={0.4}{0.8}, Rightarrow, from=1, to=0]
	\arrow["f"', color={rgb,255:red,153;green,92;blue,214}, between={0.2}{0.8}, Rightarrow, from=2, to=1]
	\arrow["{\phi_h}", color={rgb,255:red,92;green,92;blue,214}, between={0.2}{0.6}, Rightarrow, from=3, to=2]
\end{tikzcd}
\end{equation*}
 is just $e_{(kf)^{\op}}$  (as $kf=gh$), so that composite square is mapped by $Q$ to the quintet square representing the commutative diagram $Q_0(\alpha)$ in $\mathcal{K}$. However, the images of the conjoint squares $\psi_k$ and $\phi_h$ in $\Qt{\mathcal{K}}$ are identities in $\mathcal{K}$, so 
 \begin{equation*}
     \begin{tikzcd}
	Y & B \\
	X & A
	\arrow[""{name=0, anchor=center, inner sep=0}, "{{k^*}}"{inner sep=.8ex}, "\shortmid"{marking}, from=1-1, to=1-2]
	\arrow["{{g^{\op}}}"', from=1-1, to=2-1]
	\arrow["{{f^{\op}}}", from=1-2, to=2-2]
	\arrow[""{name=1, anchor=center, inner sep=0}, "{{h^*}}"'{inner sep=.8ex}, "\shortmid"{marking}, from=2-1, to=2-2]
	\arrow["f"', color={rgb,255:red,153;green,92;blue,214}, between={0.2}{0.8}, Rightarrow, from=1, to=0]
\end{tikzcd}
 \end{equation*}
 \noindent itself must be mapped to $Q_0(\alpha)$ as a 2-cell in $\mathcal{K}$.
 We conclude that the inverse of $Q(\overline{\mathcal{B}}_{\alpha})$ (if it exists) is the Beck-Chevalley 2-cell $\mathcal{B}_{\alpha}$  of the $\mathfrak{C}$-pullback $\alpha$. We show such an inverse exists by inverting $\overline{\mathcal{B}_{\alpha}}$ itself.

We can build this inverse by viewing the $\mathfrak{C}$-pullback as the inner part of a loose composite 
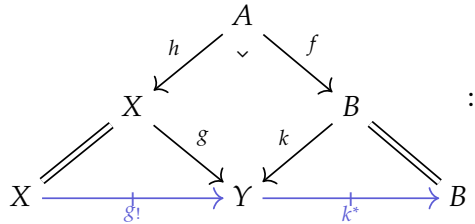
\begin{center}\begin {tikzcd}[cramped]
	&& A \\
	& X && B \\
	X && Y && B
	\arrow ["h"', from=1-3, to=2-2]
	\arrow ["f", from=1-3, to=2-4]
	\arrow ["\lrcorner "{anchor=center, pos=0.125, rotate=-45}, draw=none, from=1-3, to=3-3]
	\arrow [equal, from=2-2, to=3-1]
	\arrow ["g", from=2-2, to=3-3]
	\arrow ["k"', from=2-4, to=3-3]
	\arrow [equal, from=2-4, to=3-5]
	\arrow ["{g_!}"', "\shortmid "{marking, text={rgb,255:red,92;green,92;blue,214}}, color={rgb,255:red,92;green,92;blue,214}, from=3-1, to=3-3]
	\arrow ["{k^*}"', "\shortmid "{marking, text={rgb,255:red,92;green,92;blue,214}}, color={rgb,255:red,92;green,92;blue,214}, from=3-3, to=3-5]
\end {tikzcd}.\end{center}

 \noindent There is a unique isomorphism of spans \(\gamma _1 \colon  g_! \odot  k^* \to  (X \xleftarrow {h} A \xrightarrow {f} B)\)
from the universal property of the $\mathfrak{C}$-pullback, and a unique isomorphism of spans \(\gamma _2 \colon  (X \xleftarrow {h} A \xrightarrow {f} B) \to  h^* \odot  f_!\).
The composite \(\gamma _2 \circ  \gamma _1\) then provides the necessary inverse morphism of spans.
\end{proof}\end{proposition}

As a final remark on companions and conjoints, we observe that the components of a tight transformation of double functors into $\Qt{\mathcal{K}}$ at companions and conjoints are determined by the tight components:

\begin{proposition}
\label{proposition: tight transformation determination}
    Let $Q, Q' \colon \mathbb{A} \to \Qt{\mathcal{K}}$ be normal pseudofunctors and $\rho \colon Q \to Q'$ be a tight transformation. For any tight morphism $B \to A \in \mathbb{A}$, we have that $\rho_{f^*} = \rho_f^{-1}$ and $\rho_{f_!} = \mathrm{mate}(\rho_f)$.
\end{proposition}

\begin{proof}
    Cell naturality of $\rho$ with respect to the companion square $\psi_f$ gives
    \begin{equation*}
        \begin{tikzcd}
	{QB } && QB && {QB } && QB && QB \\
	\\
	QB && QA && {Q'B} & {=} & {Q'B} && {Q'B} \\
	\\
	{Q'B} && {Q'A} && {Q'A} && {Q'B} && {Q'A}
	\arrow[""{name=0, anchor=center, inner sep=0}, "\shortmid"{marking}, equals, from=1-1, to=1-3]
	\arrow[equals, from=1-1, to=3-1]
	\arrow[""{name=1, anchor=center, inner sep=0}, "\shortmid"{marking}, equals, from=1-3, to=1-5]
	\arrow["Qf", from=1-3, to=3-3]
	\arrow["{\rho_B}", from=1-5, to=3-5]
	\arrow[""{name=2, anchor=center, inner sep=0}, "\shortmid"{marking}, equals, from=1-7, to=1-9]
	\arrow["{{\rho_B}}"', from=1-7, to=3-7]
	\arrow["{\rho_B}", from=1-9, to=3-9]
	\arrow[""{name=3, anchor=center, inner sep=0}, "{Qf^* }"{inner sep=.8ex}, "\shortmid"{marking}, from=3-1, to=3-3]
	\arrow["{\rho_B}"', from=3-1, to=5-1]
	\arrow["{\rho_A}", from=3-3, to=5-3]
	\arrow["{Q'f}", from=3-5, to=5-5]
	\arrow[""{name=4, anchor=center, inner sep=0}, "\shortmid"{marking}, equals, from=3-7, to=3-9]
	\arrow[equals, from=3-7, to=5-7]
	\arrow["{Q'f}", from=3-9, to=5-9]
	\arrow[""{name=5, anchor=center, inner sep=0}, "{Q'f^*}"'{inner sep=.8ex}, "\shortmid"{marking}, from=5-1, to=5-3]
	\arrow[""{name=6, anchor=center, inner sep=0}, "\shortmid"{marking}, equals, from=5-3, to=5-5]
	\arrow[""{name=7, anchor=center, inner sep=0}, "{{Q'f^*}}"'{inner sep=.8ex}, "\shortmid"{marking}, from=5-7, to=5-9]
	\arrow["{{Q\psi_f}}"{description}, color={rgb,255:red,214;green,92;blue,92}, draw=none, from=0, to=3]
	\arrow["{\rho_f}"{description}, color={rgb,255:red,92;green,92;blue,214}, draw=none, from=1, to=6]
	\arrow["{{v_{\rho_B}}}"{description}, color={rgb,255:red,214;green,92;blue,214}, draw=none, from=2, to=4]
	\arrow["{\rho_{f^*}}"{description}, color={rgb,255:red,92;green,92;blue,214}, draw=none, from=3, to=5]
	\arrow["{{Q'\psi_{f}}}"{description}, color={rgb,255:red,214;green,92;blue,92}, draw=none, from=4, to=7]
\end{tikzcd}
    \end{equation*}

\noindent in $\Qt{\mathcal{K}}$. But $Q$ and $Q'$ are normal, therefore preserve companion pairs (by Proposition~\ref{Proposition:Preservation of companions and conjoints.}). Since companion squares in $\Qt{\mathcal{K}}$ are identities in $\mathcal{K}$, we conclude that $\rho_{f^*} = \rho^{-1}_f$. Similarly, using naturality of $\rho$ with respect to the conjoint square $\hat{\eta}_f$ and the fact that $Q$ and $Q'$ preserve conjoints, we deduce that $\rho_{f_!} = \mathrm{mate}(\rho_f)$.
\end{proof}

By factorising loose morphisms in terms of companions and conjoints and using compositional coherence of tight transformations, we conclude the following:

\begin{corollary}
\label{corollary;tight transformations determined by tight component}
    Let $\mathbb{A}$ be a double category where every loose morphism is a composite of companions and conjoints. Then any tight transformation of normal double functors $Q \xrightarrow{\rho} Q' \colon \mathbb{A} \to \Qt{\mathcal{K}}$ is determined by its tight component.
\end{corollary}

This applies to $\Span{\mathfrak{C}}^{\op}$, of course: any $\mathfrak{C}$-span $X_1 \xleftarrow{x_1} X \xrightarrow{x_2}$ is canonically factorisable as $x_1^* \odot (x_2)_!$. In particular, this means that determining whether the monoidal laxator $\mu$ of a double functor is a companion commuter transformation can be reduced to checking if each mate of a square $\mu_{\id_{\cod(f)}, f}$ is a companion commuter (noting that in this case so are the mates of squares $\mu_{f, \cod(f)}$, using the symmetry):

\begin{corollary}
\label{corollary: reduced commuter condition}
    Let $\mathfrak{C}$ be a cartesian adequate triple, $\mathcal{K}$ be a symmetric monoidal 2-category, and $Q \colon \Span{\mathfrak{C}}^{\op} \to \Qt{\mathcal{K}}$ be a lax symmetric monoidal double functor with monoidal laxator $\mu$. If the  squares $\mathrm{mate}(\mu_{\id_{\cod f}, f})$ are companion commuters for every tight morphism $f$, then $\mu$ is a companion commuter transformation.
\end{corollary}


We now consider the monoidal structure and the Frob\"enius laws. If a double functor \(Q\colon  \Span{\mathfrak{C}}^{\op} \to  \Qt{\mathcal{K}}\) is  lax symmetric monoidal,
then the tight components \(\mu _0\) and \(I_0\) of its monoidal laxator provide a lax monoidal structure for \(Q_0 \colon  \mathsf{C}^{\op} \to  \mathcal {K}\).
If the adequate triple $\mathfrak{C}$ is cartesian and $\mathcal{K}$ is a symmetric monoidal 2-category, then we can transfer the ``external'' monoidal structure of \(Q_0\) to the objects \(QA\), making them symmetric pseudomonoids by Proposition~\ref{Proposition:Internal vs External monoidal structures}. 
All that is left to show is that the Frob\"enius reciprocity property holds with respect to these monoidal structures on the fibres, i.e., for every \(f \colon  A \to  B\) in \(\mathfrak{C}^r\), we seek to invert the canonical cell \(\exists  f \otimes _A (Qf \times  \mathrm{id}_{QA}) \to  \otimes _B (\mathrm{id}_{QB} \times  \exists f)\) in $\mathcal{K}$ given by 
 
  \begin{equation*}
      \adjustbox{scale=0.7}
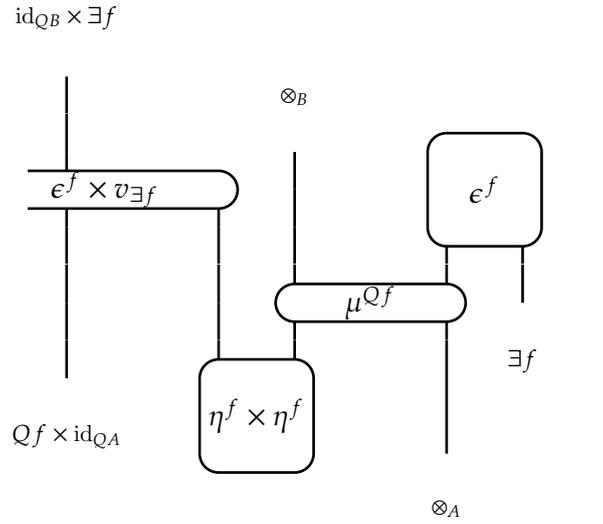
{\begin{tikzpicture}[baseline=6em, marker/.style={circle, fill, inner sep=1pt, text=white}]
\node[coordinate, label=below:{{\scriptsize $Qf \times \id_{QA}$}}] (input1) at (0,0) {};
\node[coordinate, label=below:{{\scriptsize $\otimes_{QA}$}}] (input2) at ($(input1)+(4,0)$) {};
\node[coordinate, label=below:{{\scriptsize $\exists f$}}] (input3) at ($(input2)+(1.5,0)$)  {};

\node[coordinate, label=above:{{ $\id_{QB} \times \exists f$}}] (output1) at (0,7) {};
\node[coordinate, label=above:{{ $\otimes_{QB}$}}] (output2) at ($(output1)+(2.5,0)$) {};
\node[coordinate,  label=above:{}] (output3) at ($(output2)+(2, 0)$) {};

\node[box, box ports south=5, box ports north=5] (Box1) at ($(input1)+(1, 5)$) {\epsilon^{Qf} \times v_{\exists f}};
\node[box, box ports south=5, box ports north=5, minimum width=6 em] (Box2) at ($(output3)+(-1, -4)$) {\mu^{Qf}};

\wires[black]{Box1 = {south.2 = input1.north, north.2=output1.south, }, Box2={}}{};


\node[coordinate] (cup1) at ($(Box2.south.1)+(-1,0)$) {};
\wires[black, looseness=1.5]{cup1 ={south = Box2.south.2, north= Box1.south.4} }{}; 

\node[coordinate] (cap1) at ($(Box2.north.5)+(1,0)$) {};
\wires[black, looseness=2]{cap1 ={north = Box2.north.5, south=input3.north}, Box2= {north.1 = output2.south, south.4 = input2.north} }{};

\end{tikzpicture}}
  \end{equation*}

 \noindent where
\(\mu ^{Qf}\) is the monoidal laxator of the strong monoidal functor \((QB, \otimes _B, I_B) \xrightarrow{Qf}  (QA, \otimes _A, I_A)\).  Under a mild assumption on $\mathfrak{C}$, we can construct this inverse out of \(\mu \) (using that it is a companion commuter transformation) and Beck-Chevalley cells:

\begin{lemma}
\label{lemma:Companion commuter implies hyperdoctrine}
\par{}Let $\mathfrak{C}$ be a cartesian adequate triple such that $\mathfrak{C}^l$ includes all diagonals, $\mathcal{K}$ be a symmetric monoidal 2-category, and $ Q\colon \Span{\mathfrak{C}}^{\op} \to \Qt{\mathcal{K}}$ be a lax symmetric monoidal double functor with companion commuter monoidal laxators. Then \(Q_0\) is a $\mathfrak{C}$-regular $\mathcal{K}$-hyperdoctrine.

\begin{proof}
We already know that $Q_0$ is a $\mathfrak{C}$-existential $\mathcal{K}$-structure from Proposition~\ref{prop:Double functor induces existential structure}, so all that's left is to check the Frob\"enius law. \\

Let $f \colon A \to B \in \mathfrak{C}^r$. Recall from Proposition~\ref{Proposition:Internal vs External monoidal structures} that 
\begin{center}\begin {tikzcd}[cramped]
	{QB \otimes  QB} & QB \\
	{QA \otimes  QA} & QA
	\arrow [""{name=0, anchor=center, inner sep=0}, "{\otimes_{QB}}", from=1-1, to=1-2]
	\arrow ["{Qf\otimes  Qf}"', from=1-1, to=2-1]
	\arrow ["Qf", from=1-2, to=2-2]
	\arrow [""{name=1, anchor=center, inner sep=0}, "{\otimes_{QA}}"', from=2-1, to=2-2]
	\arrow ["{\mu^{Qf}}"', color={rgb,255:red,92;green,92;blue,214}, shorten <=4pt, shorten >=4pt, Rightarrow, from=1, to=0]
\end {tikzcd} = \begin {tikzcd}[cramped]
	{QB \otimes  QB} & {Q(B\times  B)} && QB \\
	{QA \otimes  QA} & {Q(A\times  A)} && QA
	\arrow [""{name=0, anchor=center, inner sep=0}, "{\mu_{B,B}}", from=1-1, to=1-2]
	\arrow ["{Qf\otimes  Qf}"', from=1-1, to=2-1]
	\arrow [""{name=1, anchor=center, inner sep=0}, "{Q(\Delta _B)}", from=1-2, to=1-4]
	\arrow ["{Q(f \times  f)}", from=1-2, to=2-2]
	\arrow ["Qf", from=1-4, to=2-4]
	\arrow [""{name=2, anchor=center, inner sep=0}, "{\mu_{A,A}}"', from=2-1, to=2-2]
	\arrow [""{name=3, anchor=center, inner sep=0}, "{Q(\Delta _A)}"', from=2-2, to=2-4]
	\arrow ["{(\mu _0)_{f,f}}"', color={rgb,255:red,92;green,92;blue,214}, shorten <=4pt, shorten >=4pt, Rightarrow, from=2, to=0]
	\arrow ["Q(\Delta_f) "{description}, color={rgb,255:red,92;green,92;blue,214}, draw=none, from=3, to=1]
\end {tikzcd}.\end{center}
  By Proposition~\ref{proposition: tight transformation determination}, \((\mu _0)_{f,f} = \mu ^{-1}_{f^*, f^*}\) as cells in \(\mathcal {K}\), so we need to find an inverse to 

  \begin{equation*}
    \nabla\coloneqq  \  \adjustbox{scale=0.7}{\begin{tikzpicture}[baseline=6em, marker/.style={circle, fill, inner sep=1pt, text=white}]

\node[box, box ports south=2, box ports north=1, minimum width=8 em] (Box1) at ($(0,2)$) {\epsilon^{Qf} \otimes v_{\exists f}};

\node[box, box ports north=2, minimum width=6 em] (Box2) at ($(2,0)$) {\eta^{Qf} \otimes \eta^{Qf}};

\node[box, box ports north=2, box ports south=2, minimum width=6 em] (Box3) at ($(4,2)$) {\mu^{-1}_{f^*,f^*}};

\node[box, box ports north=2, box ports south=2, minimum width=6 em] (Box4) at ($(6,4)$) {Q(\Delta_f)};

\node[box, box ports north=2, box ports south=2, minimum width=6 em] (Box5) at ($(8,6)$) {Q(\hat{\epsilon}_f)=\epsilon^{Qf}};
    
\node[coordinate, label=below:{{\scriptsize $Qf \otimes \id_{QA}$}}] (input1) at ($(Box1.south.1)+(0,-2)$) {};

\node[coordinate, label=below:{{\scriptsize $\mu_{A,A}$}}] (input2) at ($(Box3.south.2)+(0,-2)$) {};

\node[coordinate, label=below:{{\scriptsize $Q(\Delta_A)$}}] (input3) at ($(Box4.south.2)+(0, -4)$)  {};

\node[coordinate, label=below:{{\scriptsize $\exists f$}}] (input4) at ($(Box5.south.2)+(0, -6)$)  {};

\node[coordinate, label=above:{{\scriptsize $\id_{QB} \otimes \exists f$}}] (output1) at ($(Box1.north.1)+(0,4)$) {};

\node[coordinate, label=above:{{\scriptsize $\mu_{B,B}$}}] (output2) at ($(Box3.north.1)+(0,4)$) {};

\node[coordinate,  label=above:{{\scriptsize $Q(\Delta_B)$}}] (output3) at ($(Box4.north.1)+(0, 2)$) {};

\wires[black]{Box1 ={north.1=output1.south, south.1=input1.north, south.2=Box2.north.1}}{}

\wires[black]{Box3 ={north.1=output2.south, north.2 = Box4.south.1, south.1=Box2.north.2, south.2=input2.north}}{}

\wires[black]{Box4 ={north.1=output3.south, south.2=input3.north, north.2= Box5.south.1}}{}

\wires[black]{Box5 ={south.2 = input4.north}}{}


\end{tikzpicture}}
  \end{equation*}

Now, using the loose compositional coherence of $\mu$ to split $\mu^{-1}_{f^*, f^*}$ as $\mu ^{-1}_{e_B, f^*} \mid  \mu ^{-1}_{f^*, e_A}$ and resorting to the functoriality of mates, we can write $\nabla$ as

  \begin{equation*}
     \adjustbox{scale=0.7}{\begin{tikzpicture}[baseline=6em, marker/.style={circle, fill, inner sep=1pt, text=white}]

\node[box, box ports south=2, box ports north=1, minimum width=8 em] (Box1) at ($(0,2)$) {\epsilon^{Qf} \otimes v_{\exists f}};

\node[box, box ports north=2, minimum width=6 em] (Box2) at ($(2,0)$) {\eta^{Qf} \otimes \eta^{Qf}};

\node[box, box ports north=2, box ports south=1, minimum width=6 em] (Box3) at ($(3.5,1.5)$) {=};

\node[box, box ports north=2, box ports south=2, minimum width=6 em, color=blue] (Box4) at ($(10,9)$) {Q(\Delta_f)};

\node[box, box ports north=2, box ports south=2, minimum width=6 em, color=blue] (Box5) at ($(12,11)$) {Q(\hat{\epsilon}_f)=\epsilon^{Qf}};

\node[box, box ports north=2, box ports south=2, minimum width=6 em] (BoxA) at ($(7,3)$) {\mu^{-1}_{f^*, e_A}};

\node[box, box ports north=2, box ports south=2, minimum width=6 em, color=red] (BoxB) at ($(5.5,5)$) {\mu^{-1}_{e_B, f^*}};

\node[box, box ports north=1, box ports south=2, minimum width=6 em, color=blue] (BoxC) at ($(9,7)$) {Q(=)};
    
\node[coordinate, label=below:{{\scriptsize $Qf \otimes \id_{QA}$}}] (input1) at ($(Box1.south.1)+(0,-2)$) {};

\node[coordinate, label=below:{{\scriptsize $\mu_{A,A}$}}] (input2) at ($(BoxA.south.2)+(0,-3)$) {};

\node[coordinate, label=below:{{\scriptsize $Q(\Delta_A)$}}] (input3) at ($(Box4.south.2)+(0, -9)$)  {};

\node[coordinate, label=below:{{\scriptsize $\exists f$}}] (input4) at ($(Box5.south.2)+(0, -11)$)  {};

\node[coordinate, label=above:{{\scriptsize $\id_{QB} \otimes \exists f$}}] (output1) at ($(Box1.north.1)+(0,9)$) {};

\node[coordinate, label=above:{{\scriptsize $\mu_{B,B}$}}] (output2) at ($(BoxB.north.1)+(0,6)$) {};

\node[coordinate,  label=above:{{\scriptsize $Q(\Delta_B)$}}] (output3) at ($(Box4.north.1)+(0, 2)$) {};

\wires[black]{Box1 ={north.1=output1.south, south.1=input1.north, south.2=Box2.north.1}}{}

\wires[black]{Box3 ={south.1=Box2.north.2}}{}

\wires[blue]{Box4 ={north.1=output3.south, south.2=input3.north, north.2= Box5.south.1, south.1=BoxC.north.1}}{}

\wires[blue]{Box5 ={south.2 = input4.north}}{}

\wires[black]{BoxA ={south.1 = Box3.north.2, south.2=input2.north}}{}

\wires[red]{BoxB={north.1=output2.south, south.2=BoxA.north.1}}{}

\wires[blue]{BoxA={north.2=BoxC.south.2}}{}


\node[coordinate] (cap1) at ($(Box3.north.1)+(0,2)$) {};

\node[coordinate] (cup1) at ($(cap1)+(1,0)$) {};

\wires[black, looseness=2]{cap1 ={south = Box3.north.1, north=cup1.north}}{}

\wires[red, looseness=2]{cup1 ={south = BoxB.south.1}}{}

\node[coordinate] (cap2) at ($(BoxB.north.2)+(0,1)$) {};

\node[coordinate] (cup2) at ($(cap2)+(1,0)$) {};

\wires[red, looseness=2]{cap2 ={south = BoxB.north.2, north=cup2.north}}{};

\wires[blue, looseness=2]{cup2 = {south = BoxC.south.1}}{};

\node[coordinate, label=left:{\scriptsize $\id_{QB} \otimes Qf$}] (label1) at ($(Box3.north.1)+(0,1)$){};

\node[coordinate, label=below:{\scriptsize $Qf \otimes \id_{QA}$}] (label2) at ($(Box3.north.2)+(1.5,0.2)$){};

\node[coordinate, label=right:{\scriptsize $\mu_{B,A}$}] (label3) at ($(BoxB.south.2)+(0.2,-0.4)$){};

\node[coordinate, label=right:{\scriptsize $Q(f \times \id_A)$}] (label4) at ($(BoxA.north.2)+(0.5,1)$){};

\node[coordinate, label=left:{\scriptsize $Q(f \times f)$}] (label5) at ($(BoxC.north.1)+(-0.1,0.7)$){};

\node[coordinate, label=right:{\scriptsize $Qf$}] (label6) at ($(Box4.north.2)+(0.5,0.4)$){};

\node[coordinate, label=left:{\scriptsize $Q(\id_B \times f)$}] (label7) at ($(BoxB.north.2)+(0.4,1.8)$){};

\end{tikzpicture}}
  \end{equation*}

Now, since $\mathfrak{C}^l$ contains all diagonals, considering the graph of $f$ yields a $\mathfrak{C}$-pullback 

\begin{equation*}
    \begin{tikzcd}
	A & {A\times A} && {B \times A} \\
	\\
	B &&& {B \times B}
	\arrow["{\Delta_A}", from=1-1, to=1-2]
	\arrow[""{name=0, anchor=center, inner sep=0}, "f"', from=1-1, to=3-1]
	\arrow["\lrcorner"{anchor=center, pos=0.125}, draw=none, from=1-1, to=3-4]
	\arrow["{f \times \id_A}", from=1-2, to=1-4]
	\arrow[""{name=1, anchor=center, inner sep=0}, "{\id_B \times f}", from=1-4, to=3-4]
	\arrow["{\Delta_B}"', from=3-1, to=3-4]
	\arrow["\Gamma"{description}, color={rgb,255:red,92;green,92;blue,214}, draw=none, from=0, to=1]
\end{tikzcd}.
\end{equation*}

The blue region in the diagram is the mate of $Q(\Gamma)$, and hence is $\mathcal{B}_{\Gamma}^{-1}$. The red component of the diagram is the mate of $\mu_{e_B, f^*}^{-1}$, which is $\mu_{e_B, f_!}$ by Proposition~\ref{proposition: tight transformation determination}. So that, too, is invertible (as $\mu_{e_B, f_!}$ is a companion commuter square in $\Qt{\mathcal{K}}$). The rest of the diagram simplifies through the triangle identities, and we conclude that $\mathcal {F}^r_f$ can be constructed from $\mu_0$ and the Beck-Chevalley 2-cell for $\Gamma$ as   the following pasting diagram in $\mathcal{K}$:

\begin{equation*}
    \begin{tikzcd}
	& {QB \otimes QA} && {QB \otimes QB} && {Q(B \times B)} & QB \\
	{\mathcal{F}^r_f \coloneqq} & {QB \otimes QA} && {Q(B \times A)} && {Q(B \times B)} & QB \\
	& {QB \otimes QA} && {Q(B \times A)} & {Q(A\  \times A)} & QA & QB \\
	& {QB \times QB} && {QA\otimes QA} & {Q(A \times A)} & QA & QB
	\arrow["{\id_{QB} \otimes \exists f}", from=1-2, to=1-4]
	\arrow[""{name=0, anchor=center, inner sep=0}, equals, from=1-2, to=2-2]
	\arrow["{\mu_{B,B}}", from=1-4, to=1-6]
	\arrow["{Q(\Delta_B)}", from=1-6, to=1-7]
	\arrow[""{name=1, anchor=center, inner sep=0}, equals, from=1-6, to=2-6]
	\arrow[equals, from=1-7, to=2-7]
	\arrow["{\mu_{B,A}}"', from=2-2, to=2-4]
	\arrow[equals, from=2-2, to=3-2]
	\arrow["{\exists(\id_B \times f)}", from=2-4, to=2-6]
	\arrow[""{name=2, anchor=center, inner sep=0}, equals, from=2-4, to=3-4]
	\arrow["{Q(\Delta_B)}"', from=2-6, to=2-7]
	\arrow[""{name=3, anchor=center, inner sep=0}, equals, from=2-7, to=3-7]
	\arrow["{\mu_{B,A}}", shift left, from=3-2, to=3-4]
	\arrow[""{name=4, anchor=center, inner sep=0}, "{\id_{QB} \times \exists f}"', from=3-2, to=4-2]
	\arrow["{Q(f \times \id_A)}"{pos=0.3}, shift left, from=3-4, to=3-5]
	\arrow["{Q(\Delta_A)}", shift left, from=3-5, to=3-6]
	\arrow[""{name=5, anchor=center, inner sep=0}, equals, from=3-5, to=4-5]
	\arrow["{\exists f}", shift left, from=3-6, to=3-7]
	\arrow[equals, from=3-7, to=4-7]
	\arrow["{Qf \otimes \id_{QA}}"', from=4-2, to=4-4]
	\arrow["{\mu_{A,A}}"', from=4-4, to=4-5]
	\arrow["{Q(\Delta_A)}"', from=4-5, to=4-6]
	\arrow["{\exists f}"', from=4-6, to=4-7]
	\arrow["{\mu_{e_B, f_!}^{-1}\ =\ \mathrm{mate}(\mu_{\id_B, f})^{-1}}"{description}, color={rgb,255:red,214;green,92;blue,92}, draw=none, from=0, to=1]
	\arrow["{\mathcal{B}_{\Gamma}}"{description}, color={rgb,255:red,92;green,92;blue,214}, draw=none, from=2, to=3]
	\arrow["{\mu_{f, \id_A}}"{description}, draw=none, from=4, to=5]
\end{tikzcd}
\end{equation*}

                    \end{proof}\end{lemma}

This concludes the retrieval of a $\mathfrak{C}$-regular hyperdoctrines from double functors. Putting Lemma~\ref{lemma: regular hyperdoctrine extends to double functor} and Lemma~\ref{lemma:Companion commuter implies hyperdoctrine} and together, we have:
\begin{theorem}
\label{MainTheorem}
\par{} Let $\mathfrak{C}$ be a cartesian adequate triple and $\mathcal{K}$ be a symmetric monoidal 2-category. Every $\mathfrak{C}$-regular 
 $\mathcal{K}$-hyperdoctrine $P$ can be extended to a lax symmetric monoidal pseudo double functor
  $P^{\bullet} \colon \Span{\mathfrak{C}}^{\op} \to \mathcal{K}$ whose monoidal laxators are companion commuter transformations. Conversely,
  if the class $\mathfrak{C}^l$ includes all diagonals, then the tight component of every such double functor which is strict is a $\mathfrak{C}$-regular $\mathcal{K}$-hyperdoctrine.  \end{theorem}

\begin{corollary}
    Let $\mathfrak{C}$ be a cartesian adequate triple for which $\mathfrak{C}^l$ includes all diagonals, $\mathcal{K}$ be a symmetric monoidal 2-category, and $P$ be a $\mathfrak{C}$-existential symmetric monoidal $\mathcal{K}$-structure. If the mate of $\mu^P_{\id_{\cod(f)}, f}$ is invertible in $\mathcal{K}$ for each morphism $f$ in $\mathsf{C}$, then $P$ is a $\mathfrak{C}$-regular $\mathcal{K}$-hyperdoctrine.
\end{corollary}

\begin{proof}
 By Lemma~\ref{corollary:Double functors from existential structures}, $P^{\bullet}$ is a pseudo double functor. Without loss of generality, we may assume that it is strict (c.f. Theorem 7.5 in \cite{grandis-1999-limits} and note that strictification does not change the tight components of the double functor). Under the hypotheses, Corollary~\ref{corollary: reduced commuter condition}  implies that  the monoidal laxator of $P^{\bullet}$ is a companion commuter transformation. Lemma~\ref{lemma:Companion commuter implies hyperdoctrine} then gives that $P^{\bullet}_0=P$ is a $\mathfrak{C}$-regular $\mathcal{K}$-hyperdoctrine.
\end{proof}

To conclude this section, a word on morphisms. If $\rho \colon P_1 \to P_2$ is a \emph{strong} morphism of  $\mathfrak{C}$-existential $\mathcal{K}$-structures (i.e., is part of a 1-morphism in $\ExistStr$), then clearly $\rho^{\bullet}$ (by construction) is a companion commuter transformation. Thus, $\ExistStr$ embeds into $\Ar^{\colax}(\Dbl)_{\mathrm{cc}}$.
Conversely, the tight component of any companion commuter transformation $\tau \colon Q_1 \to Q_2 \colon \Span{\mathsf{C}}^{\op} \to \Qt{\mathcal{K}}$ of strict double functors is a strong morphism of $\mathfrak{C}$-existential $\mathcal{K}$-structures by Proposition~\ref{proposition: tight transformation determination}.

\section{Discussion}
\label{section:Discussion}
\label{section:discussion}

\par{} We have shown that generalised regular hyperdoctrines correspond to lax symmetric monoidal pseudo double functors with companion commuter laxators, thus embedding into $\SM^{\lax}(\Ar^{\colax}(\Dbl))_{\mathrm{cc}}$. Since this is a 2-category, the result makes it simpler to consider the compositionality of regular hyperdoctrines. For instance, we can 2-functorially make adjustments to both the contexts and the semantics of the doctrine. Moreover, the result suggests one may be able to see regular hyperdoctrines as algebraic structures. While a traditional regular hyperdoctrine is not obviously one, lax symmetric monoidal pseudo double functors are 2-algebraic, and there is only the companion commuter property left to capture algebraically.

 \par{}These ideas can be applied to the double categorical approach to compositionality of systems, where a theory of systems
 is interpreted as a symmetric monoidal loose right module 
on a symmetric monoidal double category. Theorem 3.3.2.2 in \cite{libkind-2025-towards} gives that the restriction of a loose bimodule
along a lax symmetric monoidal pseudo double functor with companion commuter monoidal constraints is again symmetric monoidal. More precisely, if
\(P \colon  \mathsf{C}^{\op} \to  \mathbb {Q}\mathsf {t}(\mathcal {K})\) is a $\mathfrak{C}$-regular $\mathcal{K}$-hyperdoctrine, then we can
restrict the bimodule \(\mathbb {Q}\mathsf {t}(\mathcal {K})(-, -)\) along  \(P^\bullet  \colon \Span{\mathfrak{C}}^{\op} \to  \Qt{\mathcal{K}}\) on the right
 (and \(* \xrightarrow {*} \Qt{\mathcal{K}}\) on the left) to obtain a new systems theory \(\mathbb {Q}\mathsf {t}(\mathcal {K})(*, P^{\bullet })\). A system for this theory consists
 of a context \(\Gamma  \in  \mathsf {C}\), together with a predicate \(\phi  \in  P(\Gamma )\). Its interfaces are objects of \(\mathsf {C}\), and interactions are given by the logical action of \(\mathfrak{C}\)-spans. \par{}As a particular case of interest, let \(\mathcal {K} = \mathsf{Pos}\), \(T\) be a finite set of labels and \(\mathsf {C} = (\mathsf {Finset} \downarrow  T)^{^{\mathsf {op}}}\) be the free cartesian category on \(T\).
 If \(P \colon  \mathsf{C}^{\op} \to  \mathsf{Pos}\) is a conventional regular hyperdoctrine, then the restriction above provides a theory of systems
 with \(\mathbb {C}\mathrm{ospan}(\mathsf{Finset} \downarrow  T)\) as its symmetric monoidal double category of interfaces and interactions. Its loose part is then a hypergraph category in the sense of \cite{fong-2018-hypergraph};
 we can interpret the cospans of labelled finite sets as \emph{undirected wiring diagrams}, which can then be used to model the operadic compositionality of port-plugging systems \cite{spivak-2013-operad}. The construction in this paper therefore explains how substitution and existential quantification should be performed when undirected wiring diagrams are composed by nesting.
  This provides a way to reason about the specifications of such systems via a graphical regular logic --- it is very similar to the mechanism explained in \cite{fong-2018-graphical}, but here we use the free cartesian
  category on \(T\) rather than the free regular category on \(T\), at the cost of assuming inhabitedness in the logic (see Remark 4.12 in \cite{fong-2018-graphical}).

We conclude the paper by mentioning a few clear directions for further work.  The key features used in the paper are that spans form a cartesian Beck-Chevalley double category (in the sense of \cite{dawson-2010-span}),  that
\(\mathbb {Q}\mathsf {t}(\mathcal {K})\) is symmetric monoidal and has all companions and conjoints, and that the monoidal constraints of \(Q\) are companion commuters. In this
way, we could equally deem, for example,  a lax symmetric monoidal pseudo double functor \(\mathbb {R}\mathsf {el}(\mathcal {E})^{^{\mathsf {op}}} \to  \mathbb {Q}\mathsf {t}(\mathcal {K})\) from a double category of relations
with companion commuter constraints to be a type of hyperdoctrine; the Beck-Chevalley condition is handled at the level of the domain, and the Frob\"enius property by the monoidal structure on the double functor.

\par{} Regular ``double hyperdoctrines" in this sense whose double category of contexts is \(\Span{\mathfrak{C}}\) or \(\mathbb {R}\mathsf {el}(\mathcal {E})\),  are ``loosely trivial", in the sense that they are fully determined by their tight components. In both of those
cases we have strong tabulators, and can canonically factor any loose maps as a composite of companions and conjoints then run the argument in Corollary~\ref{corollary;tight transformations determined by tight component}. There are other cartesian Beck-Chevalley double categories of contexts one could take without strong tabulators, such as 
\(\mathsf {V}\mathbb{M}\mathrm{at}\) where \(V\) is a monoidal category with pullbacks and whose tensor product distributes over sums (see \cite{dawson-2010-span}), where the full double structure may not be recovered from the tight one. It is thus plausible one can obtain genuinely new notions of hyperdoctrine from such contexts, where the actions on predicates does not necessarily arise from substitution and quantification alone.

On a different axis, we observe that $\Qt{\mathcal{K}}$, could be replaced. We know for instance that for a regular $\mathsf{Cat}$-hyperdoctrine $P$, each object is mapped by $P^{\bullet}$ to a $T$-algebra, each tight morphism to a strong morphism of $T$-algebras, and each loose morphism to a colax morphism of $T$-algebras for $T$ the 2-monad on $\mathsf{Cat}$ whose pseudo algebras are symmetric monoidal categories; it thus makes sense to consider the semantics as lying in the double category $\mathbb{A}{\mathsf{lg}}(T; \mathsf{Cat})$ of $T$-algebras instead, whose tight morphisms are lax homomorphisms and whose loose arrows are colax homomorphisms (c.f. \cite{grandis2025multiplecategoriesgeneralquintets}). Note that strong morphisms of $T$-algebras are precisely the companiable ones, and that conjoint pairs correspond to doctrinal adjunctions.

More generally, the objects of predicates considered in this paper were symmetric pseudomonoids, but in practice logicians work with other algebraic structures such as Boolean algebras, Heyting algebras, or even toposes \cite{lambekToposes}. One could then replace $T$ with some other $2$-monad and consider ``double hyperdoctrines" taking values in $\mathbb{A}{\mathsf{lg}}(T; \mathcal{K})$. We chose to focus on regular logic in this paper, but in order to capture first-order or linear logic in a similar fashion, one would like to speak of doctrines for which
the fibres are monoidal-closed. This should be possible, mirroring Shulman's work on \emph{closed monoidal fibrations} in the fibrational setting \cite{shulman-2008-framed-bicats}.

The above discussion suggests there is a broader notion of regular hyperdoctrine to study, defined in terms of pseudo double functors $Q \colon \mathbb{C}{\mathrm{tx}}^{\op} \to \mathbb{D}$ with some additional structure that suitably interacts with the companions and conjoints of $\mathbb{D}$. The conceptual picture is that the tight component of \(Q\) moderates adjustments to contexts, while the loose components encode the
possible actions on predicates. This is in line with the zeitgeist of double categorical logic, where the tight dimension handles a theory and the loose
maps handle models (c.f. \cite{lambert-2024-cartesian}).This will be the subject of a future paper. 

Finally,
 a suitable double-categorical notion of \emph{generic predicate} (in the sense of tripos theory \cite{Tripos}) is desirable, so that higher-order logic can also be put into this double-categorical framework. A different approach to the role of double categories in regular logic was detailed by Hayato Nasu following similar motivations, with an accompanying type theory \cite{Nasu}. Connecting the two perspectives would also be valuable.

\nocite{*}\bibliographystyle{plain}\bibliography{paper.bib}







\section*{Appendix}
\label{section:appendix}
This appendix is dedicated to the proof of Lemma~\ref{lemma:BulletConstruction}. We will first verify that the map $(-)^{\bullet} \colon \ExistStr_{\mathrm{w}} \to \Ar^{\colax}(\mathrm{Dbl})$ defined in Section~\ref{section:DoctrinesAsDoublePseudofunctors} is well-defined, then verify its 2-functoriality. Since it is obviously cartesian, we're then done.

Recall that $P^{\bullet}$ acts on objects and tight arrows as $P$, and maps a span $X_1 \xleftarrow{x_1} X \xrightarrow{x_2} X$ to the composite $Px_1 \odot \exists^P x_2 = PX_1 \xrightarrow{Px_1} PX \xrightarrow{\exists^Px_2} PX_2$ (thought of as a loose quintet map). To facilitate the arguments that follow, we express the compositor, unitor, and action of $P^{\bullet}$ on squares in terms of string diagrams:


\begin{equation*}
 P^{\bullet}( \begin {tikzcd}[cramped]
	{Y_1} && {Y_2} \\
	\\
	{X_1} && {X_2}
	\arrow ["Y", "\shortmid "{marking}, from=1-1, to=1-3]
	\arrow [""{name=0, anchor=center, inner sep=0}, "{f_1}"', from=1-1, to=3-1]
	\arrow [""{name=1, anchor=center, inner sep=0}, "{f_2}", from=1-3, to=3-3]
	\arrow ["X"', "\shortmid "{marking}, from=3-1, to=3-3]
	\arrow ["\alpha "{description}, color={rgb,255:red,92;green,92;blue,214}, draw=none, from=0, to=1]
\end {tikzcd})  \hspace{0.2 cm} = \hspace{1 cm} \begin{tikzpicture}[baseline=6em, marker/.style={circle, fill, inner sep=1pt, text=white}]
\node[coordinate, label=below:{{\scriptsize $Pf_1$}}] (input1) at (0,0) {};
\node[coordinate, label=below:{{\scriptsize $Px_1$}}] (input2) at ($(input1)+(2,0)$) {};
\node[coordinate, label=below:{{\scriptsize $\exists^P x_2$}}] (input3) at ($(input2)+(3,0)$)  {};

\node[coordinate, label=above:{{\scriptsize $Py_1$}}] (output1) at (0,5) {};
\node[coordinate, label=above:{{\scriptsize $\exists^P y_2$}}] (output2) at ($(output1)+(1,0)$) {};
\node[coordinate,  label=above:{{\scriptsize $Pf_2$}}] (output3) at ($(output2)+(1, 0)$) {};

\node[box, box ports south=5, box ports north=5, minimum width=10 em] (Box1) at ($(input1)+(1, 2)$) {P=_{f_1x_1}^{y_1 \alpha}};
\node[box, box ports south=5, box ports north=5, minimum width=6 em] (Box2) at ($(output3)+(1, -1)$) {P=_{y_1 \alpha}^{f_2 x_2}};

\wires[black]{Box1 = {south.2 = input1.north, south.4 = input2.north, north.2=output1.south, north.5 = Box2.south.3}, Box2={north.1=output3.south}}{};


\node[coordinate] (cup1) at ($(Box2.south.1)+(-1,0)$) {};
\wires[black, looseness=1.5]{cup1 ={south = Box2.south.2, north = output2.south} }{}; 

\node[coordinate] (cap1) at ($(Box2.north.5)+(1,0)$) {};
\wires[black, looseness=2]{cap1 ={north = Box2.north.5, south=input3.north} }{};

\node[coordinate, label=above:{{\scriptsize $P\alpha$}}] (wirelabel1) at ($(Box2.south.3)+(0.20,-0.75)$) {};
\end{tikzpicture}
\end{equation*}

\begin{equation*}
 (P^{\bullet}_{\odot})_{X, X'}   = \hspace{1 cm} \begin{tikzpicture}[baseline=8em, marker/.style={circle, fill, inner sep=1pt, text=white}]
\node[coordinate, label=below:{{\scriptsize $Px_1$}}] (input1) at (0,0) {};
\node[coordinate, label=below:{{\scriptsize $\exists^P x_2$}}] (input2) at ($(input1)+(2,0)$) {};
\node[coordinate, label=below:{{\scriptsize $P x_2'$}}] (input3) at ($(input2)+(3,0)$)  {};
\node[coordinate, label=below:{{\scriptsize $\exists^P x_3'$}}] (input4) at ($(input3)+(2,0)$)  {};

\node[coordinate, label=above:{{\scriptsize $P(x_1 \chi)$}}] (output1) at (1.2,6) {};
\node[coordinate, label=above:{{\scriptsize $\exists^P (x_3'\chi')$}}] (output2) at ($(output1)+(4.5,0)$) {};

\node[box, box ports south=5, box ports north=5, minimum width=10 em] (Box1) at ($(input2)+(1.5, 2)$) {\mathcal{B}_{x_2, x_2'}^{\chi, \chi'}};
\node[box, box ports south=5, box ports north=5, minimum width=6 em] (Box2) at ($(Box1.north.2)+(-1.5, 1.5)$) {{\gamma^P_{x_1, \chi}}};
\node[box, box ports south=5, box ports north=5, minimum width=6 em] (Box3) at ($(Box1.north.4)+(1.5, 1.5)$) {{\gamma^{\exists^P}_{\chi', x_3'}}};

\wires[black]{Box1 = {south.1 = input2.north, south.5=input3.north, north.1 = Box2.south.4, north.5 = Box3.south.2}, Box2= {south.2 = input1.north, north.3 = output1.south}, Box3= {south.4=input4.north, north.3=output2.south}}{};



\end{tikzpicture}
\end{equation*}

\begin{equation*}
 (P^{\bullet}_e)_X \hspace{0.2 cm} = \hspace{0.5 cm} \begin{tikzpicture}[baseline=1em, marker/.style={circle, fill, inner sep=1pt, text=white}]
\node[coordinate, label=below:{{\scriptsize $P(\id_X)$}}] (input1) at (0,0) {};
\node[coordinate, label=below:{{\scriptsize $\exists^P (\id_X)$}}] (input2) at ($(input1)+(3,0)$) {};
\wires[black, looseness=2]{input1 ={north = input2.north} }{};

\wires[black]{}{};



\end{tikzpicture}
\end{equation*}

We will not use the fact that $L, P, F$ and $\rho$ are strict unless necessary, in order to suggest that the restriction of hyperdoctrines to 2-functorial ones is not an essential one (although relaxing this hypothesis requires working with double pseudofunctors and tight pseudotransformations instead).

\begin{proposition}
If $P \colon \mathsf{C}^{\op} \to \mathcal{K}$ is a $\mathfrak{C}$-existential $\mathcal{K}$-structure, then $P^{\bullet} \colon \Span{\mathfrak{C}}^{\op} \to \Qt{\mathcal{K}}$ is a pseudo double functor.
\end{proposition}

\begin{proof}
We verify all the conditions in Definition~\ref{def:DoublePseudofunctor} in turn.\\

 $\bullet$ \underline{Naturality of \(P^{\bullet }_{\odot }\)}:\\ 

Suppose given $
\begin{tikzcd}
	{Y_1} && {Y_2} && {Y_3} \\
	\\
	{X_1} && {X_2} && {X_3}
	\arrow[""{name=0, anchor=center, inner sep=0}, "Y"{inner sep=.8ex}, "\shortmid"{marking}, from=1-1, to=1-3]
	\arrow["{f_1}"', from=1-1, to=3-1]
	\arrow[""{name=1, anchor=center, inner sep=0}, "{Y'}"{inner sep=.8ex}, "\shortmid"{marking}, from=1-3, to=1-5]
	\arrow["{f_2}", from=1-3, to=3-3]
	\arrow["{f_3}", from=1-5, to=3-5]
	\arrow[""{name=2, anchor=center, inner sep=0}, "X"'{inner sep=.8ex}, "\shortmid"{marking}, from=3-1, to=3-3]
	\arrow[""{name=3, anchor=center, inner sep=0}, "{X'}"'{inner sep=.8ex}, "\shortmid"{marking}, from=3-3, to=3-5]
	\arrow["\alpha"{description}, color={rgb,255:red,92;green,92;blue,214}, draw=none, from=0, to=2]
	\arrow["\beta"{description}, color={rgb,255:red,92;green,92;blue,214}, draw=none, from=1, to=3]
\end{tikzcd}$ in $\Span{\mathfrak{C}}^{\op}$, where $X \odot X'$ and $Y \odot Y'$ are given by the $\mathfrak{C}$-pullbacks $
\begin{tikzcd}
	{X\odot X'} && {X'} & {Y\odot Y'} && {Y'} \\
	\\
	X && {X_2} & Y && {Y_2}
	\arrow["b", from=1-1, to=1-3]
	\arrow["a"', from=1-1, to=3-1]
	\arrow["\lrcorner"{anchor=center, pos=0.125}, draw=none, from=1-1, to=3-3]
	\arrow["{x_2'}", from=1-3, to=3-3]
	\arrow["d", from=1-4, to=1-6]
	\arrow["c"', from=1-4, to=3-4]
	\arrow["\lrcorner"{anchor=center, pos=0.125}, draw=none, from=1-4, to=3-6]
	\arrow["{y_2'}", from=1-6, to=3-6]
	\arrow["{x_2}"', from=3-1, to=3-3]
	\arrow["{y_2}"', from=3-4, to=3-6]
\end{tikzcd}$.

Then, splitting $\eta^{P(y_3'd)}$ and $\epsilon^{P(x_3'b)}$ via the way adjunctions compose and using the triangle identities yields

\begin{equation*}
(P^{\bullet}_{\odot})_{Y,Y'}^{-1} \circ P^{\bullet}(\alpha \mid \beta)=
\begin{tikzpicture}[baseline=7em, marker/.style={circle, fill, inner sep=1pt, text=white}]
\node[coordinate, label=below:{{\scriptsize $Pf_1$}}] (input1) at (0,0) {};
\node[coordinate, label=below:{{\scriptsize $Px_1$}}] (input2) at ($(input1)+(4,0)$) {};
\node[coordinate, label=below:{{\scriptsize $Pa$}}] (input3) at ($(input2)+(4,0)$)  {};
\node[coordinate, label=below:{{\scriptsize $\exists (x'_3 b)$}}] (input4) at ($(input3)+(2.3,0)$) {};

\node[coordinate, label=above:{{\scriptsize $Py_1$}}] (output1) at (0,5) {};
\node[coordinate, label=above:{{\scriptsize $\exists y_2$}}] (output2) at ($(output1)+(1,0)$) {};
\node[coordinate,  label=above:{{\scriptsize $P(y_2')$}}] (output3) at ($(output2)+(1.5,0)$) {};
\node[coordinate, label=above:{{\scriptsize $\exists y_3'$}}] (output4) at ($(output3)+(3.5,0)$) {};
\node[coordinate, label=above:{{\scriptsize $Pf_3$}}] (output5) at ($(output4)+(2,0)$) {};

\node[box, box ports south=5, box ports north=5, minimum width=25 em] (Box1) at ($(input1)+(4, 1)$) {=_{f_1 x_1 a}^{y_1 c (\beta \alpha)}};
\node[box, box ports south=5, box ports north=5] (Box2) at ($(Box1.north.2)+(1, 2.5)$) {=_{y_2 c}^{y_2'd}};
\node[box, box ports south=5, box ports north=5] (Box3) at ($(input2)+(4, 3.5)$) {=_{(y_3' d) (\beta \alpha)}^{f_3 (x_3'b)}};
\node[box, box ports south=5, box ports north=5] (Box4) at ($(Box2)+(2.5,-0.5)$) {\gamma^{\exists^P}_{y'_3, d}};

\wires[black]{Box1 = {north.1 = output1.south,  south.1 = input1.north, south.3 = input2.north, south.5 = input3.north, north.3 = Box2.south.4}}{};
\path (Box1.north.3) -- (Box2.south.4) node[pos=0.5, right] {\scriptsize $Pc$}; 
\wires[black]{Box2 = {north.1 = output3.south}}{}


\node[coordinate] (cup1) at ($(Box2.south.2)+(-2,0)$) {};
\wires[black, looseness=1.5]{Box2 ={south.2 = cup1.south} }{};
\wires[black, looseness=0.5] {cup1 = {north = output2.south}}{}
\path (cup1.south) -- (Box2.south.1) node[pos=.9, below] {\scriptsize $Py_2$};
\node[coordinate] (cup2) at ($(Box3.north.5)+(2,0)$) {};
\wires[black, looseness=1.5] {cup2 = {north = Box3.north.5, south =input4.north}}{}
\wires[black, looseness=2] {Box4 = {south.3 = Box3.south.2}}{}
\path (Box4.south.3) -- (Box3.south.2) node[pos=0.65 , below] {\scriptsize $P(y_3'd)$}; 

\node[coordinate] (cap2) at ($(Box3.south.1)+(-1,0)$) {};

\wires[black, looseness=3]{Box2 = {north.5 = Box4.north.2} }{};
\path (Box2.north.5) -- (Box4.north.2) node[pos=0.15 , above] {\scriptsize $Pd$}; 

\wires[black, looseness=1.4]{Box4 = {north.5 = output4.south}}{};
\wires[black, looseness=1.4]{Box3 = {south.4 = Box1.north.5, north.3 = output5.south}}{};
\path (Box3.south.4) -- (Box1.north.5) node[pos=0.6, right] {\scriptsize $P(\beta \alpha)$}; 
\end{tikzpicture}%
\end{equation*}

\begin{equation*}
    = \begin{tikzpicture}[baseline=7em, marker/.style={circle, fill, inner sep=1pt, text=white}]
\node[coordinate, label=below:{{\scriptsize $Pf_1$}}] (input1) at (0,0) {};
\node[coordinate, label=below:{{\scriptsize $Px_1$}}] (input2) at ($(input1)+(4,0)$) {};
\node[coordinate, label=below:{{\scriptsize $Pa$}}] (input3) at ($(input2)+(4,0)$)  {};
\node[coordinate, label=below:{{\scriptsize $\exists (x'_3 b)$}}] (input4) at ($(input3)+(2.3,0)$) {};

\node[coordinate, label=above:{{\scriptsize $Py_1$}}] (output1) at (0,5) {};
\node[coordinate, label=above:{{\scriptsize $\exists y_2$}}] (output2) at ($(output1)+(1,0)$) {};
\node[coordinate,  label=above:{{\scriptsize $P(y_2')$}}] (output3) at ($(output2)+(2,0)$) {};
\node[coordinate, label=above:{{\scriptsize $\exists y_3'$}}] (output4) at ($(output3)+(3,0)$) {};
\node[coordinate, label=above:{{\scriptsize $Pf_3$}}] (output5) at ($(output4)+(2,0)$) {};

\node[box, box ports south=5, box ports north=5, minimum width=25 em] (Box1) at ($(input1)+(4, 1)$) {=_{f_1 x_1 a}^{y_1 c (\beta \alpha)}};
\node[box, box ports south=5, box ports north=5] (Box2) at ($(Box1.north.2)+(1, 2)$) {=_{y_2 c}^{y_2'd}};
\node[box, box ports south=5, box ports north=5] (Box3) at ($(input2)+(3.5, 3.5)$) {=_{y_3' d (\beta \alpha)}^{f_3 (x_3'b)}};

\wires[black]{Box1 = {north.1 = output1.south,  south.1 = input1.north, south.3 = input2.north, south.5 = input3.north, north.3 = Box2.south.4}}{};
\path (Box1.north.3) -- (Box2.south.4) node[pos=0.5, right] {\scriptsize $Pc$}; 
\wires[black]{Box2 = {north.1 = output3.south}}{}


\node[coordinate] (cup1) at ($(Box2.south.2)+(-2,0)$) {};
\wires[black, looseness=1.5]{Box2 ={south.2 = cup1.south} }{};
\wires[black, looseness=0.5] {cup1 = {north = output2.south}}{}
\path (cup1.south) -- (Box2.south.1) node[pos=.9, below] {\scriptsize $Py_2$};
\node[coordinate] (cup2) at ($(Box3.north.5)+(2,0)$) {};
\wires[black, looseness=1.5] {cup2 = {north = Box3.north.5, south =input4.north}}{}

\node[coordinate] (cap1) at ($(Box2.north.5)+(1,0)$) {};
\node[coordinate] (cap2) at ($(Box3.south.1)+(-1,0)$) {};
\node[coordinate] (aux1) at ($(cap1)+(0,-1.5)$) {};
\wires[black, looseness=1.5]{Box2 = {north.5 = cap1.north} }{};
\wires[black, looseness=1.5]{cap1 = {south = aux1.north} }{};
\wires[black, looseness=1.4]{aux1 = {south = Box3.south.2}}{};
\wires[black, looseness=1.4]{cap2 = {south = Box3.south.1, north = output4.south}}{};
\wires[black, looseness=1.4]{Box3 = {south.4 = Box1.north.5, north.3 = output5.south}}{};
\path (Box3.south.4) -- (Box1.north.5) node[pos=0.5, right] {\scriptsize $P(\beta \alpha)$}; 
\path (Box3.south.2) -- (aux1.south) node[pos=-0.1, below] {\scriptsize $Pd$}; 
\end{tikzpicture}
\end{equation*}

 \begin{equation*}
     = \begin{tikzpicture}[baseline=7em, marker/.style={circle, fill, inner sep=1pt, text=white}]
\node[coordinate, label=below:{{\scriptsize $Pf_1$}}] (input1) at (0,0) {};
\node[coordinate, label=below:{{\scriptsize $Px_1$}}] (input2) at ($(input1)+(4,0)$) {};
\node[coordinate, label=below:{{\scriptsize $Pa$}}] (input3) at ($(input2)+(4,0)$)  {};
\node[coordinate, label=below:{{\scriptsize $\exists (x'_3 b)$}}] (input4) at ($(input3)+(1.3,0)$) {};

\node[coordinate, label=above:{{\scriptsize $Py_1$}}] (output1) at (0,6) {};
\node[coordinate, label=above:{{\scriptsize $\exists y_2$}}] (output2) at ($(output1)+(1,0)$) {};
\node[coordinate,  label=above:{{\scriptsize $P(y_2')$}}] (output3) at ($(output2)+(1,0)$) {};
\node[coordinate, label=above:{{\scriptsize $\exists y_3'$}}] (output4) at ($(output3)+(2,0)$) {};
\node[coordinate, label=above:{{\scriptsize $Pf_3$}}] (output5) at ($(output4)+(2,0)$) {};

\node[box, box ports south=5, box ports north=5, minimum width=25 em] (Box1) at ($(input1)+(4, 1)$) {=_{f_1 x_1 a}^{y_1 c (\beta \alpha)}};
\node[box, box ports south=5, box ports north=5] (Box2) at ($(Box1.north.2)+(2, 1.5)$) {=_{y_2 c}^{y_2'd}};
\node[box, box ports south=5, box ports north=5] (Box3) at ($(input2)+(2, 5.)$) {=_{y_3' d (\beta \alpha)}^{f_3 (x_3'b)}};

\wires[black]{Box1 = {north.1 = output1.south,  south.1 = input1.north, south.3 = input2.north, south.5 = input3.north, north.3 = Box2.south.4}}{};
\path (Box1.north.3) -- (Box2.south.4) node[pos=0.5, right] {\scriptsize $Pc$}; 
\wires[black]{Box2 = {north.1 = output3.south, north.5 = Box3.south.2}}{}
\wires[black]{Box3 = {north.3 = output5.south}}{}


\node[coordinate] (cup1) at ($(Box2.south.2)+(-2,0)$) {};
\wires[black, looseness=1.5]{Box2 ={south.2 = cup1.south} }{};
\wires[black, looseness=0.5] {cup1 = {north = output2.south}}{}
\path (cup1.south) -- (Box2.south.1) node[pos=.9, below] {\scriptsize $Py_2$};
\node[coordinate] (cup2) at ($(Box3.south.1)+(-1,0)$) {};
\wires[black, looseness=1]{cup2 ={south  = Box3.south.1, north = output4.south} }{};

\wires[black, looseness=1.4]{Box3 = {south.4 = Box1.north.5}}{};
\path (Box3.south.4) -- (Box1.north.5) node[pos=0.5, right] {\scriptsize $P(\beta \alpha)$}; 
\node[coordinate] (cap1) at ($(Box3.north.5)+(+2,0)$) {};
\wires[black, looseness=2]{cap1 ={north  =Box3.north.5, south = input4.north} }{};
\path (Box2.north.5) -- (Box3.south.2) node[pos=0.5, below] {\scriptsize  $Pd$}; 
\end{tikzpicture}
 \end{equation*}

\begin{equation*}
    = \begin{tikzpicture}[baseline=7em, marker/.style={circle, fill, inner sep=1pt, text=white}]
\node[coordinate, label=below:{{\scriptsize $Pf_1$}}] (input1) at (0,0) {};
\node[coordinate, label=below:{{\scriptsize $Px_1$}}] (input2) at ($(input1)+(4,0)$) {};
\node[coordinate, label=below:{{\scriptsize $Pa$}}] (input3) at ($(input2)+(4,0)$)  {};
\node[coordinate, label=below:{{\scriptsize $\exists  b$}}] (input4) at ($(input3)+(1.3,0)$) {};
\node[coordinate, label=below:{{\scriptsize $\exists  x_3'$}}] (input5) at ($(input4)+(1.3,0)$) {};

\node[coordinate, label=above:{{\scriptsize $Py_1$}}] (output1) at (0,6) {};
\node[coordinate, label=above:{{\scriptsize $\exists y_2$}}] (output2) at ($(output1)+(1,0)$) {};
\node[coordinate,  label=above:{{\scriptsize $P(y_2')$}}] (output3) at ($(output2)+(1,0)$) {};
\node[coordinate, label=above:{{\scriptsize $\exists y_3'$}}] (output4) at ($(output3)+(2,0)$) {};
\node[coordinate, label=above:{{\scriptsize $Pf_3$}}] (output5) at ($(output4)+(2,0)$) {};

\node[box, box ports south=5, box ports north=5, minimum width=25 em] (Box1) at ($(input1)+(4, 1)$) {=_{f_1 x_1 a}^{y_1 c (\beta \alpha)}};
\node[box, box ports south=5, box ports north=5] (Box2) at ($(Box1.north.2)+(2, 1.5)$) {=_{y_2 c}^{y_2'd}};
\node[box, box ports south=5, box ports north=5] (Box3) at ($(input2)+(2, 5)$) {=_{y_3' d (\beta \alpha)}^{f_3 (x_3'd)}};

\wires[black]{Box1 = {north.1 = output1.south,  south.1 = input1.north, south.3 = input2.north, south.5 = input3.north, north.3 = Box2.south.4}}{};
\path (Box1.north.3) -- (Box2.south.4) node[pos=0.5, right] {\scriptsize $Pc$}; 
\wires[black]{Box2 = {north.1 = output3.south, north.5 = Box3.south.2}}{}
\wires[black]{Box3 = {north.3 = output5.south}}{}


\node[coordinate] (cup1) at ($(Box2.south.2)+(-2,0)$) {};
\wires[black, looseness=1.5]{Box2 ={south.2 = cup1.south} }{};
\wires[black, looseness=0.5] {cup1 = {north = output2.south}}{}
\path (cup1.south) -- (Box2.south.1) node[pos=.9, below] {\scriptsize $Py_2$};
\node[coordinate] (cup2) at ($(Box3.south.1)+(-1,0)$) {};
\wires[black, looseness=1]{cup2 ={south  = Box3.south.1, north = output4.south} }{};

\wires[black, looseness=1.4]{Box3 = {south.4 = Box1.north.5}}{};
\path (Box3.south.4) -- (Box1.north.5) node[pos=0.6, left] {\scriptsize $P(\beta \alpha)$}; 
\node[coordinate] (cap1) at ($(Box3.north.4)+(1.5,0)$) {};
\wires[black, looseness=1.5]{cap1 ={north  =Box3.north.5, south = input4.north} }{};
\path (Box2.north.5) -- (Box3.south.2) node[pos=0.5, below] {\scriptsize  $Pd$}; 
\node[coordinate] (cap2) at ($(Box3.north.4)+(+2,0)$) {};
\wires[black, looseness=1.5]{cap2 ={north  =Box3.north.4, south = input5.north} }{};
\end{tikzpicture}
\end{equation*}

\begin{equation*}
    =\adjustbox{scale=0.9}{\begin{tikzpicture}[baseline=12em, marker/.style={circle, fill, inner sep=1pt, text=white}]
\node[coordinate, label=below:{{\scriptsize $Pf_1$}}] (input1) at (0,0) {};
\node[coordinate, label=below:{{\scriptsize $Px_1$}}] (input2) at ($(input1)+(2,0)$) {};
\node[coordinate, label=below:{{\scriptsize $Pa$}}] (input3) at ($(input2)+(3.5,0)$)  {};
\node[coordinate, label=below:{{\scriptsize $\exists  b$}}] (input4) at ($(input3)+(1.3,0)$) {};
\node[coordinate, label=below:{{\scriptsize $\exists  x_3'$}}] (input5) at ($(input4)+(1.3,0)$) {};

\node[coordinate, label=above:{{\scriptsize $Py_1$}}] (output1) at (0,10) {};
\node[coordinate, label=above:{{\scriptsize $\exists y_2$}}] (output2) at ($(output1)+(1,0)$) {};
\node[coordinate,  label=above:{{\scriptsize $P(y_2')$}}] (output3) at ($(output2)+(1,0)$) {};
\node[coordinate, label=above:{{\scriptsize $\exists y_3'$}}] (output4) at ($(output3)+(2,0)$) {};
\node[coordinate, label=above:{{\scriptsize $Pf_3$}}] (output5) at ($(output4)+(1.5,0)$) {};

\node[box, box ports south=5, box ports north=5, minimum width=10 em] (Box1) at ($(input1)+(1, 1)$) {=_{f_1 x_1 a}^{ y_1  \alpha}};
\node[box, box ports south=5, box ports north=5] (Box2) at ($(Box1.north.2)+(2, 1.5)$) {=_{y_2 \alpha}^{f_2 x_2}};
\node[box, box ports south=5, box ports north=5] (Box3) at ($(input2)+(3.65, 4.5)$) {=_{x_2 a}^{x_2' b}};
\node[box, box ports south=5, box ports north=5] (Box4) at ($(input2)+(2.5 , 6)$) {=_{f_2 x_2'}^{{y_2'\beta}}};
\node[box, box ports south=5, box ports north=5] (Box5) at ($(input2)+(4,8.5)$) {=_{f_2 x_2'}^{y_2'\beta}};

\wires[black]{Box1 = {north.1 = output1.south,  south.2= input1.north, south.4 = input2.north, north.4 = Box2.south.5}}{};
\wires[black]{Box4 = {south.1 = Box2.north.2, south.4 = Box3.north.1, north.1 = output3.south, north.4 = Box5.south.4}}{};
\wires[black]{Box5 = {north.1 = output5.south}}{};


\node[coordinate] (cup1) at ($(Box2.south.2)+(-2,0)$) {};
\wires[black, looseness=1.5]{Box2 ={south.2 = cup1.south} }{};
\wires[black, looseness=0.5] {cup1 = {north = output2.south}}{}
\path (cup1.south) -- (Box2.south.1) node[pos=.9, below] {\scriptsize $Py_2$};
\node[coordinate] (cup2) at ($(Box5.south.2)+(-1.5,0)$) {};
\wires[black, looseness=1.5]{cup2 ={south  =Box5.south.2, north = output4.south} }{};

 
\node[coordinate] (cap1) at ($(Box3.north.4)+(1.5,0)$) {};
\wires[black, looseness=1.5]{cap1 ={north  =Box3.north.5, south = input4.north} }{};
\node[coordinate] (cap2) at ($(Box2.north.5)+(1.25,0)$) {};
\wires[red, looseness=3]{cap2 ={north  =Box2.north.5} }{};
\path (cap2.north) -- (Box2.north.5) node[pos=-0.8, right] {\scriptsize  $\textcolor{red}{\ \ Px_2}$}; 
\node[coordinate] (aux1) at ($(cap2.south)+(1.25,0)$) {};
\wires[red, looseness=3]{aux1 ={south =cap2.south, north= Box3.south.1} }{};
\wires[black]{Box3={south.5 = input3.north}}{}
\node[coordinate] (cap3) at ($(Box5.north.5)+(2,0)$) {};
\wires[black, looseness=2]{cap3 ={north  =Box5.north.5, south = input5.north} }{};
\path (Box1.north.4) -- (Box2.south.5) node[pos=0.3, right] {\scriptsize  $P\alpha$}; 
\path (Box2.north.2) -- (Box4.south.1) node[pos=0.5, above] {\scriptsize  $Pf_2$};
\path (Box3.north.1) -- (Box4.south.4) node[pos=0.5, right] {\scriptsize  $Px_2'$};
\path (Box4.north.4) -- (Box5.south.4) node[pos=0.4, right] {\scriptsize  $P\beta$};
\end{tikzpicture}}
\end{equation*}

\begin{equation*}
    =\adjustbox{scale=0.9}{\begin{tikzpicture}[baseline=12em, marker/.style={circle, fill, inner sep=1pt, text=white}]
\node[coordinate, label=below:{{\scriptsize $Pf_1$}}] (input1) at (0,0) {};
\node[coordinate, label=below:{{\scriptsize $Px_1$}}] (input2) at ($(input1)+(2,0)$) {};
\node[coordinate, label=below:{{\scriptsize $Pa$}}] (input3) at ($(input2)+(3.5,0)$)  {};
\node[coordinate, label=below:{{\scriptsize $\exists  b$}}] (input4) at ($(input3)+(1.3,0)$) {};
\node[coordinate, label=below:{{\scriptsize $\exists  x_3'$}}] (input5) at ($(input4)+(1.3,0)$) {};

\node[coordinate, label=above:{{\scriptsize $Py_1$}}] (output1) at (0,10) {};
\node[coordinate, label=above:{{\scriptsize $\exists y_2$}}] (output2) at ($(output1)+(1,0)$) {};
\node[coordinate,  label=above:{{\scriptsize $P(y_2')$}}] (output3) at ($(output2)+(1,0)$) {};
\node[coordinate, label=above:{{\scriptsize $\exists y_3'$}}] (output4) at ($(output3)+(2,0)$) {};
\node[coordinate, label=above:{{\scriptsize $Pf_3$}}] (output5) at ($(output4)+(1.5,0)$) {};

\node[box, box ports south=5, box ports north=5, minimum width=10 em, color=olive] (Box1) at ($(input1)+(1, 1)$) {=_{f_1 x_1 a}^{ y_1  \alpha}};
\node[box, box ports south=5, box ports north=5, color=olive] (Box2) at ($(Box1.north.2)+(2, 1.5)$) {=_{y_2 \alpha}^{f_2 x_2}};
\node[box, box ports south=5, box ports north=5, color=blue] (Box3) at ($(input2)+(3.65, 4.5)$) {=_{x_2 a}^{x_2' b}};
\node[box, box ports south=5, box ports north=5] (Box4) at ($(input2)+(2.5 , 6.15)$) {=_{f_2 x_2'}^{y_2'\beta}};
\node[box, box ports south=5, box ports north=5] (Box5) at ($(input2)+(4,8.5)$) {=_{f_2 x_2'}^{y_2'\beta}};

\wires[olive]{Box1 = {north.1 = output1.south,  south.2= input1.north, south.4 = input2.north, north.4 = Box2.south.5}}{};
\wires[black]{Box4 = {south.4 = Box3.north.1, north.1 = output3.south, north.4 = Box5.south.4}}{};
\wires[olive]{Box2 = {north.2 = Box4.south.1}}{}
\wires[black]{Box5 = {north.1 = output5.south}}{};


\node[coordinate] (cup1) at ($(Box2.south.2)+(-2,0)$) {};
\wires[olive, looseness=1.5]{Box2 ={south.2 = cup1.south} }{};
\wires[olive, looseness=0.5] {cup1 = {north = output2.south}}{}
\path (cup1.south) -- (Box2.south.1) node[pos=.9, below] {\scriptsize $\textcolor{olive}{Py_2}$};
\node[coordinate] (cup2) at ($(Box5.south.2)+(-1.5,0)$) {};
\wires[black, looseness=1.5]{cup2 ={south  =Box5.south.2, north = output4.south} }{};

 
\node[coordinate] (cap1) at ($(Box3.north.4)+(1.5,0)$) {};
\wires[blue, looseness=1.5]{cap1 ={north  =Box3.north.5, south = input4.north} }{};
\node[coordinate] (cap2) at ($(Box2.north.5)+(1.25,0)$) {};
\wires[olive, looseness=3]{cap2 ={north  =Box2.north.5} }{};
\path (cap2.north) -- (Box2.north.5) node[pos=-0.8, right] {\scriptsize  $\textcolor{blue}{\ \ Px_2}$}; 
\node[coordinate] (aux1) at ($(cap2.south)+(1.25,0)$) {};
\wires[blue, looseness=3]{aux1 ={south =cap2.south, north= Box3.south.1} }{};
\wires[blue]{Box3={south.5 = input3.north}}{}
\node[coordinate] (cap3) at ($(Box5.north.5)+(2,0)$) {};
\wires[black, looseness=2]{cap3 ={north  =Box5.north.5} }{};
\wires[blue, looseness=2]{cap3 ={south = input5.north} }{};
\path (Box1.north.4) -- (Box2.south.5) node[pos=0.3, right] {\scriptsize  $\textcolor{olive}{P\alpha}$}; 
\path (Box2.north.2) -- (Box4.south.1) node[pos=0.5, above] {\scriptsize  $\textcolor{olive}{Pf_2}$};
\path (Box3.north.1) -- (Box4.south.4) node[pos=0.5, right] {\scriptsize  $Px_2'$};
\path (Box4.north.4) -- (Box5.south.4) node[pos=0.4, below] {\scriptsize  $P\beta$};
\end{tikzpicture}}
=(\textcolor{olive}{P^{\bullet}(\alpha)} \mid P^{\bullet}(\beta)) \textcolor{blue}{(P^\bullet_{\odot})_{X, X'}^{-1}}.
\end{equation*}

$\bullet$ \underline{Double associativity of $P^{\bullet}$:} Recall that $P^{\bullet}_{\odot}$ at composable spans is computed in terms of the Beck-Chevalley cell for the pullback defining span composites. Given composable spans \(X_1 \overset {X}{\mathrel {\mkern 3mu\vcenter {\hbox {$\shortmid $}}\mkern -10mu{\to }}} X_2 \overset {X'}{\mathrel {\mkern 3mu\vcenter {\hbox {$\shortmid $}}\mkern -10mu{\to }}} X_3 \overset {X''}{\mathrel {\mkern 3mu\vcenter {\hbox {$\shortmid $}}\mkern -10mu{\to }}} X_4\),  double associativity of $P^{\bullet}$ thus follows from observing that the pullback $R$ in the diagram 
\begin{equation*}
    \begin{tikzcd}
	R && {X'\odot X''} && {X''} \\
	\\
	{X\odot X'} && {X'} && {X_3} \\
	\\
	X && {X_2}
	\arrow[from=1-1, to=1-3]
	\arrow[from=1-1, to=3-1]
	\arrow["\lrcorner"{anchor=center, pos=0.125}, draw=none, from=1-1, to=3-3]
	\arrow[from=1-3, to=1-5]
	\arrow[from=1-3, to=3-3]
	\arrow["\lrcorner"{anchor=center, pos=0.125}, draw=none, from=1-3, to=3-5]
	\arrow[from=1-5, to=3-5]
	\arrow[from=3-1, to=3-3]
	\arrow[from=3-1, to=5-1]
	\arrow["\lrcorner"{anchor=center, pos=0.125}, draw=none, from=3-1, to=5-3]
	\arrow[from=3-3, to=3-5]
	\arrow[from=3-3, to=5-3]
	\arrow[from=5-1, to=5-3]
\end{tikzcd}
\end{equation*}
can equivalently be described as $(X \odot X') \odot X''$ or $X \odot (X' \odot X'')$, and taking Beck-Chevalley cells.\\

Naturality of $(P^{\bullet})_e$ and the unital coherence  and double unitality of $P^{\bullet}$ are trivial.
 
\end{proof}

\begin{proposition}
If
\begin{tikzcd}
	{\mathsf{C}_1^{\op}} && {\mathcal{K}_1} \\
	\\
	{\mathsf{C}_2^{\op}} && {\mathcal{K}_2}
	\arrow[""{name=0, anchor=center, inner sep=0}, "{P_1}", from=1-1, to=1-3]
	\arrow["{F^{\op}}"', from=1-1, to=3-1]
	\arrow["L", from=1-3, to=3-3]
	\arrow[""{name=1, anchor=center, inner sep=0}, "{P_2}"', from=3-1, to=3-3]
	\arrow["\rho", color={rgb,255:red,92;green,92;blue,214}, between={0.2}{0.8}, Rightarrow, from=0, to=1]
\end{tikzcd}
is a 1-morphism in $\ExistStr_{\mathrm{w}}$, then 
\begin{equation*}
\begin{tikzcd}
	{\Span{\mathsf{C}_1}^{\op}} && {\Qt{\mathcal{K}_1}} \\
	\\
	{\Span{\mathsf{C}_2}^{\op}} && {\Qt{\mathcal{K}_2}}
	\arrow[""{name=0, anchor=center, inner sep=0}, "{P_1^{\bullet}}", from=1-1, to=1-3]
	\arrow["{\Span{F}^{\op}}"', from=1-1, to=3-1]
	\arrow["L", from=1-3, to=3-3]
	\arrow[""{name=1, anchor=center, inner sep=0}, "{P_2^{\bullet}}"', from=3-1, to=3-3]
	\arrow["{\rho^{\bullet}}", color={rgb,255:red,92;green,92;blue,214}, between={0.2}{0.8}, Rightarrow, from=0, to=1]
\end{tikzcd}
\end{equation*}
is a 1-morphism in $\Ar^{\colax}(\mathrm{Dbl})$.
\end{proposition}

\begin{proof}
    We must show that $\rho^{\bullet}$ is a tight transformation $L P_1^{\bullet} \to P_2^{\bullet}\ \Span{F}^{\op}$.\\ 
    
    $\bullet$ \underline{Tight naturality of $\rho^{\bullet}$}: clear, since the tight component of $\rho^{\bullet}$ is just $\rho$.\\

 $\bullet$ \underline{Unital coherence of $\rho^{\bullet}$}: given an object $A \in \Span{\mathfrak{C}_1}^{\op}$, we have\\

 \begin{equation*}
  \frac{\textcolor{red}{(\Qt{L}\ {P_1^{\bullet}}_e)_A}}{\rho^{\bullet}_{e_A}}= \adjustbox{scale=0.7}{\begin{tikzpicture}[baseline=3em, marker/.style={circle, fill, inner sep=1pt, text=white}]
\node[box, box ports south=5, box ports north=5, minimum width=5cm, color=black] (Box1) at (0,0) {{\rho}^{-1}_{\id_A}};
\node[box, box ports south=2, box ports north=2, minimum width=2cm, color=black] (Box2) at ($(Box1)+(1.5, 2)$) {\rho_{\id_A}};

\node[coordinate, label=below:{{\scriptsize $\rho_A$}}]                          (input1) at ($(Box1.south.1)+(0,-1)$) {};
\node[coordinate, label={[text=black]below:{{\scriptsize $P_2 F^{\op} \id_A$}}}]             (input2) at ($(Box1.south.3)+(0,-1)$) {};
\node[coordinate, label={[text=black]below:{{\scriptsize $\exists^{P_2}F\id_A$}}}]            (input3) at ($(input2)+(3,0)$) {};

\node[coordinate, label={[text=black]above:{{\scriptsize $\rho_A$}}}]            (output1) at ($(Box2.north.1) + (0,1)$) {};

\node[coordinate] (cup1) at ($(Box2.south.1)+(-1,0)$) {};
\node[coordinate] (cup2) at ($(Box2.north.2)+(+1,0)$) {};

\node[coordinate] (cap1) at ($(cup1)+(-1,0)$) {};

 \wires[black, looseness=3]{cup1 = {south = Box2.south.1}}{};

 \wires[red, looseness=3]{cup1= {north = cap1.north}}{}

\wires[black]{cap1 = {south = Box1.north.2}}{}

\wires[black]{Box2 = {south.2 = Box1.north.5,  north.1 = output1.south}}{}

\wires[black, looseness=3]{Box2 = {north.2 = cup2.north}}{};

\wires[black]{Box1 = {south.1 = input1.north, south.3 = input2.north}}{}

\wires[black]{input3 = {north = cup2.south}}{}

\node[coordinate, label=above:{\scriptsize $LP_1A$}] (label1) at ($(Box1.north.1)+(0.2,0.5)$) {};

\node[coordinate, label=above:{\scriptsize $\rho_A$}] (label2) at ($(Box1.north.4)+(0.5,0.5)$) {};
\end{tikzpicture}} = \adjustbox{scale=0.7}{\begin{tikzpicture}[baseline=3em, marker/.style={circle, fill, inner sep=1pt, text=white}]
\

\node[coordinate, label=below:{{\scriptsize $\rho_A$}}]                          (input1) at (0,-1.5) {};
\node[coordinate, label={[text=black]below:{{\scriptsize $P_2 F^{\op} \id_A$}}}]             (input2) at ($(input1)+(2,0)$) {};
\node[coordinate, label={[text=black]below:{{\scriptsize $\exists^{P_2}F\id_A$}}}]            (input3) at ($(input2)+(2,0)$) {};

\node[coordinate, label={[text=black]above:{{\scriptsize $\rho_A$}}}]            (output1) at ($(0, 3.5)$) {};


\wires[black]{output1 = {south = input1.north}}{}

\wires[blue, looseness=3]{input2 = {north = input3.north}}{}
\end{tikzpicture}} = \frac{e_{\rho_A}}{\textcolor{blue}{(P_2^{\bullet}\  \Span{F}^{\op}_e)_A}} .  
 \end{equation*} 

  $\bullet$ \underline{Compositional coherence of \(\rho^{\bullet }\)}:\\

It is easier to work with the inverse of the Beck-Chevalley cells instead, so we must show that
  \begin{equation*}
\frac{\textcolor{red}{L{\mathcal{B}^1}^{-1}}}{\textcolor{blue}{\rho^{\bullet}_{X \odot X'}}} = \frac{\textcolor{violet}{\rho^{\bullet}_X \ \mid \ \rho^{\bullet}_{X'}}}{\textcolor{olive}{{\mathcal{B}^2_F}^{-1}}},
  \end{equation*}
for any composable $\mathfrak{C}$-spans $X, X'$, where $\mathcal{B}^i$ refers to  Beck-Chevalley cells for the $\mathfrak{C}_i$-existential structure $P_i$. Thus, if we write 
\begin{equation*}
    \begin{tikzcd}
	{P_2FX_2} & {P_2FX'} \\
	{P_2FX} & {P_2F(X\odot X')}
	\arrow[""{name=0, anchor=center, inner sep=0}, "{P_2F(x_2')}", from=1-1, to=1-2]
	\arrow["{P_2Fx_2}"', from=1-1, to=2-1]
	\arrow["{P_2F\chi'}", from=1-2, to=2-2]
	\arrow[""{name=1, anchor=center, inner sep=0}, "{P_2F\chi}"', from=2-1, to=2-2]
	\arrow["\alpha"{description}, shift left, draw=none, from=0, to=1]
\end{tikzcd}
\end{equation*}
for the 2-cell whose mate is $(\mathcal{B}^2_{Fx_2, Fx_2'})^{-1}$, i.e., the image of the pullback of $x_2' \in \mathfrak{C}^r$ along $x_2 \in \mathfrak{C}^l$ under $P_2F$, then we have:\footnote{Recall that $F$ is a morphism of adequate triples.}

\begin{equation*}
    \frac{\textcolor{violet}{\rho^{\bullet}_X \ \mid \ \rho^{\bullet}_{X'}}}{\textcolor{teal}{{\mathcal{B}^2_F}^{-1}}} = \adjustbox{scale=0.7}{\begin{tikzpicture}[baseline=7em, marker/.style={circle, fill, inner sep=1pt, text=white}]
\node[box, box ports south=5, box ports north=5, minimum width=5cm, color=violet] (Box1) at (0,0) {\rho^{-1}_{x_1}};
\node[box, box ports south=2, box ports north=2, minimum width=2cm, color=violet] (Box2) at ($(Box1.north.4)+(0.5, 1.5)$) {\rho_{x_2}};
\node[box, box ports south=5, box ports north=5, minimum width=5cm, color=violet] (Box3) at ($(Box2.north.2)+(1, 1.5)$) {\rho^{-1}_{x_2'}};
\node[box, box ports south=2, box ports north=2, minimum width=2cm, color=violet] (Box4) at ($(Box3)+(1.5, 2)$) {\rho_{x_3'}};

\node[box, box ports south=1, box ports north=1, minimum width=2cm, color=teal] (Box5) at ($(Box1)+(5.5, 0)$) {P_2 F(\alpha)};
\node[box, box ports south=2, box ports north=1, minimum width=2cm, color=teal] (Box6) at ($(Box5.south)+(0, -1.5)$)   {\gamma^{P_2 F}_{x_2, \chi}};
\node[box, box ports south=1, box ports north=2, minimum width=2cm, color=teal] (Box7) at ($(Box5.north)+(0, 1.5)$) {{\gamma^{P_2 F}}^{-1}_{x_2', \chi'}};

\node[coordinate, label=below:{{\scriptsize $\rho_{X_1}$}}]                          (input1) at ($(Box1.south.1)+(0,-3)$) {};
\node[coordinate, label=below:{{\scriptsize $P_2 F x_1$}}]                           (input2) at ($(input1)+(4,0)$) {};
\node[coordinate, label={[text=teal]below:{{\scriptsize $P_2 F \chi$}}}]             (input3) at ($(input2)+(4,0)$) {};
\node[coordinate, label={[text=teal]below:{{\scriptsize $\exists^{P_2} F \chi'$}}}]  (input4) at ($(input3)+(1.5,0)$) {};
\node[coordinate, label=below:{{\scriptsize $\exists^{P_2} F x_3'$}}]                (input5) at ($(input4)+(1.5,0)$) {};

\node[coordinate, label={[text=violet]above:{{\scriptsize $L P_1 x_1$}}}]            (output1) at ($(Box1.north.1) + (0,7)$) {};
\node[coordinate, label={[text=violet]above:{{\scriptsize $L \exists^{P_1} x_2$}}}]  (output2) at ($(output1)+(1.5,0)$) {};
\node[coordinate, label={[text=violet]above:{{\scriptsize $L P_1 x_2'$}}}]           (output3) at ($(output2)+(1.5,0)$) {};
\node[coordinate, label={[text=violet]above:{{\scriptsize $L \exists^{P_1} x_3'$}}}] (output4) at ($(output3)+(1.5,0)$) {};
\node[coordinate, label={[text=violet]above:{{\scriptsize $\rho_{X_3'}$}}}]          (output5) at ($(output4)+(1.5,0)$) {};

\wires[violet]{Box1 = {south.1 = input1.north, south.5 = input2.north, north.1 = output1.south, north.5 = Box2.south.2}}{};
\node[coordinate] (cup1) at ($(Box2.south.1)+(-1.5,0)$) {};
\wires[violet, looseness=1.5]{Box2 ={south.1 = cup1.south}}{};
\wires[violet]{cup1 = {north = output2.south}}{};
\wires[violet]{Box3 = {south.1 = Box2.north.1, north.1 = output3.south, north.5 = Box4.south.2}};
\node[coordinate] (cup2) at ($(Box4.south.1)+(-1.5,0)$) {};
\wires[violet, looseness=1.5]{Box4 ={south.1 = cup2.south}}{};
\wires[violet]{cup2 = {north = output4.south}}{};
\wires[violet]{Box4 = {north.1 = output5.south}};
\node[coordinate] (cup3) at ($(Box4.north.2)+(4,0)$) {};
\wires[violet, looseness=1.5]{cup3 ={north = Box4.north.2, south = input5.north}}{};
\node[coordinate] (cup4) at ($(Box2.north.2)+(1.5,0)$) {};
\node[coordinate] (line1) at ($(cup4.south)+(0,-2.5)$) {};
\wires[violet, looseness=1.5]{cup4 ={north = Box2.north.2, south = line1.north}}{};

\wires[teal]{Box6 = {south.2 = input3.north, north.1 = Box5.south.1}};
\node[coordinate] (cup5) at ($(Box6.south.1)+(-1.5,0)$) {};
\wires[teal, looseness=2]{cup5 ={south = Box6.south.1, north = line1.south}}{};
\wires[teal]{Box7 = {north.1 = Box3.south.5, south.1 = Box5.north.1}};
\node[coordinate] (cup6) at ($(Box7.north.2)+(1.5,0)$) {};
\wires[teal, looseness=2]{cup6 ={north = Box7.north.2, south = input4.north}}{};

\path (Box1.north.5) -- (Box2.south.2) node[pos=0.5, right] [text=violet]{\scriptsize $\rho_{X_1}$};
\path (Box2.north.1) -- (Box3.south.1) node[pos=0.5, right] [text=violet]{\scriptsize $\rho_{X_2}$};
\path (Box7.north.1) -- (Box3.south.5) node[pos=0.5, left] [text=teal]{\scriptsize $P_2 Fx_2'$};
\path (Box5.north.1) -- (Box7.south.1) node[pos=0.5, left] [text=teal]{\scriptsize $P_2 F(x_2'\chi')$};
\path (Box6.north.1) -- (Box5.south.1) node[pos=0.5, left] [text=teal]{\scriptsize $P_2 F(x_2\chi)$};
\end{tikzpicture}}
\end{equation*}

\begin{equation*}
    = \adjustbox{scale=0.6}{\begin{tikzpicture}[baseline=14em, marker/.style={circle, fill, inner sep=1pt, text=white}]
\node[box, box ports south=5, box ports north=5, minimum width=5cm, color=black] (Box1) at (0,0) {\rho^{-1}_{x_1}};
\node[box, box ports south=2, box ports north=2, minimum width=2cm, color=black] (Box2) at ($(Box1.north.4)+(0.5, 1.5)$) {\rho_{x_2}};

\node[box, box ports south=2, box ports north=1, minimum width=2cm, color=black] (Box6) at ($(Box2.north)+(1, 1.5)$)   {\gamma^{P_2 F}_{x_2, \chi}};
\node[box, box ports south=1, box ports north=1, minimum width=2cm, color=black] (Box5) at ($(Box6.north)+(0, 1.5)$) {P_2 F(\alpha)};
\node[box, box ports south=1, box ports north=2, minimum width=2cm, color=black] (Box7) at ($(Box5.north)+(0, 1.5)$) {{\gamma^{P_2 F}}^{-1}_{x_2', \chi'}};

\node[box, box ports south=5, box ports north=5, minimum width=5cm, color=black] (Box3) at ($(Box7.north)+(0.5, 1.5)$) {\rho^{-1}_{x_2'}};
\node[box, box ports south=2, box ports north=2, minimum width=2cm, color=black] (Box4) at ($(Box3.north)+(1.5, 1.5)$) {\rho_{x_3'}};

\node[coordinate, label=below:{{\scriptsize $\rho_{X_1}$}}]                          (input1) at ($(Box1.south.1)+(0,-1)$) {};
\node[coordinate, label=below:{{\scriptsize $P_2 F x_1$}}]                           (input2) at ($(input1)+(4,0)$) {};
\node[coordinate, label={[text=black]below:{{\scriptsize $P_2 F \chi$}}}]             (input3) at ($(input2)+(1,0)$) {};
\node[coordinate, label={[text=black]below:{{\scriptsize $\exists^{P_2} F \chi'$}}}]  (input4) at ($(input3)+(1.5,0)$) {};
\node[coordinate, label=below:{{\scriptsize $\exists^{P_2} F x_3'$}}]                (input5) at ($(input4)+(2,0)$) {};

\node[coordinate, label={[text=black]above:{{\scriptsize $L P_1 x_1$}}}]            (output1) at ($(Box1.north.1) + (0,12.5)$) {};
\node[coordinate, label={[text=black]above:{{\scriptsize $L \exists^{P_1} x_2$}}}]  (output2) at ($(output1)+(1.5,0)$) {};
\node[coordinate, label={[text=black]above:{{\scriptsize $L P_1 x_2'$}}}]           (output3) at ($(output2)+(1.5,0)$) {};
\node[coordinate, label={[text=black]above:{{\scriptsize $L \exists^{P_1} x_3'$}}}] (output4) at ($(output3)+(1.5,0)$) {};
\node[coordinate, label={[text=black]above:{{\scriptsize $\rho_{X_3'}$}}}]          (output5) at ($(output4)+(1.5,0)$) {};

\wires[black]{Box1 = {south.1 = input1.north, south.5 = input2.north, north.1 = output1.south, north.5 = Box2.south.2}}{};
\node[coordinate] (cup1) at ($(Box2.south.1)+(-1.5,0)$) {};
\wires[black, looseness=1.5]{Box2 ={south.1 = cup1.south}}{};
\wires[black]{cup1 = {north = output2.south}}{};
\wires[black]{Box3 = {south.1 = Box2.north.1, north.1 = output3.south, north.5 = Box4.south.2}};
\node[coordinate] (cup2) at ($(Box4.south.1)+(-1.5,0)$) {};
\wires[black, looseness=1.5]{Box4 ={south.1 = cup2.south}}{};
\wires[black]{cup2 = {north = output4.south}}{};
\wires[black]{Box4 = {north.1 = output5.south}};
\node[coordinate] (cup3) at ($(Box4.north.2)+(1.5,0)$) {};
\wires[black, looseness=2]{cup3 ={north = Box4.north.2, south = input5.north}}{};
\node[coordinate] (line1) at ($(Box2.north.2)+(0,0.6)$) {};
\wires[black]{Box2 = {north.2 = line1.south}};

\wires[black]{Box6 = {south.1 = line1.north, south.2 = input3.north, north.1 = Box5.south.1}};
\wires[black]{Box7 = {north.1 = Box3.south.2, south.1 = Box5.north.1}};
\node[coordinate] (cup6) at ($(Box7.north.2)+(1.5,0)$) {};
\wires[black, looseness=1.5]{cup6 ={north = Box7.north.2, south = input4.north}}{};

\path (Box1.north.5) -- (Box2.south.2) node[pos=0.5, left] [text=black]{\scriptsize $\rho_{X_1}$};
\path (Box2.north.1) -- (Box3.south.1) node[pos=0.5, left] [text=black]{\scriptsize $\rho_{X_2}$};
\path (Box7.north.1) -- (Box3.south.2) node[pos=0.5, right] [text=black]{\scriptsize $P_2 Fx_2'$};
\path (Box5.north.1) -- (Box7.south.1) node[pos=0.5, right] [text=black]{\scriptsize $P_2 F(x_2'\chi')$};
\path (Box6.north.1) -- (Box5.south.1) node[pos=0.5, right] [text=black]{\scriptsize $P_2 F(x_2\chi)$};
\end{tikzpicture}}
\end{equation*}

On the other hand, using the triangle identities and that  $\rho$, $L\eta^{P_1(x_3' \chi')}$ and $\epsilon^{P_2}_F$ preserve compositors, we have
\begin{equation*}
  \frac{\textcolor{red}{L{\mathcal{B}^1}^{-1}}}{\textcolor{blue}{\rho^{\bullet}_{X \odot X'}}}   =\adjustbox{scale=0.6}{\begin{tikzpicture}[baseline=3em, marker/.style={circle, fill, inner sep=1pt, text=white}]
\node[box, box ports south=5, box ports north=5, minimum width=5cm, color=red] (Box1) at (0,0) {{\gamma^{ L P_1}}^{-1}_{x_1, \chi}};
\node[box, box ports south=1, box ports north=2, minimum width=2cm, color=red] (Box2) at ($(Box1)+(4, 0)$) {{\gamma^{L \exists^{P_1}}}^{-1}_{\chi',x_3'}};
\node[box, box ports south=2, box ports north=1, minimum width=2cm, color=red] (Box3) at ($(Box1)+(1.5, 2)$) {\gamma^{L P_1}_{x_2, \chi}};
\node[box, box ports south=1, box ports north=1, minimum width=2cm, color=red] (Box4) at ($(Box3)+(0, 2)$) {LP_1(\alpha)};
\node[box, box ports south=1, box ports north=2, minimum width=2cm, color=red] (Box5) at ($(Box4)+(0, 2)$) {{\gamma^{L P_1}}^{-1}_{x_2', \chi'}};

\node[box, box ports south=6, box ports north=6, minimum width=6cm, color=blue] (Box6) at ($(Box1)+(4.5, -4)$) {\rho^{-1}_{x_1 \chi}};
\node[box, box ports south=2, box ports north=2, minimum width=2cm, color=blue] (Box7) at ($(Box6)+(2, 2.5)$) {\rho_{x_3' \chi'}};

\node[box, box ports south=2, box ports north=1, minimum width=3cm, color=blue] (Box8) at ($(Box6)+(1.5, -2)$) {\gamma^{P_2 F}_{x_1, \chi}};
\node[box, box ports south=2, box ports north=1, minimum width=3cm, color=blue] (Box9) at ($(Box8)+(3.5, 0)$) {\gamma^{\exists^{P_2} F}_{x_3', \chi'}};

\node[coordinate, label={[text=blue]below:{{\scriptsize $\rho_{X_1}$}}}]                          (input1) at ($(Box6.south.1)+(0,-3)$) {};
\node[coordinate, label={[text=blue]below:{{\scriptsize $P_2 F x_1$}}}]              (input2) at ($(Box8.south.1)+(0,-1)$) {};
\node[coordinate, label={[text=blue]below:{{\scriptsize $P_2 F \chi$}}}]             (input3) at ($(input2)+(1.5,0)$) {};
\node[coordinate, label={[text=blue]below:{{\scriptsize $\exists^{P_2} F \chi'$}}}]  (input4) at ($(input3)+(2,0)$) {};
\node[coordinate, label={[text=blue]below:{{\scriptsize $\exists^{P_2} F x_3'$}}}]   (input5) at ($(input4)+(1.5,0)$) {};

\node[coordinate, label={[text=red]above:{{\scriptsize $L P_1 x_1$}}}]            (output1) at ($(Box1.north.1) + (0,7)$) {};
\node[coordinate, label={[text=red]above:{{\scriptsize $L \exists^{P_1} x_2$}}}]  (output2) at ($(output1)+(1.5,0)$) {};
\node[coordinate, label={[text=red]above:{{\scriptsize $L P_1 x_2'$}}}]           (output3) at ($(output2)+(1.5,0)$) {};
\node[coordinate, label={[text=red]above:{{\scriptsize $L \exists^{P_1} x_3'$}}}] (output4) at ($(output3)+(3.5,0)$) {};
\node[coordinate, label={[text=red]above:{{\scriptsize $\rho_{X_3'}$}}}]          (output5) at ($(output4)+(1.5,0)$) {};

\node[coordinate] (cup1) at ($(Box5.north)+(2,0)$) {};
\node[coordinate] (cup2) at ($(Box3.south)+(-2,0)$) {};
\node[coordinate] (cup3) at ($(Box7.north)+(3,0.5)$) {};
\node[coordinate] (cup4) at ($(Box7.south)+(-2.5,-0.5)$) {};
\node[coordinate] (cup5) at ($(Box7.south.1)+(0,-0.5)$) {};
\node[coordinate] (cup6) at ($(Box7.north.2)+(0, 0.5)$) {};
\node[coordinate] (line1) at ($(output5)+(0,-8)$) {};
\node[coordinate] (line2) at ($(Box6.south)+(5,0)$) {};

\wires[red]{Box1 = {north.1 = output1.south, north.5 = Box3.south.2}}{};
\wires[red]{Box2 = {north.1 = cup1.south, north.2 = output4.south}}{};
\wires[red, looseness=1.5]{Box3 = {south.1 = cup2.south, north.1 = Box4.south}}{};
\wires[red, looseness=1.5]{Box5 = {south.1 = Box4.north, north.1 = output3.south, north.2 = cup1.north}}{};
\wires[red]{output2 = {south = cup2.north}};
\wires[red]{output5 = {south = line1.north}};

\wires[blue]{Box6 = {north.1 = Box1.south.5, north.6 = Box7.south.2, south.1 = input1.north, south.5 = Box8.north}}{};
\wires[blue]{Box7 = {north.1 = line1.south, north.2 = cup6.south, south.1 = cup5.north, south.2 = Box6.north.6}}{};
\wires[blue]{Box2 = {south = cup4.north}}{};
\wires[blue]{Box9 = {north = cup3.south}}{};
\wires[blue, looseness=1.5]{cup4 = {south = cup5.south}}{};
\wires[blue, looseness=2]{cup3 = {north = cup6.north}}{};

\wires[blue]{Box8 = {south.1 = input2.north, south.2 = input3.north}}{};
\wires[blue]{Box9 = {south.1 = input4.north, south.2 = input5.north}}{};

\path (Box3.north) -- (Box4.south) node[pos=0.5, left] [text=red]{\scriptsize $LP_1(x_2 \chi)$};
\path (Box4.north) -- (Box5.south) node[pos=0.5, left] [text=red]{\scriptsize $LP_1(x_2' \chi')$};
\path (Box1.north.5) -- (Box3.south.2) node[pos=0.5, right] [text=red]{\scriptsize $LP_1\chi$};
\path (cup1.south) -- (Box2.north.1) node[pos=0.5, left] [text=red]{\scriptsize $L\exists^{P_1}\chi'$};
\path (cup1.south) -- (Box2.north.1) node[pos=0.5, left] [text=red]{\scriptsize $L\exists^{P_1}\chi'$};

\path (Box1.south.5) -- (Box6.north.1) node[pos=0.5, left] [text=blue]{\scriptsize $LP_1(x_1 \chi)$};
\path (Box2.south) -- (cup4.north) node[pos=0.5, left] [text=blue]{\scriptsize $L\exists^{P_1}(x_3' \chi')$};
\path (Box7.south.1) -- (cup5.north) node[pos=0.5, left] [text=blue]{\scriptsize $LP_1(x_3' \chi')$};
\path (Box7.south.2) -- (Box6.north.6) node[pos=0.5, right] [text=blue]{\scriptsize $\rho_{X\odot X'}$};
\path (Box7.north.2) -- (cup6.south) node[pos=0.5, right] [text=blue]{\scriptsize $P_2F(x_3'\chi')$};
\path (Box8.north) -- (Box6.south.5) node[pos=0.5, left] [text=blue]{\scriptsize $P_2F(x_1 \chi)$};
\path (Box9.north) -- (line2.south) node[pos=0.5, left] [text=blue]{\scriptsize $\exists^{P_2}F(x_3'\chi')$};
\end{tikzpicture}}
\end{equation*}

\begin{equation*}
    = \adjustbox{scale=0.6}{\begin{tikzpicture}[baseline=7em, marker/.style={circle, fill, inner sep=1pt, text=white}]
\node[box, box ports south=5, box ports north=5, minimum width=5cm, color=black] (Box1) at (0,0) {{\gamma^{ L P_1}}^{-1}_{x_1, \chi}};
\node[box, box ports south=2, box ports north=1, minimum width=2cm, color=black] (Box3) at ($(Box1)+(1.5, 2)$) {\gamma^{L P_1}_{x_2, \chi}};
\node[box, box ports south=1, box ports north=1, minimum width=2cm, color=black] (Box4) at ($(Box3)+(0, 2)$) {LP_1(\alpha)};
\node[box, box ports south=1, box ports north=2, minimum width=2cm, color=black] (Box5) at ($(Box4)+(0, 2)$) {{\gamma^{L P_1}}^{-1}_{x_2', \chi'}};

\node[box, box ports south=2, box ports north=2, minimum width=2cm, color=black] (Box2) at ($(Box1)+(7.5, 4)$) {{\gamma^{P_2 F}_{x_3', \chi'}}^{-1}};
\node[box, box ports south=2, box ports north=2, minimum width=2cm, color=black]  (Box6) at ($(Box2)+(-1, -2)$) {\rho_{x_3' \chi'}};
\node[box, box ports south=2, box ports north=2, minimum width=2cm, color=black] (Box7) at ($(Box6)+(-1, -2)$) {\gamma^{L P_2}_{x_3', \chi'}};

\node[box, box ports south=3, box ports north=3, minimum width=3cm, color=black] (Box8) at ($(Box1)+(3, -2.5)$) {\gamma^{L P_1}_{x_1, \chi}};
\node[box, box ports south=4, box ports north=4, minimum width=4cm, color=black] (Box9) at ($(Box8)+(2.5, -2)$) {\rho^{-1}_\chi};
\node[box, box ports south=3, box ports north=3, minimum width=3cm, color=black] (Box10) at ($(Box8)+(0, -4)$) {\rho^{-1}_{x_1}};

\node[coordinate, label=below:{{\scriptsize $\rho_{X_1}$}}]                          (input1) at ($(Box10.south.1)+(0,-1)$) {};
\node[coordinate, label={[text=black]below:{{\scriptsize $P_2 F^{\op} x_1$}}}]             (input2) at ($(Box10.south.3)+(0,-1)$) {};
\node[coordinate, label={[text=black]below:{{\scriptsize $P_2 F^{\op} \chi$}}}]            (input3) at ($(input2)+(3,0)$) {};
\node[coordinate, label={[text=black]below:{{\scriptsize $\exists^{P_2} F \chi'$}}}]  (input4) at ($(input3)+(2,0)$) {};
\node[coordinate, label={[text=black]below:{{\scriptsize $\exists^{P_2} F x_3'$}}}]   (input5) at ($(input4)+(1.2,0)$) {};

\node[coordinate, label={[text=black]above:{{\scriptsize $L P_1 x_1$}}}]            (output1) at ($(Box1.north.1) + (0,7)$) {};
\node[coordinate, label={[text=black]above:{{\scriptsize $L \exists^{P_1} x_2$}}}]  (output2) at ($(output1)+(1.5,0)$) {};
\node[coordinate, label={[text=black]above:{{\scriptsize $L P_1 x_2'$}}}]           (output3) at ($(output2)+(1.5,0)$) {};
\node[coordinate, label={[text=black]above:{{\scriptsize $L \exists^{P_1} x_3'$}}}] (output4) at ($(output3)+(3,0)$) {};
\node[coordinate, label={[text=black]above:{{\scriptsize $\rho_{X_3'}$}}}]          (output5) at ($(output4)+(2,0)$) {};

\node[coordinate] (cup1) at ($(Box5.north)+(2,0)$) {};
\node[coordinate] (cup2) at ($(Box3.south)+(-2,0)$) {};
\node[coordinate] (cup3) at ($(Box2.north)+(2.7,0)$) {};
\node[coordinate] (cup4) at ($(Box2.north)+(1.5,0)$) {};
\node[coordinate] (cup5) at ($(Box7.south)+(-2,0)$) {};
\node[coordinate] (cup6) at ($(Box7.south)+(-1.5,0)$) {};

\node[coordinate] (line1) at ($(cup1)+(0,-5)$) {};
\node[coordinate] (line2) at ($(line1)+(0.5,0)$) {};

\wires[black]{Box1 = {north.1 = output1.south, north.5 = Box3.south.2}}{};
\wires[black, looseness=1.5]{Box3 = {south.1 = cup2.south, north.1 = Box4.south}}{};
\wires[black, looseness=1.5]{Box5 = {south.1 = Box4.north, north.1 = output3.south, north.2 = cup1.north}}{};
\wires[black]{output2 = {south = cup2.north}};
\wires[black]{cup1 = {south = line1.north}};
\wires[black]{output4 = {south = line2.north}};
\wires[black]{output5 = {south = Box6.north.1}};

\wires[black, looseness=1.5]{Box2 = {north.1 = cup3.north, north.2 = cup4.north, south.1 = Box6.north.2}}{};
\wires[black, looseness=1.5]{Box7 = {north.2 = Box6.south.1, south.1 = cup6.south, south.2 = cup5.south}}{};
\wires[black]{cup5 = {north = line1.south}};
\wires[black]{cup6 = {north = line2.south}};

\wires[black]{Box8 = {north.1 = Box1.south.5}};
\wires[black]{Box9 = {north.1 = Box8.south.3, north.4 = Box6.south.2, south.1 = Box10.north.3, south.4 = input3.north}};
\wires[black]{Box10 = {north.1 = Box8.south.1, north.3 = Box9.south.1, south.1 = input1.north, south.3 = input2.north}};
\wires[black]{cup3 = {south = input5.north}};
\wires[black]{cup4 = {south = input4.north}};

\path (Box3.north) -- (Box4.south) node[pos=0.5, left] [text=black]{\scriptsize $LP_1(x_2 \chi)$};
\path (Box4.north) -- (Box5.south) node[pos=0.5, left] [text=black]{\scriptsize $LP_1(x_2' \chi')$};
\path (Box1.north.5) -- (Box3.south.2) node[pos=0.5, right] [text=black]{\scriptsize $LP_1\chi$};
\path (cup1.south) -- (cup5.north) node[pos=0.5, left] [text=black]{\scriptsize $L\exists^{P_1}\chi'$};

\path (Box1.south.5) -- (Box8.north.1) node[pos=0.5, left] [text=black]{\scriptsize $L P_1(x_1,\chi)$};
\path (Box8.south.1) -- (Box10.north.1) node[pos=0.5, left] [text=black]{\scriptsize $L P_2(x_1)$};
\path (Box8.south.3) -- (Box9.north.1) node[pos=0.5, left] [text=black]{\scriptsize $L P_1(\chi)$};

\path (Box6.south.1) -- (Box7.north.2) node[pos=0.5, left] [text=black]{\scriptsize $LP_1(x_3'\chi')$};
\path (Box6.south.2) -- (Box9.north.4) node[pos=0.5, right] [text=black]{\scriptsize $\rho_{X\odot X'}$};
\path (Box2.south.1) -- (Box6.north.2) node[pos=0.5, right] [text=black]{\scriptsize $P_2F(x_3'\chi')$};
\end{tikzpicture}}
\end{equation*}

\begin{equation*}
   = \adjustbox{scale=0.5}{\begin{tikzpicture}[baseline=17em, marker/.style={circle, fill, inner sep=1pt, text=white}]
\node[box, box ports south=4, box ports north=4, minimum width=4cm, color=black] (Box1) at (0,0) {\rho^{-1}_{x_1}};
\node[box, box ports south=6, box ports north=6, minimum width=6cm, color=black] (Box2) at ($(Box1)+(4, 2)$) {\rho^{-1}_\chi};

\node[box, box ports south=3, box ports north=1, minimum width=3cm, color=black] (Box3) at ($(Box2)+(-1.5, 2)$) {\gamma^{L P_1}_{x_2, \chi}};
\node[box, box ports south=1, box ports north=1, minimum width=3cm, color=black] (Box4) at ($(Box3)+(0, 2)$) {LP_1(\alpha)};
\node[box, box ports south=5, box ports north=5, minimum width=5cm, color=black] (Box5) at ($(Box4)+(1, 2)$) {{\gamma^{L P_1}}^{-1}_{x_2', \chi'}};

\node[box, box ports south=2, box ports north=2, minimum width=2cm, color=black] (Box6) at ($(Box5)+(1.5, 2)$) {{\gamma^{LP_1}_{x_3', \chi'}}^{-1}};

\node[box, box ports south=2, box ports north=2, minimum width=2cm, color=black] (Box7) at ($(Box6)+(1, 2)$) {\rho_{x_3' \chi'}};

\node[box, box ports south=2, box ports north=2, minimum width=2cm, color=black] (Box8) at ($(Box7)+(1, 2)$) {{\gamma^{P_2F}}^{-1}_{x_3', \chi'}};

\node[coordinate, label=below:{{\scriptsize $\rho_{X_1}$}}]                           (input1) at ($(Box1.south.1)+(0,-1)$) {};
\node[coordinate, label={[text=black]below:{{\scriptsize $P_2 F^{\op} x_1$}}}]        (input2) at ($(Box1.south.4)+(0,-1)$) {};
\node[coordinate, label={[text=black]below:{{\scriptsize $P_2 F^{\op} \chi$}}}]       (input3) at ($(input2)+(5,0)$) {};
\node[coordinate, label={[text=black]below:{{\scriptsize $\exists^{P_2} F \chi'$}}}]  (input4) at ($(input3)+(2.5,0)$) {};
\node[coordinate, label={[text=black]below:{{\scriptsize $\exists^{P_2} F x_3'$}}}]   (input5) at ($(input4)+(1.2,0)$) {};

\node[coordinate, label={[text=black]above:{{\scriptsize $L P_1 x_1$}}}]            (output1) at ($(Box1.north.1) + (0,15.5)$) {};
\node[coordinate, label={[text=black]above:{{\scriptsize $L \exists^{P_1} x_2$}}}]  (output2) at ($(output1)+(1.5,0)$) {};
\node[coordinate, label={[text=black]above:{{\scriptsize $L P_1 x_2'$}}}]           (output3) at ($(output2)+(1.5,0)$) {};
\node[coordinate, label={[text=black]above:{{\scriptsize $L \exists^{P_1} x_3'$}}}] (output4) at ($(output3)+(1.5,0)$) {};
\node[coordinate, label={[text=black]above:{{\scriptsize $\rho_{X_3'}$}}}]          (output5) at ($(output4)+(2.5,0)$) {};

\node[coordinate] (cup1) at ($(Box3.south.1)+(-1.5,0)$) {};
\node[coordinate] (cup2) at ($(Box6.south.1)+(-1.5,0)$) {};
\node[coordinate] (cup3) at ($(Box8.north.2)+(1.5,0)$) {};
\node[coordinate] (cup4) at ($(Box8.north.2)+(2.7,0)$) {};

\wires[black]{Box1 = {south.1 = input1.north, south.4 = input2.north, north.1 = output1.south, north.4 = Box2.south.1}}{};
\wires[black]{Box2 = {south.6 = input3.north, north.2 = Box3.south.2, north.6 = Box7.south.2}}{};
\wires[black, looseness=1.5]{Box3 = {south.1 = cup1.south, north.1 = Box4.south}}{};
\wires[black]{cup1 = {north = output2.south}}{};
\wires[black]{Box5 = {south.2 = Box4.north, north.1 = output3.south, north.5 = Box6.south.2}}{};
\wires[black, looseness=1.5]{Box6 = {south.1 = cup2.south, north.2 = Box7.south.1}}{};
\wires[black]{cup2 = {north = output4.south}}{};
\wires[black]{Box7 = {north.1 = output5.south, north.2 = Box8.south.1}}{};
\wires[black, looseness=1.5]{Box8 = {north.1 = cup4.north, north.2 = cup3.north}}{};
\wires[black]{cup3 = {south = input4.north}}{};
\wires[black]{cup4 = {south = input5.north}}{};

\path (Box1.north.4) -- (Box2.south.1) node[pos=0.5, left] [text=black]{\scriptsize $\rho_X$};
\path (Box2.north.2) -- (Box3.south.2) node[pos=0.5, right] [text=black]{\scriptsize $LP_1(\chi)$};
\path (Box3.north) -- (Box4.south) node[pos=0.5, right] [text=black]{\scriptsize $LP_1(x_2 \chi)$};
\path (Box4.north) -- (Box5.south.2) node[pos=0.5, right] [text=black]{\scriptsize $LP_1(x_2' \chi')$};
\path (Box2.north.6) -- (Box7.south.2) node[pos=0.5, right] [text=black]{\scriptsize $\rho_{X\odot X'}$};
\path (Box6.north.2) -- (Box7.south.1) node[pos=0.5, left] [text=black]{\scriptsize $LP_1(x_3'\chi')$};
\path (Box7.north.2) -- (Box8.south.1) node[pos=0.5, right] [text=black]{\scriptsize $P_2F(x_3'\chi')$};
\end{tikzpicture}}
\end{equation*}

\begin{equation*}
   = \adjustbox{scale=0.5}{\begin{tikzpicture}[baseline=17em, marker/.style={circle, fill, inner sep=1pt, text=white}]
\node[box, box ports south=4, box ports north=4, minimum width=4cm, color=black] (Box1) at (0,0) {\rho^{-1}_{x_1}};
\node[box, box ports south=6, box ports north=6, minimum width=6cm, color=black] (Box2) at ($(Box1)+(4, 2)$) {\rho^{-1}_\chi};

\node[box, box ports south=3, box ports north=1, minimum width=3cm, color=black] (Box3) at ($(Box2)+(-1.5, 2)$) {\gamma^{L P_1}_{x_2, \chi}};
\node[box, box ports south=1, box ports north=1, minimum width=3cm, color=black] (Box4) at ($(Box3)+(0, 2)$) {LP_1(\alpha)};
\node[box, box ports south=5, box ports north=5, minimum width=5cm, color=black] (Box5) at ($(Box4)+(1, 2)$) {{\gamma^{L P_1}}^{-1}_{x_2', \chi'}};
\node[box, box ports south=2, box ports north=2, minimum width=2cm, color=black] (Box6) at ($(Box5)+(2.5, 1.5)$) {\rho_{\chi'}};

\node[box, box ports south=2, box ports north=2, minimum width=2cm, color=black] (Box7) at ($(Box6)+(-1, 1.5)$) {\rho_{x_3'}};
\node[box, box ports south=2, box ports north=1, minimum width=2cm, color=black] (Box8) at ($(Box7)+(1, 1.5)$) {{\gamma^{P_2 F}_{x_3', \chi'}}^{-1}};
\node[box, box ports south=1, box ports north=2, minimum width=2cm, color=black] (Box9) at ($(Box8)+(0, 2)$) {{\gamma^{P_2F}}^{-1}_{x_3', \chi'}};

\node[coordinate, label=below:{{\scriptsize $\rho_{X_1}$}}]                           (input1) at ($(Box1.south.1)+(0,-1)$) {};
\node[coordinate, label={[text=black]below:{{\scriptsize $P_2 F^{\op} x_1$}}}]        (input2) at ($(Box1.south.4)+(0,-1)$) {};
\node[coordinate, label={[text=black]below:{{\scriptsize $P_2 F^{\op} \chi$}}}]       (input3) at ($(input2)+(5,0)$) {};
\node[coordinate, label={[text=black]below:{{\scriptsize $\exists^{P_2} F \chi'$}}}]  (input4) at ($(input3)+(1.5,0)$) {};
\node[coordinate, label={[text=black]below:{{\scriptsize $\exists^{P_2} F x_3'$}}}]   (input5) at ($(input4)+(1.2,0)$) {};

\node[coordinate, label={[text=black]above:{{\scriptsize $L P_1 x_1$}}}]            (output1) at ($(Box1.north.1) + (0,15.5)$) {};
\node[coordinate, label={[text=black]above:{{\scriptsize $L \exists^{P_1} x_2$}}}]  (output2) at ($(output1)+(1.5,0)$) {};
\node[coordinate, label={[text=black]above:{{\scriptsize $L P_1 x_2'$}}}]           (output3) at ($(output2)+(1.5,0)$) {};
\node[coordinate, label={[text=black]above:{{\scriptsize $L \exists^{P_1} x_3'$}}}] (output4) at ($(output3)+(1.5,0)$) {};
\node[coordinate, label={[text=black]above:{{\scriptsize $\rho_{X_3'}$}}}]          (output5) at ($(output4)+(1.5,0)$) {};

\node[coordinate] (cup1) at ($(Box3.south.1)+(-1.5,0)$) {};
\node[coordinate] (cup2) at ($(Box7.south.1)+(-1.5,0)$) {};
\node[coordinate] (cup3) at ($(Box9.north.2)+(1.5,0)$) {};
\node[coordinate] (cup4) at ($(Box9.north.2)+(2.7,0)$) {};
\node[coordinate] (line1) at ($(Box7.south.2)+(1,0)$) {};

\wires[black]{Box1 = {south.1 = input1.north, south.4 = input2.north, north.1 = output1.south, north.4 = Box2.south.1}}{};
\wires[black]{Box2 = {south.6 = input3.north}}{};
\wires[black, looseness=1.5]{Box3 = {south.1 = cup1.south, north.1 = Box4.south, south.2 = Box2.north.2}}{};
\wires[black]{cup1 = {north = output2.south}}{};
\wires[black]{Box5 = {south.2 = Box4.north, north.1 = output3.south, north.5 = Box6.south.1}}{};
\wires[black]{Box6 = {south.2 = Box2.north.6, north.1 = Box7.south.2, north.2 = line1.south}}{};
\wires[black, looseness=1.5]{Box7 = {south.1 = cup2.south, north.1 = output5.south, north.2 = Box8.south.1}}{};
\wires[black]{cup2 = {north = output4.south}}{};
\wires[black]{line1 = {north = Box8.south.2}}{};
\wires[black]{Box8 = {north.1 = Box9.south.1}}{};
\wires[black, looseness=1.5]{Box9 = {north.1 = cup4.north, north.2 = cup3.north}}{};
\wires[black]{cup3 = {south = input4.north}}{};
\wires[black]{cup4 = {south = input5.north}}{};

\path (Box1.north.4) -- (Box2.south.1) node[pos=0.5, left] [text=black]{\scriptsize $\rho_X$};
\path (Box2.north.2) -- (Box3.south.2) node[pos=0.5, right] [text=black]{\scriptsize $LP_1(x)$};
\path (Box3.north) -- (Box4.south) node[pos=0.5, right] [text=black]{\scriptsize $LP_1(x_2 \chi)$};
\path (Box4.north) -- (Box5.south.2) node[pos=0.5, right] [text=black]{\scriptsize $LP_1(x_2' \chi')$};
\path (Box2.north.6) -- (Box6.south.2) node[pos=0.5, left] [text=black]{\scriptsize $\rho_{X\odot X'}$};
\path (Box8.north.1) -- (Box9.south.1) node[pos=0.5, right] [text=black]{\scriptsize $P_2F(x_3'\chi')$};
\end{tikzpicture}}= \frac{\rho^{\bullet}_X \ \mid \ \rho^{\bullet}_{X'}}{{\mathcal{B}^2_F}^{-1}}
\end{equation*}

$\bullet$ \underline{Cell naturality of $\rho^{\bullet}$}: given a square $\begin{tikzcd}
	{Y_1} && {Y_2} \\
	\\
	{X_1} && {X_2}
	\arrow[""{name=0, anchor=center, inner sep=0}, "Y"{inner sep=.8ex}, "\shortmid"{marking}, from=1-1, to=1-3]
	\arrow["{f_1^{\op}}"', from=1-1, to=3-1]
	\arrow["{f_2^{\op}}", from=1-3, to=3-3]
	\arrow[""{name=1, anchor=center, inner sep=0}, "X"'{inner sep=.8ex}, "\shortmid"{marking}, from=3-1, to=3-3]
	\arrow["\alpha"', color={rgb,255:red,92;green,92;blue,214}, between={0.2}{0.8}, Rightarrow, from=1, to=0]
\end{tikzcd}$ in $\Span{\mathfrak{C}_1}^{\op}$, we must verify that

\begin{equation}
\label{cell naturality}
    \adjustbox{scale=0.6}{\begin{tikzpicture}[baseline=5em, marker/.style={circle, fill, inner sep=1pt, text=white}]
    
\node[coordinate, label=below:{{\scriptsize $LP_1f_1$}}]                          (input1) at ($(-1,-1)$) {};
\node[coordinate, label={[text=black]below:{{\scriptsize $\rho_{X_1}$}}}]             (input2) at ($(0,0)+(2,-1)$) {};
\node[coordinate, label={[text=black]below:{{\scriptsize $P_2^{\bullet}FX$}}}]            (input3) at ($(4,-1)$) {};

\node[coordinate, label={[text=black]above:{{\scriptsize $LP^{\bullet}_1Y$}}}]            (output1) at ($(Box1.north.1) + (1,5)$) {};
\node[coordinate, label={[text=black]above:{{\scriptsize $\rho_{Y_2}$}}}]  (output2) at ($(Box3.north.1)+(0.5, 3)$) {};
\node[coordinate, label={[text=black]above:{{\scriptsize $P_2Ff_2$}}}]           (output3) at ($(Box3.north.3)+(-1,3)$) {};


\node[box, box ports south=2, box ports north=2, minimum width = 5cm, color=black] (Box1) at (0,2) {{L{P_1^{\bullet}} \alpha}};

\node[box, box ports south=2, box ports north=2, minimum width = 3cm, color=black] (Box2) at (3,0) {\rho^{\bullet}_X};

\node[box, box ports south=2, box ports north=2, minimum width = 3cm , color=black] (Box3) at (3,4) {\rho_{f_2}};

\wires[black]{Box1 = {north.1=output1.south, south.1 = input1.north, north.2= Box3.south.1}}{};

\wires[black]{Box2 = {north.1=Box1.south.2, north.2= Box3.south.2, south.1 = input2.north, south.2 = input3.north}}{};

\wires[black]{Box3 = {north.1=output2.south, north.2= output3.south}}{};

\node[coordinate, label=left:{\scriptsize $LP_1f_2$}] (label1) at ($(Box1.north.2)+(0, 0.6)$) {};

\node[coordinate, label=below:{\scriptsize $LP_1^{\bullet}X$}] (label2) at ($(Box1.south.2)+(0.0,-0.5)$) {};

\node[coordinate, label=right:{\scriptsize $\rho_{X_2}$}] (label2) at ($(Box3.south.2)+(0.0,-1.6)$) {};

\end{tikzpicture}} =  \adjustbox{scale=0.7}{\begin{tikzpicture}[baseline=5em, marker/.style={circle, fill, inner sep=1pt, text=white}]



\node[box, box ports south=2, box ports north=2, minimum width = 5cm, color=blue] (Box1) at (6,2) {{P_2^{\bullet}F} \alpha};

\node[box, box ports south=2, box ports north=2, minimum width = 3cm, color=black] (Box2) at (3,0) {\rho_{f_1}};

\node[box, box ports south=2, box ports north=2, minimum width = 3cm , color=red] (Box3) at (3,4) {\rho^{\bullet}_Y};

\node[coordinate, label=below:{{\scriptsize $LP_1f_1$}}]                          (input1) at ($(Box2.south.1) + (0,-0.5)$) {};

\node[coordinate, label={[text=black]below:{{\scriptsize $\rho_{X_1}$}}}]             (input2) at ($(Box2.south.2) + (0,-0.5)$) {};

\node[coordinate, label={[text=black]below:{{\scriptsize $P_2^{\bullet}FX$}}}]            (input3) at ($(Box1.south.2)+(0, -2.5)$) {};

\node[coordinate, label={[text=black]above:{{\scriptsize $P_2Ff_2$}}}]            (output3) at ($(Box1.north.2) + (0,3)$) {};

\node[coordinate, label={[text=black]above:{{\scriptsize $\rho_{Y_2}$}}}]  (output2) at ($(Box3.north.2)+(0, 1)$) {};

\node[coordinate, label={[text=black]above:{{\scriptsize $LP^{\bullet}_1Y$}}}]           (output1) at ($(Box3.north.1)+(0,1)$) {};

\wires[black]{Box1 = {north.1=Box3.south.2, south.1 = Box2.north.2, north.2= output3.south, south.2=input3.north}}{};

\wires[black]{Box2 = {north.1=Box3.south.1, south.1 = input1.north, south.2 = input2.north}}{};

\wires[black]{Box3 = {north.1=output1.south, north.2= output2.south}}{};

\node[coordinate, label=left:{\scriptsize $P_2Ff_1$}] (label1) at ($(Box1.south.1)+(0.8, -0.6)$) {};

\node[coordinate, label=right:{\scriptsize $P_2^{\bullet}FY$}] (label2) at ($(Box3.south.2)+(0.5,-0.5)$) {};

\node[coordinate, label=left:{\scriptsize $\rho_{Y_1}$}] (label3) at ($(Box3.south.1)+(0.0,-1.6)$) {};
\end{tikzpicture}}
\end{equation}

Using the triangle identity for $\exists^{P_1} x_2 \dashv P_1 x_2$, the left-hand side of (\ref{cell naturality}) reduces to 

\begin{equation*}
\adjustbox{scale=0.6}{\begin{tikzpicture}[baseline=5em, marker/.style={circle, fill, inner sep=1pt, text=white}]

\node[box, box ports south=2, box ports north=2, minimum width = 3cm, color=black] (Box1) at (4,1) {\rho_{x_1}^{-1}};

\node[box, box ports south=2, box ports north=2, minimum width = 5cm, color=black] (Box2) at ($(Box1.north.1)+(-3,2)$) {LP_1(=_{f_1, x_1}^{y_1, \alpha})};

\node[box, box ports south=2, box ports north=2, minimum width = 3cm , color=black] (Box3) at ($(Box2.north.2) + (0,2)$) {LP_1(=_{y_2, \alpha}^{f_2, x_2})};

\node[box, box ports south=2, box ports north=2, minimum width = 3cm, color=black] (Box4) at ($(Box3.north.2) + (1,2)$) {\rho_{x_2}};

\node[box, box ports south=2, box ports north=2, minimum width = 3cm, color=black] (Box5) at ($(Box4.north.1) + (-1,2)$) {\rho_{f_2}};

\node[coordinate] (cup1) at ($(Box3.south.1)+(-1.5,0)$) {};

\node[coordinate] (cap1) at ($(Box4.north.2)+(4,0)$) {};

\node[coordinate, label=below:{{\scriptsize $LP_1f_1$}}]                          (input1) at ($(Box2.south.1) + (0,-3.5)$) {};

\node[coordinate, label=below:{{\scriptsize $\rho_{X_1}$}}]                          (input2) at ($(Box1.south.1) + (0,-1)$) {};

\node[coordinate, label=below:{{\scriptsize $P_2Fx_1$}}]                          (input3) at ($(Box1.south.2) + (0,-1)$) {};

\node[coordinate, label=below:{{\scriptsize $\exists^{P_2}Fx_2$}}]                          (input4) at ($(cap1.south) + (0,-9.5)$) {};

\node[coordinate, label=above:{{\scriptsize $LP_1 y_1$}}]                          (output1) at ($(Box2.north.1) + (-1,8.5)$) {};

\node[coordinate, label=above:{{\scriptsize $L\exists^{P_1}y_2$}}]                          (output2) at ($(cup1.north) + (0,7)$) {};

\node[coordinate, label=above:{{\scriptsize $\rho_{Y_2}$}}]                          (output3) at ($(Box5.north.1) + (0,1)$) {};

\node[coordinate, label=above:{{\scriptsize $P_2Ff_2$}}]                          (output4) at ($(Box5.north.2) + (0,1)$) {};

\wires[black, looseness=3]{cup1={north=output2.south, south= Box3.south.1} }{};

\wires[black, looseness=3]{cap1={north=Box4.north.2, south= input4.north} }{};

\wires[black]{Box1= {north.1 = Box2.south.2, north.2 = Box4.south.2, south.1=input2.north, south.2=input3.north}}{};

\wires[black]{Box2={north.1 = output1.south, north.2= Box3.south.2, south.1=input1.north}}{};

\wires[black]{Box4={south.1=Box3.north.2, north.1 = Box5.south.2}}{};

\wires[black]{Box5={south.1=Box3.north.1, north.1 = output3.south, north.2=output4.south}}{};

\node[coordinate, label=right:{\scriptsize $\rho_{X}$}] (label1) at ($(Box1.north.2)+(-0.2, 3)$) {};

\node[coordinate, label=left:{\scriptsize $LP_1x_1$}] (label2) at ($(Box1.north.1)+(-0.5, 0.5)$) {};

\node[coordinate, label=right:{\scriptsize $LP_1\alpha$}] (label3) at ($(Box3.south.2)+(0, -0.7)$) {};

\node[coordinate, label=right:{\scriptsize $LP_1x_2$}] (label4) at ($(Box4.south.1)+(-0.1, -0.7)$) {};

\node[coordinate, label=left:{\scriptsize $LP_1f_2$}] (label5) at ($(Box5.south.1)+(0, -2)$) {};

\node[coordinate, label=right:{\scriptsize $\rho_{x_2}$}] (label6) at ($(Box5.south.2)+(0.2, -0.5)$) {};

\end{tikzpicture}}
\end{equation*}

Using the naturality of $\rho$ with respect to the equality cell for $f_2 x_2=y_2 \alpha$ lets us cancel out $\rho_{f_2}$ at the top and introduce $\rho_{f_2 x_2} = \rho_{y_2 \alpha}$. We can similarly use naturality of $\rho^{-1}$ with respect to the equality cell for $f_1x_1 = y_1 \alpha$ to eliminate $\rho^{-1} x_1$ and introduce $\rho_{f_1}$ and $\rho^{-1}_{f_1, x_1}$. This lets us rewrite the left-handside of Equation~\ref{cell naturality} as

\begin{equation*}
\adjustbox{scale=0.8}{\begin{tikzpicture}[baseline=5em, marker/.style={circle, fill, inner sep=1pt, text=white}]

\node[box, box ports south=2, box ports north=2, minimum width = 3cm, color=black] (Box1) at (2,1) {\rho_{f_1}};

\node[box, box ports south=2, box ports north=1, minimum width = 5cm, color=black] (Box2) at ($(Box1.north.1)+(4,2)$) {\gamma^{P_2F}_{f_1, x_1}};

\node[box, box ports south=2, box ports north=2, minimum width = 6cm, color=black] (Box5) at ($(Box2.north.1) + (-2,2)$) {\rho^{-1}_{f_1x_1} \textcolor{orange}{= \rho^{-1}_{y_1 \alpha}}};

\node[box, box ports south=1, box ports north=2, minimum width = 6cm, color=black] (Box6) at ($(Box5.north.1) + (-2,2)$) {{\gamma^{LP_1}}^{-1}_{y_1, \alpha}};

\node[box, box ports south=2, box ports north=2, minimum width = 3cm , color=black] (Box3) at ($(Box6.north.2) + (0,4)$) {\rho_{y_2}};

\node[box, box ports south=1, box ports north=2, minimum width = 3cm, color=black] (Box4) at ($(Box3.north.2) + (1,4)$) {{\gamma^{P_2F}}^{-1}_{f_2, x_2}}; 

\node[box, box ports south=2, box ports north=2, minimum width = 3cm, color=black] (Box7) at ($(Box5.north.2) + (-1,4)$) {\rho_{\alpha}}; 

\node[box, box ports south=2, box ports north=1, minimum width = 3cm, color=black] (Box8) at ($(Box3.north.2) + (2,1)$) {{\gamma^{P_2F}}_{y_2, \alpha}}; 

\node[coordinate] (cup1) at ($(Box3.south.1)+(-1.5,0)$) {};

\node[coordinate] (cap1) at ($(Box4.north.2)+(7,0)$) {};

\node[coordinate, label=below:{{\scriptsize $LP_1f_1$}}]                          (input1) at ($(Box1.south.1) + (0,-1)$) {};

\node[coordinate, label=below:{{\scriptsize $\rho_{X_1}$}}]                          (input2) at ($(Box1.south.2) + (0,-1)$) {};

\node[coordinate, label=below:{{\scriptsize $P_2Fx_1$}}]                          (input3) at ($(Box2.south.2) + (0,-3.5)$) {};

\node[coordinate, label=below:{{\scriptsize $\exists^{P_2}Fx_2$}}]                          (input4) at ($(cap1.south) + (0,-18.5)$) {};

\node[coordinate, label=above:{{\scriptsize $LP_1 y_1$}}]                          (output1) at ($(Box6.north.1) + (-1,10.5)$) {};

\node[coordinate, label=above:{{\scriptsize $L\exists^{P_1}y_2$}}]                          (output2) at ($(cup1.north) + (0,7)$) {};

\node[coordinate, label=above:{{\scriptsize $\rho_{Y_2}$}}]                          (output3) at ($(Box3.north.1) + (0,6.5)$) {};

\node[coordinate, label=above:{{\scriptsize $P_2Ff_2$}}]                          (output4) at ($(Box4.north.1) + (0,2)$) {};

\wires[black, looseness=3]{cup1={north=output2.south, south= Box3.south.1} }{};

\wires[black, looseness=2]{cap1={north=Box4.north.2, south= input4.north} }{};

\wires[black]{Box1= {south.1=input1.north, south.2=input2.north, north.2=Box2.south.1, north.1=Box5.south.1}}{};

\wires[black]{Box2= {south.2=input3.north, north.1= Box5.south.2, }}{};

\wires[black]{Box5= {north.1=Box6.south.1, north.2=Box7.south.2}}{};

\wires[black]{Box6= {north.1=output1.south, north.2=Box7.south.1}}{};

\wires[black]{Box7= {north.1=Box3.south.2, north.2=Box8.south.2}}{};

\wires[black]{Box8= {south.1=Box3.north.2, north.1= Box4.south.1}}{};

\wires[black]{Box3= {north.1=output3.south}}{};

\wires[black]{Box4= {north.1=output4.south}}{};


\node[coordinate, label=right:{\scriptsize $P2Ff_1$}] (label1) at ($(Box1.north.2) + (0.5, 0.5)$){};

\node[coordinate, label=right:{\scriptsize $P2F(f_1x_1)$}] (label2) at ($(Box2.north.1) + (0, 0.5)$){};

\node[coordinate, label=right:{\scriptsize $\rho_X$}] (label2) at ($(Box5.north.2) + (0, 2)$){};

\node[coordinate, label=left:{\scriptsize $LP_1\alpha$}] (label3) at ($(Box6.north.2) + (0, 0.5)$){};

\node[coordinate, label=right:{\scriptsize $\rho_Y$}] (label4) at ($(Box3.south.2) + (0, -1)$){};

\node[coordinate, label=right:{\scriptsize $P_2F\alpha$}] (label5) at ($(Box8.south.2) + (0, -1.5)$){};

\node[coordinate, label=right:{\scriptsize $P_2F(f_2x_2)=P_2F(y_2\alpha)$}] (label5) at ($(Box8.north.1) + (-0.5, 1)$){};

\node[coordinate, label=left:{\scriptsize $\rho_{Y_1}$}] (label6) at ($(Box1.north.1) + (0, 1.5)$){};

\node[coordinate, label=left:{\scriptsize $LP_1(y_1\alpha)$}] (label7) at ($(Box6.south.1) + (1, -1)$){};


\end{tikzpicture}}
\end{equation*}

\begin{equation*}
   = \adjustbox{scale=0.7}{\begin{tikzpicture}[baseline=15em, marker/.style={circle, fill, inner sep=1pt, text=white}]
\node[box, box ports south=6, box ports north=6, minimum width=6cm, color=black] (Box1) at (0,0) {\rho_{f_1}};
\node[box, box ports south=3, box ports north=3, minimum width=3cm, color=black] (Box2) at ($(Box1)+(3.5, 2)$) {\gamma_{f_1,\chi_1}^{P_2F}};
\node[box, box ports south=3, box ports north=3, minimum width=3cm, color=orange] (Box3) at ($(Box2)+(0, 2)$) {\gamma_{y_1,\alpha}^{P_2F^{-1}}};
\node[box, box ports south=6, box ports north=6, minimum width=6cm, color=orange] (Box4) at ($(Box1)+(0, 6)$) {\rho_{y_1^{-1}}};
\node[box, box ports south=3, box ports north=3, minimum width=3cm, color=orange] (Box5) at ($(Box3)+(0, 4)$) {\rho_{\alpha^{-1}}};
\node[box, box ports south=3, box ports north=3, minimum width=3cm, color=black] (Box6) at ($(Box5)+(0, 2)$) {\rho_\alpha};
\node[box, box ports south=3, box ports north=3, minimum width=3cm, color=black] (Box7) at ($(Box6)+(-2, 2)$) {\rho_{y_2}};

\node[coordinate, label=below:{{\scriptsize $LP_1f_1$}}]                           (input1) at ($(Box1.south.1)+(0,-1)$) {};
\node[coordinate, label={[text=black]below:{{\scriptsize $\rho_{X_1}$}}}]        (input2) at ($(Box1.south.6)+(0,-1)$) {};
\node[coordinate, label={[text=black]below:{{\scriptsize $P_2Fx_1$}}}]       (input3) at ($(input2)+(2,0)$) {};
\node[coordinate, label={[text=black]below:{{\scriptsize $\exists^{P_2}Fx_2$}}}]       (input4) at ($(input3)+(1.5,0)$) {};

\node[coordinate, label={[text=black]above:{{\scriptsize $L P_1 y_1$}}}]            (output1) at ($(Box1.north.1) + (0,13)$) {};
\node[coordinate, label={[text=black]above:{{\scriptsize $L \exists^{P_1} y_2$}}}]  (output2) at ($(output1)+(1.5,0)$) {};
\node[coordinate, label={[text=black]above:{{\scriptsize $\rho_{Y_2}$}}}]           (output3) at ($(output2)+(1.5,0)$) {};
\node[coordinate, label={[text=black]above:{{\scriptsize $P_2Ff_2$}}}] (output4) at ($(output3)+(2,0)$) {};

\node[coordinate] (line1) at ($(Box1.north.1)+(0,3)$) {};
\node[coordinate] (line2) at ($(Box4.north.1)+(0,1.5)$) {};
\node[coordinate] (line3) at ($(Box6.north.3)+(0,1.25)$) {};
\node[coordinate] (cup1) at ($(Box7.south.1)+(-1.5,0)$) {};
\node[coordinate] (cup2) at ($(line3)+(1.5,0)$) {};
\wires[black]{Box1 = {south.1 = input1.north, south.6 = input2.north, north.1 = line1.south, north.6 = Box2.south.1}}{};
\wires[black]{Box2 = {south.3 = input3.north, north.2 = Box3.south.2}}{};
\wires[orange]{Box3 = {north.1 = Box4.south.6, north.3 = Box5.south.3}}{};
\wires[orange]{Box4 = {south.1 = line1.north, north.1 = line2.south, north.6 = Box5.south.1}}{};
\wires[black]{Box6 = {south.1 = Box5.north.1, south.3 = Box5.north.3, north.1 = Box7.south.3, north.3 = line3.south}}{};
\wires[black, looseness=1.75]{Box7 = {south.1 = cup1.south, north.1 = output3.south, north.3 = output4.south}}{};
\wires[black]{line2 = {north = output1.south}}{};
\wires[black]{cup1 = {north = output2.south}}{};
\wires[black, looseness=1.75]{cup2 = {north = line3.north, south = input4.north}}{};

\path (line1.north) -- (Box4.south.1) node[pos=0.5, left] [text=orange]{\scriptsize $\rho_{Y_1}$};
\path (Box1.north.6) -- (Box2.south.1) node[pos=0.5, left] [text=black]{\scriptsize $P_2Ff_1$};
\path (Box2.north.2) -- (Box3.south.2) node[pos=0.5, left] [text=black]{\scriptsize $P_2F(f_1x_1) = \textcolor{orange}{P_2F(y_1\alpha)}$};

\path (Box3.north.1) -- (Box4.south.6) node[pos=0.5, left] [text=orange]{\scriptsize $P_2Fy_1$};
\path (Box3.north.3) -- (Box5.south.3) node[pos=0.5, right] [text=orange]{\scriptsize $P_2F\alpha$};
\path (Box4.north.6) -- (Box5.south.1) node[pos=0.5, left] [text=orange]{\scriptsize $\rho_Y$};
\path (Box5.north.1) -- (Box6.south.1) node[pos=0.5, left] [text=black]{\scriptsize $LP_1\alpha$};
\path (Box5.north.3) -- (Box6.south.3) node[pos=0.5, right] [text=black]{\scriptsize $\rho_X$};
\path (Box6.north.1) -- (Box7.south.3) node[pos=0.5, left] [text=black]{\scriptsize $\rho_Y$};
\path (Box6.north.3) -- (line3.south) node[pos=0.5, right] [text=black]{\scriptsize $P_2F\alpha$};
\end{tikzpicture}}
\end{equation*}

\begin{equation*}
    =\adjustbox{scale=0.7}{\begin{tikzpicture}[baseline=12em, marker/.style={circle, fill, inner sep=1pt, text=white}]
\node[box, box ports south=6, box ports north=6, minimum width=6cm, color=black] (Box1) at (0,0) {\rho_{f_1}};
\node[box, box ports south=3, box ports north=3, minimum width=3cm, color=blue] (Box2) at ($(Box1)+(3.5, 2)$) {P_2F{\cong}_{f_1,\chi_1}^{y_1, \alpha}};
\node[box, box ports south=6, box ports north=6, minimum width=6cm, color=red] (Box3) at ($(Box1)+(0, 4)$) {\rho_{y_1^{-1}}};
\node[box, box ports south=3, box ports north=3, minimum width=3cm, color=red] (Box4) at ($(Box3)+(1.5, 2)$) {\rho_{y_2}};
\node[box, box ports south=3, box ports north=3, minimum width=3cm, color=blue] (Box5) at ($(Box4)+(2, 2)$) {P_2F{\cong}_{y_2, \alpha}^{f_2,\chi_2}};

\node[coordinate, label=below:{{\scriptsize $LP_1f_1$}}]                           (input1) at ($(Box1.south.1)+(0,-1)$) {};
\node[coordinate, label={[text=black]below:{{\scriptsize $\rho_{X_1}$}}}]        (input2) at ($(Box1.south.6)+(0,-1)$) {};
\node[coordinate, label={[text=blue]below:{{\scriptsize $P_2Fx_1$}}}]       (input3) at ($(input2)+(2,0)$) {};
\node[coordinate, label={[text=blue]below:{{\scriptsize $\exists^{P_2}Fx_2$}}}]       (input4) at ($(input3)+(1.5,0)$) {};

\node[coordinate, label={[text=red]above:{{\scriptsize $L P_1 y_1$}}}]            (output1) at ($(Box1.north.1) + (0,9)$) {};
\node[coordinate, label={[text=red]above:{{\scriptsize $L \exists^{P_1} y_2$}}}]  (output2) at ($(output1)+(1.5,0)$) {};
\node[coordinate, label={[text=red]above:{{\scriptsize $\rho_{Y_2}$}}}]           (output3) at ($(output2)+(1.5,0)$) {};
\node[coordinate, label={[text=blue]above:{{\scriptsize $P_2Ff_2$}}}] (output4) at ($(output3)+(2,0)$) {};

\node[coordinate] (cup1) at ($(Box4.south.1)+(-1.5,0)$) {};
\node[coordinate] (cup2) at ($(Box5.north.3)+(1.5,0)$) {};

\wires[black]{Box1 = {south.1 = input1.north, south.6 = input2.north, north.1 = Box3.south.1, north.6 = Box2.south.1}}{};
\wires[blue]{Box2 = {south.3 = input3.north, north.1 = Box3.south.6, north.3 = Box5.south.3}}{};
\wires[red]{Box3 = {north.1 = output1.south, north.6 = Box4.south.3}}{};
\wires[red, looseness=1.75]{Box4 = {south.1 = cup1.south, north.1 = output3.south}}{};
\wires[red]{cup1 = {north = output2.south}}{};
\wires[blue, looseness=1.75]{Box5 = {south.1 = Box4.north.3, north.1 = output4.south, north.3 = cup2.north}}{};
\wires[blue]{cup2 = {south = input4.north}}{};

\path (Box1.north.1) -- (Box3.south.1) node[pos=0.5, left] [text=black]{\scriptsize $\rho_{Y_1}$};
\path (Box1.north.6) -- (Box2.south.1) node[pos=0.5, left] [text=black]{\scriptsize $P_2Ff_1$};
\path (Box2.north.1) -- (Box3.south.6) node[pos=0.5, left] [text=blue]{\scriptsize $P_2Fy_1$};
\path (Box2.north.3) -- (Box5.south.3) node[pos=0.5, right] [text=blue]{\scriptsize $P_2F\alpha$};
\path (Box3.north.6) -- (Box4.south.3) node[pos=0.5, right] [text=red]{\scriptsize $\rho_X$};
\path (Box4.north.3) -- (Box5.south.1) node[pos=0.5, right] [text=blue]{\scriptsize $P_2Fy_2$};
\end{tikzpicture}}.
\end{equation*}

This is the right-hand side of (\ref{cell naturality}), so we're done.
\end{proof}

The 2-functor $(-)^{\bullet}$ is also well-defined on 2-morphisms. If $(\alpha \colon F_2 \to F_1, \beta \colon L_1 \to L_2)$ is a 2-morphism $(F_1, \rho_1, L_1) \to (F_2, \rho_2, L_2)$ in $\ExistStr_{\mathrm{w}}$, then $(-)^{\bullet}$ maps it to a pair of tight transformations since $\Span{-} \colon \mathcal{A}{dq} \to \mathrm{Dbl}$ and $\Qt{-} \colon \TwoCat \to \mathrm{Dbl}$ are well-defined 2-functors. Moreover, $P_2^{\bullet}\ \Span{\alpha^{\op}} \circ \rho^{\bullet_1} = \rho_2^{\bullet} \circ \Qt{\beta}_{P_1^{\bullet}}$. It is immediate on objects and tight morphisms, and can be readily checked on spans $X_1 \xleftarrow[]{x_1} X \xrightarrow{x_2} X_2$ by expanding the definitions and using the analogous equation for $(\alpha, \beta)$  --- just note that the morphism of spans $\Span{\alpha}^{\op}_X$ is comprised of naturality squares $\alpha_{x_1}$ and $\alpha_{x_2}$ and that the naturality square $\beta_{\exists^{P_2} x_2}$ is the mate of $\beta_{P_2x_2}$, then use the functoriality of mates.\\

This settles that $(-)^{\bullet}$ as a whole is well-defined. We are now ready to prove Lemma~\ref{lemma:BulletConstruction}:

\begin{proof}[Proof of Lemma~\ref{lemma:BulletConstruction}]

We have established that $(-)^{\bullet}$ is well-defined, and it is clearly $2$-functorial on 2-cells, as $\Span{-}^{\op}$ and $\Qt{-}$ are 2-functors. For 1-cells, suppose given a composite \begin{equation*}
    \begin{tikzcd}
	{\mathsf{C}^{\op}_1} && {\mathcal{K}_1} \\
	\\
	{\mathsf{C}^{\op}_2} && {\mathcal{K}_2} \\
	\\
	{\mathsf{C}^{\op}_3} && {\mathcal{K}_3}
	\arrow[""{name=0, anchor=center, inner sep=0}, "{P_1}", from=1-1, to=1-3]
	\arrow["{F_1^{\op}}"', from=1-1, to=3-1]
	\arrow["{L_1}", from=1-3, to=3-3]
	\arrow[""{name=1, anchor=center, inner sep=0}, "{P_2}", from=3-1, to=3-3]
	\arrow["{F_2^{\op}}"', from=3-1, to=5-1]
	\arrow["{L_2}", from=3-3, to=5-3]
	\arrow[""{name=2, anchor=center, inner sep=0}, "{P_3}"', from=5-1, to=5-3]
	\arrow["\rho", color={rgb,255:red,92;green,92;blue,214}, between={0.2}{0.8}, Rightarrow, from=0, to=1]
	\arrow["\tau", color={rgb,255:red,92;green,92;blue,214}, between={0.2}{0.8}, Rightarrow, from=1, to=2]
\end{tikzcd}
\end{equation*}
in $\ExistStr_{\mathrm{w}}$. We have that $(\frac{\rho}{\tau})^{\bullet} = \frac{\rho^{\bullet}}{ \tau^{\bullet}}$ on objects and tight morphisms, as $\phi^{\bullet}_0 = \phi$ for any $\phi$. It is also true for spans: given $X \colon X_1 \nrightarrow X_2$, we have
\begin{align*}
    (\frac{\rho}{\tau})^{\bullet}_X = (\tau \rho)^{-1}_{x_1} \mid \mathrm{mate}((\tau \rho)_{x_2}) = \frac{L_2 \rho^{-1}x_1}{\tau^{-1}_{F_1x_1}} \mid \mathrm{mate}(L_2 \rho_{x_2} \circ \tau_{F_1 x_2})\\
    = \frac{L_2 \rho^{-1}x_1}{\tau^{-1}_{F_1x_1}} \mid \frac{\mathrm{mate}(L_2 \rho_{x_2})}{\mathrm{mate}(\tau_{F_1 x_2})} = \frac{L_2 \rho^{-1}x_1 \mid \mathrm{mate}(L_2 \rho_{x_2})}{\tau^{-1}_{F_1x_1} \mid \mathrm{mate}(\tau_{F_1 x_2})} = \frac{\rho^{\bullet}_X}{\tau^{\bullet}_X}
\end{align*}
by the functoriality of mates. Thus, $(-)^{\bullet}$ is 2-functorial. It is also  cartesian, since mates and compositions are preserved by products up to coherent isomorphisms.
\end{proof}

\end{document}